\documentclass{amsart}
\usepackage{chngcntr}
\usepackage{apptools}
\usepackage{color}
\usepackage{amsthm,amssymb,verbatim}
\usepackage{graphicx}
\usepackage{enumerate}

\usepackage{calrsfs}
\DeclareMathAlphabet{\pazocal}{OMS}{zplm}{m}{n}

\newcommand{\Aa}{\mathcal A} 
 
\newcommand{\Ff}{\mathcal F}
\newcommand{\Gg}{\mathcal{G}}

\newcommand{\Reg}{\operatorname{Reg}}
\newcommand{\Rr}{\mathcal{R}}
\newcommand{\Sing}{\operatorname{Sing}}

 \newcommand{\Cone}{\operatorname{Cone}}
 \newcommand{\Cyl}{\operatorname{Cyl}}

 \newcommand{\entropy}{\operatorname{entropy}}
 
 \newcommand{\genus}{\operatorname{genus}}
 \newcommand{\garea}{\operatorname{area}_{\,\gamma}}

 \newcommand{\length}{\operatorname{length}}

  \newcommand{\MM}{\mathcal{M}}
  
\newcommand{\Pp}{\mathcal{P}}

  \newcommand{\neck}{\operatorname{neck}}
  
  \newcommand{\Ss}{\mathcal{S}}

 \newcommand{\RR}{\mathbf{R}}  
 \newcommand{\Tt}{\mathcal{T}}
 
 \newcommand{\ZZ}{\mathbf{Z}}  
 \renewcommand{\SS}{\mathbf{S}}  

    \newcommand{\dist}{\operatorname{dist}}
 
 \newcommand{\area}{\operatorname{area}}
 \newcommand{\eps}{\epsilon}
 \newcommand{\tQ}{\widetilde Q}

\newcommand{\tG}{\widetilde G}
 
 \newcommand{\Tan}{\operatorname{Tan}}
 \newcommand{\Tpos}{T_\textnormal{pos}}

\newcommand{\ee}{\mathbf e}

\newcommand{\spt}{\operatorname{spt}}

\newcommand{\Hh}{\mathcal{H}}
\usepackage{amsthm}

\newcommand{\interior}{\operatorname{interior}}
\input epsf
\def\begfig {
\begin{figure}
\small }
\def\endfig {
\normalsize
\end{figure}
}

    \newtheorem{theorem}    {Theorem}   
    \newtheorem{lemma}      [theorem]       {Lemma}
    \newtheorem{corollary}  [theorem]     {Corollary}
    \newtheorem{proposition}       [theorem]       {Proposition}
    \newtheorem{claim}{Claim}
    \newtheorem*{theorem*}{Theorem}
    \newtheorem*{claim*}{Claim}
    \theoremstyle{definition}
    \newtheorem{definition}  [theorem] {Definition}
     \newtheorem{conjecture}  [theorem] {Conjecture}
    \theoremstyle{definition}
    \newtheorem{remark}   [theorem]       {Remark}
    \newtheorem*{remark*}{Remark}
    \newtheorem*{note*}{Note}
    
\renewcommand{\thesubsection}{\thetheorem}

\usepackage{hyperref}
\usepackage[alphabetic, msc-links, backrefs]{amsrefs}
 \renewcommand{\SS}{\mathbf{S}}  

\begin{document}

\renewcommand{\thesubsection}{\thetheorem}
\title[Generating Shrinkers]{Generating Shrinkers by Mean Curvature Flow}

\author[D. Hoffman]{\textsc{D. Hoffman}}

\address{David Hoffman\newline
 Department of Mathematics\newline
 Stanford University \newline
   Stanford, CA 94305, USA\newline
{\sl E-mail address:} {\bf dhoffman@stanford.edu}}

\author[F. Mart\'in]{\textsc{F. Mart\'in}}

\address{Francisco Mart\'in\newline
Departamento de Geometr\'ia y Topolog\'ia  \newline
Instituto de Matem\'aticas de Granada (IMAG) \newline
Universidad de Granada\newline
18071 Granada, Spain\newline
{\sl E-mail address:} {\bf fmartin@ugr.es}
}
\author[B. White]{\textsc{B. White}}

\address{Brian White\newline
Department of Mathematics \newline
 Stanford University \newline 
  Stanford, CA 94305, USA\newline
{\sl E-mail address:} {\bf bcwhite@stanford.edu}
}
\dedicatory{Dedicated to Tom Ilmanen (1961--2025)}

\begin{abstract}
We prove existence for many examples of shrinkers by
producing compact, smoothly embedded surfaces that, under mean curvature flow, develop singularities at which the 
 shrinkers occur as blowups.
\end{abstract}

\thanks{F. Martín's research is partially supported by the
framework of IMAG-Mar\'ia de Maeztu grant CEX2020-001105-M and by the grant PID2024-156031NB-I00 both funded by MCIU/AEI/ 10.13039/501100011033. Some of this work was carried out while the authors were visitors at the Simons Laufer Mathematical Sciences Institute (SLMath, formerly MSRI) in Berkeley, CA during the Fall 2024 semester, in a program supported by National Science Foundation Grant No. DMS-1928930.}


\newcommand{\diag}{\operatorname{diag}}

\newcommand{\Aac}{\mathcal{A}_{\rm con}}

\newcommand{\Aar}{\mathcal{A}_{\rm rect}}

\newcommand{\pdf}[2]{\frac{\partial #1}{\partial #2}}
\newcommand{\pdt}[1]{\frac{\partial #1}{\partial t}}

\newcommand{\mcd}{\operatorname{mcd}}
\newcommand{\mdr}{\operatorname{mdr}}
\newcommand{\tM}{\tilde M}
\newcommand{\tF}{\tilde F}

\newcommand{\tmu}{\tilde \mu}

\newcommand{\Gin}{\Gamma_\textnormal{in}}

\newcommand{\Tfat}{\operatorname{T_\textnormal{fat}}}
\newcommand{\Tdisc}{\operatorname{T_\textnormal{disc}}}
\newcommand{\tTfat}{\operatorname{\tilde T_\textnormal{fat}}}
\date{March 12, 2025. Revised January 20, 2026.}

\maketitle
\tableofcontents

\newcommand{\Mreg}{M_\textnormal{reg}}
\newcommand{\tMreg}{{\tilde M}_\textnormal{reg}}

\section{Introduction}

The formation of singularities in a mean curvature flow $t\mapsto M(t)$ with smooth, compact initial data is a fundamental phenomenon that involves the underlying geometry and topology of the evolving surfaces. As the flow progresses, regions of high curvature develop, leading to singularities where the surface either pinches off, collapses, or forms more complex structures. 
If a singularity occurs at a spacetime point $(p,T)$,  then, as $t\uparrow T$, suitable rescalings of $M(t)$ about $p$ converge (at least weakly and subsequentially) to 
a {\bf shrinker}, i.e., a surface $\Sigma$
such that
\[
   t\in (-\infty,0)\mapsto |t|^{1/2}\Sigma
\]
is a mean curvature flow.
Thus, to understand mean curvature flow singularities,
it is natural to try to understand all
possible shrinkers.  
Many interesting examples of shrinkers have
been constructed, and there are some notable theorems classifying certain kinds of shrinkers.
In this paper, we present a unified approach to producing many of the known examples, as well
as a new family of examples.

Almost all of the known examples of shrinkers have been
produced by one of the following three methods:
\begin{enumerate}
    \item Desingularization.  The idea is that
    one takes two embedded shrinkers that intersect
    along one or more curves, and one tries to replace intersecting ribbons of surface by a surface modelled on Scherks's singly periodic surface.
    See~\cites{kkm, nguyen}
    \item Doubling.  Here one takes a known embedded shrinker, and produces new examples that look like two slighly perturbed copies 
    of that shrinker joined by suitably placed catenoidal necks.  See~\cite{kapouleas-doubling}.
    \item Minimax.  
    See~\cites{ketover-platonic, ketover-fat, buzano}.
\end{enumerate}
Those methods are based on the fact that a surface is a shrinker if and only if it is minimal with
respect to a certain Riemannian metric (the so-called shrinker metric).
All three methods are powerful and have produced many interesting examples, but they have some 
limitations.  In particular, desingularization and doubling only produce examples when some parameter, typically the genus, is very large.
Also, all three methods leave open the question of whether the shrinker actually arises as a blowup
of the mean curvature flow of an initially smooth, compact surface.  
This last defect has been somewhat overcome by recent work of 
 Tang-Kai Lee and Xinrui Zhao\cite{lee-zhao}, who show that
any shrinker that satisfies a natural and rather mild hypothesis does indeed arise as such
 a blowup.
 (For $2$-dimensional shrinkers in $\RR^3$, the only case considered in this paper, the hypothesis is that no end of the shrinker is asymptotically cylindrical.)
In practice, it is typically easy to show that the hypothesis holds for the examples one produces provided the genus is sufficiently large. 
But for low genus, one generally does not know whether the hypothesis holds.

(Of course, if a shrinker is itself a compact, embedded surface, then it trivially arises as such blowup.
This is the case for the examples constructed
in~\cite{kapouleas-doubling}.)

In this paper, we produce shrinkers by creating  initial surfaces that, under mean curvature flow, develop singularities at which the desired shrinkers occur as blowups.  Since we are interested in shrinkers because of mean curvature flow, it is perhaps more natural to produce them by mean curvature flow than by other methods.
And, of course, we then know that the resulting shrinkers do arise
as blowups since that is how they are produced.
In particular, the Platonic shrinkers in  
Section~\ref{platonic-section} were first produced by Ketover~\cite{ketover-platonic}
using minimax, but it was not known until now that they arise as blowups.  Similarly, 
 the genus-$g$ examples in Section~\ref{desing-section} were produced for high genus by
 desingularization~\cites{kkm, nguyen}
 and
 for every genus by minimax~\cite{buzano}, but,
   except in the case of high genus, it was not known until now that they arise as blowups.
   (Indeed, until the recent 
   result of Lee-Zhao~\cite{lee-zhao} mentioned above, it 
   was not known whether any of the desingularized
   examples arose as blowups.)

 We now briefly describe the shrinkers
 produced in this paper.
 Section~\ref{platonic-section} produces shrinkers that resemble
 skeletons of the Platonic solids.  
 See Figure~\ref{fig:cube}.
 The proof we give also gives shrinkers with
 the symmetries of a regular $n$-gon.
 Section~\ref{desing-section} produces shrinkers that, for large genus,  look like  the union
 of a sphere and a bisecting plane, with the circle
 of intersection replaced by necks.
  See Figures~\ref{fig:one-ended-section}
 and~\ref{fig:one-ended}.
 Section~\ref{new-examples-section} produces 
 new examples that are 
 shrinker analogs
  of
some well-known, non-existent minimal surfaces.
Specifically, in the 1980s, researchers wondered
whether there existed complete embedded minimal
surfaces that were, roughly speaking, 
  Costa-Hoffman-Meeks surfaces
  with additional handles.
  In particular, for each $g$, the new surface
  (if it existed) would have the same symmetries
  as the genus-$g$ Costa-Hoffman-Meeks surface
  and the same number of ends ($3$),
   but it would have
  twice the genus, namely $2g$.
Mart\'in and Weber~\cite{martin-weber}
proved that such minimal surfaces do not exist.
By contrast,  we show that (for large $g$)
there are shrinkers that have the defining properties
of those non-existent minimal surfaces.
See Figure~\ref{fig:horgan}.

 We are aware of two other papers that use mean curvature flow to prove existence of 
shrinkers:
 \cite{chu-sun}
 proves existence of shrinkers of genus one with less entropy than the Angenent tori, 
 and~\cite{iw-fat} proves existence of certain genus-$g$ shrinkers that
  cause fattening when $g$ is large.
 On the other hand, Drugan and 
  Nguyen~\cite{drugan-nguyen} used a flow
  other than mean curvature flow to prove existence
 of rotationally invariant shrinking tori. Their examples presumably coincide with the Angenent Tori, whose existence was established in \cite{angenent}  by a shooting method.

\begin{figure}
  \begin{center}
  \includegraphics[height=2.3cm]{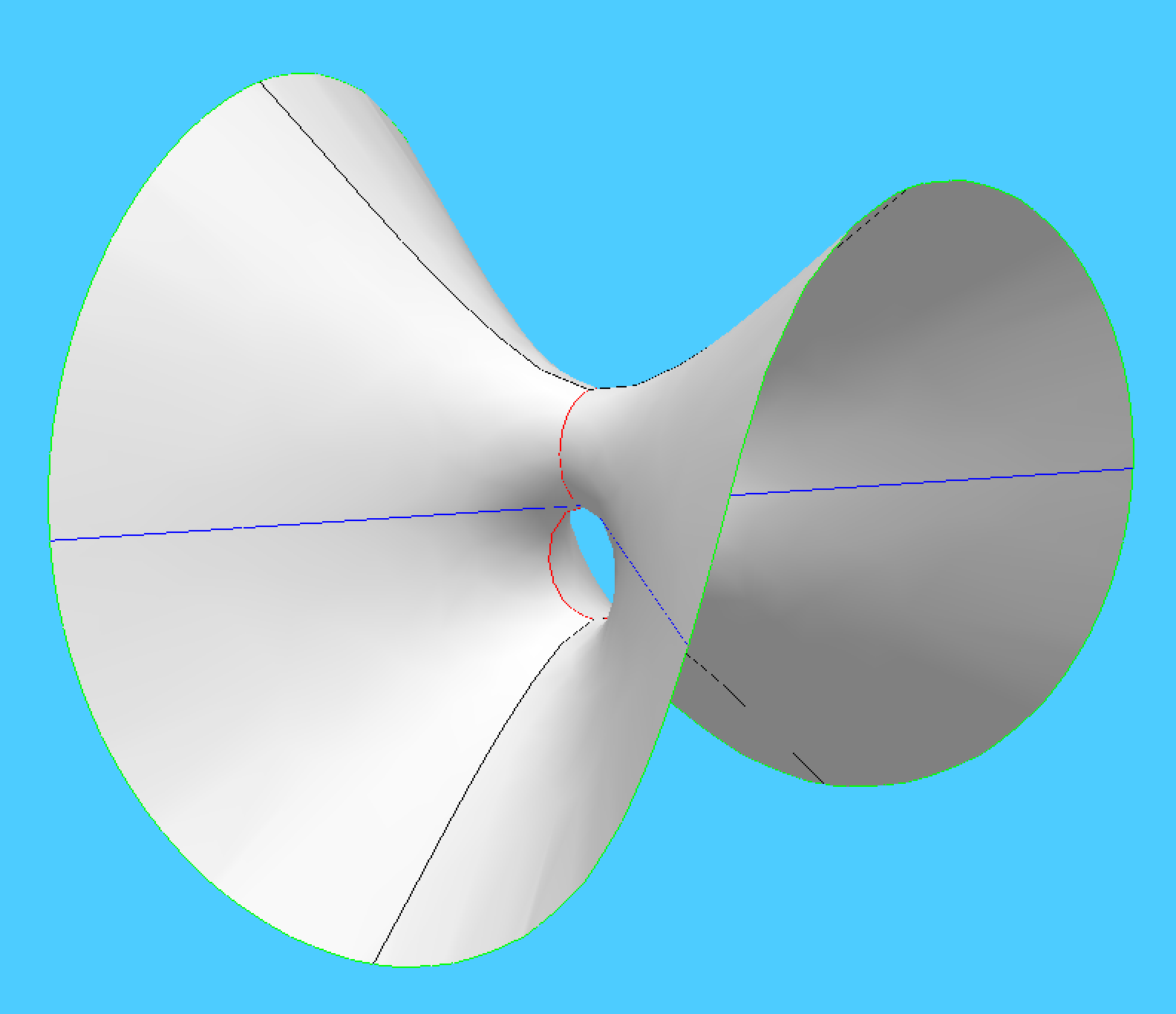} \includegraphics[height=2.3cm]{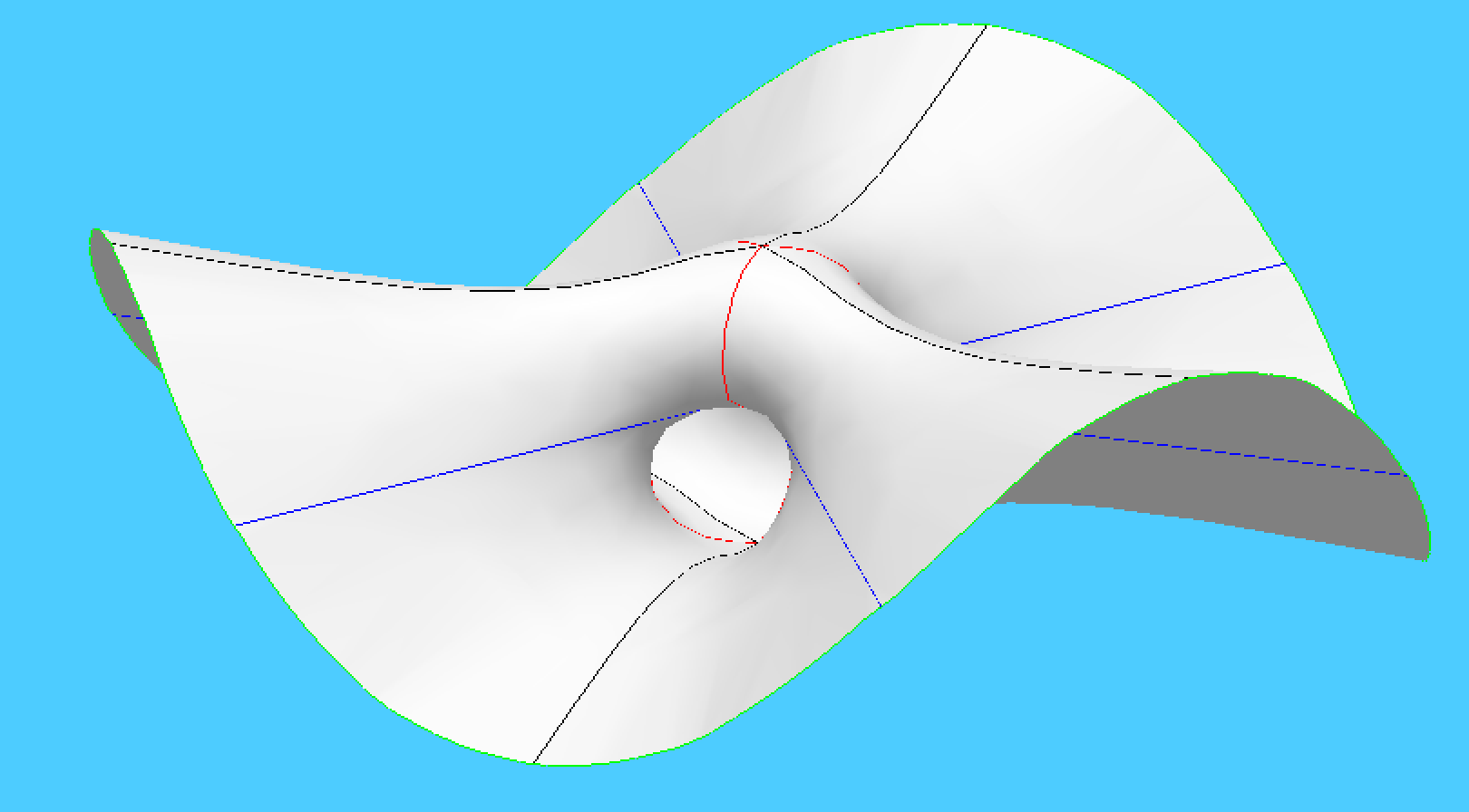} \includegraphics[height=2.3cm]{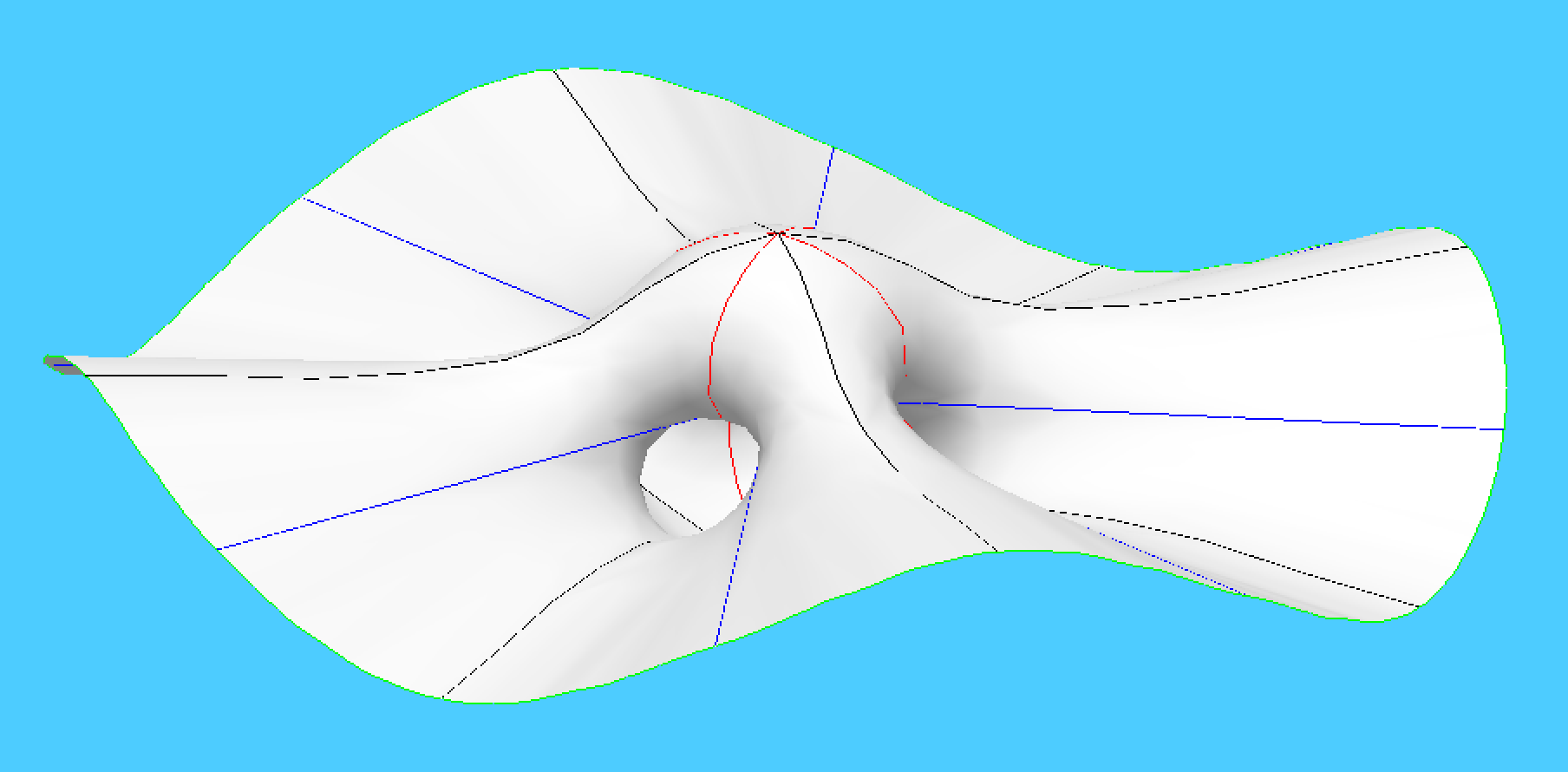} 
  \end{center}
  \begin{center}\includegraphics[height=1.8cm]{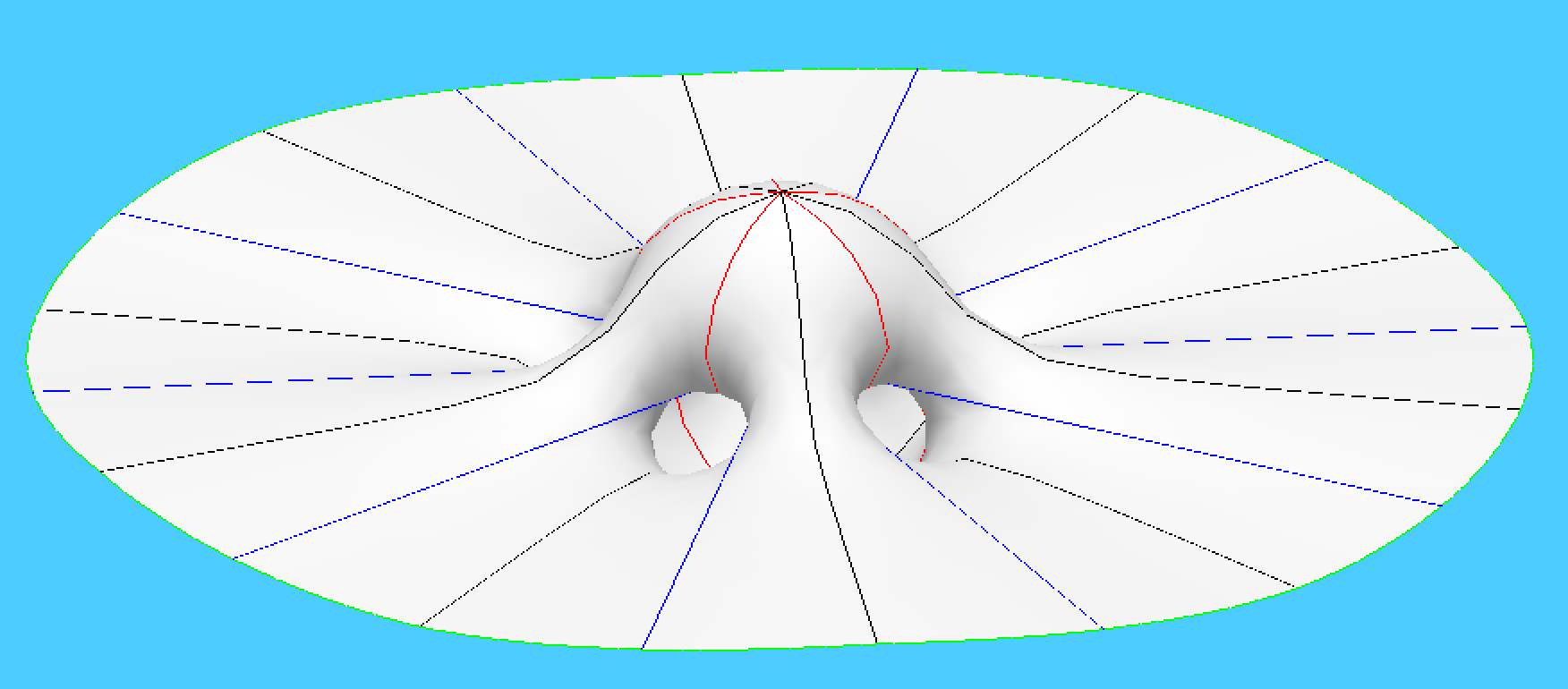} \includegraphics[height=1.8cm]{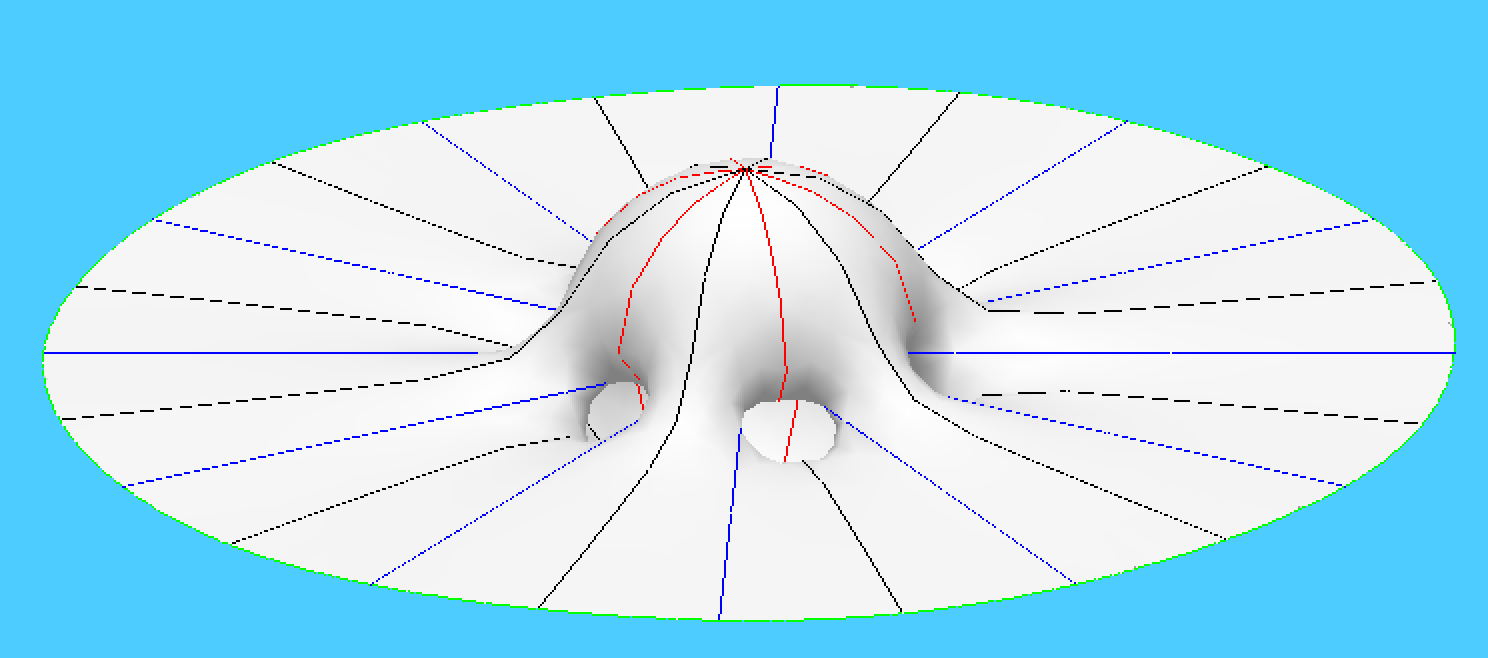} \includegraphics[height=1.8cm]{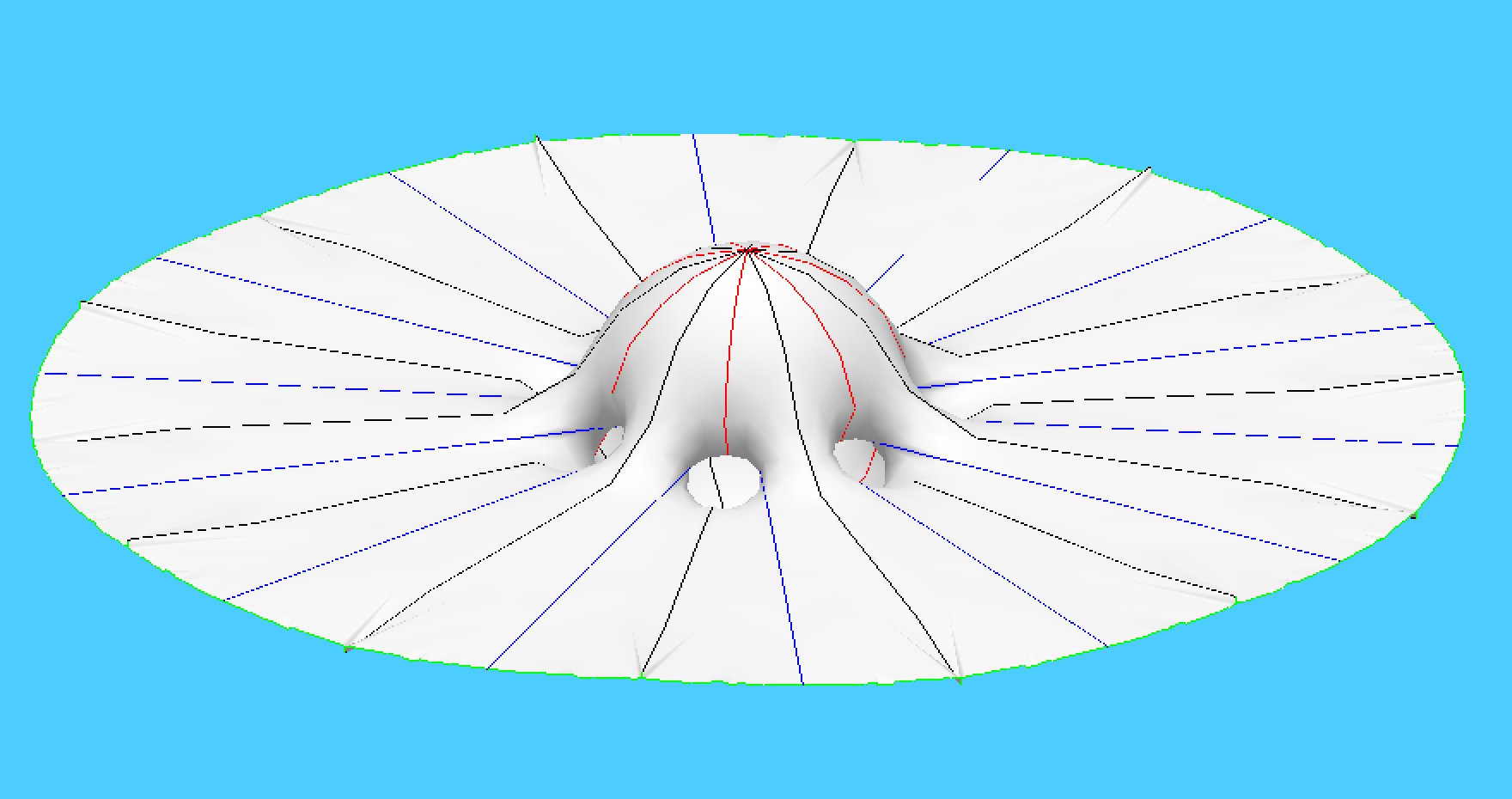}
    \end{center}

    \caption{A  
    numerical simulation of
     the shrinkers constructed in Section \ref{desing-section}, for genus $g=1, \ldots,6$. Figures courtesy of Francisco Torralbo (UGR).}
    \label{fig:one-ended-section}
\end{figure}

\section{Preliminaries}

We now recall the main
facts about mean curvature flow that we need.
To avoid pathologies,
we will assume throughout this section that the initial surface $M$ is compact,  smoothly embedded, and orientable, and that 
the ambient space is 
a complete, smooth, orientable  Riemannian  $3$-manifold  with Ricci curvature
bounded below.
The only ambient spaces needed  in this paper
are $\RR^3$ and $\SS^2\times\RR$, with the usual metrics.

If $K$ is a closed subset of $\RR^3$, or, more generally, of a smooth Riemannian $3$-manifold $N$, we let $K(t)$ denote the result of flowing $K$ for time $t$ by the level set flow.
Let $\Tfat(K)$ be the fattening time, i.e, the infimum of times $t$ such that $K(t)$ has nonempty interior.

If (as we are assuming in this section) $M$ is compact and smoothly embedded, then $\Tfat(M)$ is greater than or equal to the first time when classical mean curvature starting
from $M$ develops a singularity.
In particular, $\Tfat(M)>0$.

  By~\cite{bamler-kleiner},
$M(t)$ is a compact, smoothly embedded surface for almost 
every $t\in [0,\Tfat]$,
and the set of such times is relatively open in $[0,\Tfat]$.
Furthermore, 
\begin{equation}
\label{mu-flow}
  t\in [0,\Tfat]\mapsto 
     \mu(t):=\Hh^2\llcorner M(t)
\tag{$\star$}
\end{equation}
is a unit-regular Brakke flow.

(Roughly speaking,
a unit-regular Brakke flow is 
an integral Brakke flow that
is smooth (without sudden vanishing)
in a spacetime neighborhood of each point of Gauss density~$1$.
See~\cite{white-boundary}*{Definition~11.1}????
for the precise definition.
Weak limits of unit-regular flows are
also 
unit regular~\cite{white-local}.)

By slight abuse of terminology,
we will refer to 
\[
   t\in [0,\Tfat]\mapsto M(t)
\]
as a Brakke flow, when
we mean the Brakke flow~\eqref{mu-flow}.

By~\cite{bamler-kleiner}, if $0\le t_1<t_2\le \Tfat(M)$ are times at which $M(t)$ is smooth, then 
\begin{equation}\label{genus-decreases}
   \genus(M(t_2))\le \genus(M(t_1)).
\end{equation}

By~\cite{bamler-kleiner}, 
all tangent flows
at times $t\in (0,\Tfat]$ are given by smoothly embedded shrinkers of multiplicity $1$.  
Thus if $p\in N$ and $T\in (0,\Tfat(M)]$, then
every sequence $t_i\uparrow T$ has a subsequence such that
\begin{equation} \label{rescaled}
       (T-t_i)^{-1/2}( M(t_i) - p )
\end{equation}
converges smoothly and with multiplicity $1$ on compact subsets 
to a shrinker $\Sigma$.  
  If $\tau_i/t_i$ is bounded away from $0$ and away
  from $\infty$, then 
\[
     (T-\tau_i)^{-1/2} (M(\tau_i) - p) 
\]
also converges to the same shrinker $\Sigma$.
The shrinker $\Sigma$ is a plane if and only if $(p,T)$ is a regular point
of the flow.

(If we are working in a general ambient manifold~$N$,
then the expression~\eqref{rescaled} makes sense provided we regard $N$
as isometrically embedded in a Euclidean space.)

\begin{definition}
For $t\in (0,\Tfat(M)]$, we say $\Sigma$ is a {\bf shrinker} to $M(\cdot)$ at the spacetime
point $(p,t)$ if there exist a sequence $t_i\uparrow t$, such that
the surfaces \eqref{rescaled} converge smoothly to $\Sigma$.
\end{definition}

We expect there to be only one shrinker at each spacetime point, but that
is not known in all cases.  
(See~\cite{chodosh-schulze} for some cases in which uniqueness 
 is known.)
 However, all the shrinkers at $(p,T)$ have the same
genus (by Theorem~\ref{shrinker-genus} in
 the appendix), so the following definition makes sense.

\begin{definition}
The {\bf genus} of a spacetime point $(p,t)$ with $t\in (0,\Tfat]$
is the number $g$ such that each shrinker at $(p,t)$ has genus~$g$.
\end{definition}

At a  singular point of genus~$0$, 
 the shrinker is  either a sphere or a cylinder~\cite{brendle}. (Earlier, Kleene and M\o{}ller
proved the result assuming an axis of
rotational symmetry~\cite{kleene-moller}.)
Furthermore, the shrinker is
unique
by~\cite{colding-m}.
(See also~\cite{zhu}.)

The following result is trivial, but important.

\begin{theorem}\label{trivial-theorem}
If $(p,t)$ is a singular point with $t\le \Tfat$, and if there is an $\eps>0$ such 
that
\[
  \genus(M(\tau)\cap B(p,\eps)) \le g
\]
for every regular time $\tau \in (t-\eps,t)$, then $(p,t)$ has genus~$\le g$.
\end{theorem}

\begin{proof}
Let $t_i<t$ be regular times converging to $t$.
By passing to a subsequence, we can assume that
\[
   (t-t_i)^{-1/2} (M(t_i) - p)
\]
converges smoothly to a shrinker $\Sigma$.
Hence
\[
  (t-t_i)^{-1/2}(M(t_i)\cap B(p,\eps) - p)  \tag{*}
\]
also converges smoothly to $\Sigma$.
Since the surfaces~\thetag{*} have genus~$\le g$, the surface
 $\Sigma$ also has genus~$\le g$.
\end{proof}

\begin{corollary}\label{genera-corollary}
Suppose $T\le \Tfat$.  Then the sum of the genuses
of the singularities at time $T $ is less than or equal to the 
genus of $M$.
\end{corollary}

\begin{theorem}
\label{fat-genus-theorem}
If $\Tfat<\infty$, then there is shrinker at time $\Tfat$ with genus $\ge 1$.
\end{theorem}

\begin{proof}
By \cite{chh}*{Theorem~1.9}, there is a shrinker $\Sigma$ at time $\Tdisc\le \Tfat$ that is not a plane, a shrinking sphere, or a cylinder.
(See \cite{hw}*{Definition 2} or \cite{chh}*{p.~227} for the definition of the discrepancy time
$\Tdisc=\Tdisc(M)$.)
Since planes, spheres, and cylinders are the only smooth genus $0$ shrinkers, $\Sigma$ must have genus $\ge 1$.
Finally, by~\cite{bamler-kleiner}*{Theorem 1.8}, $\Tdisc=\Tfat$.
\end{proof}

\begin{definition}\label{Tpos} The time $\Tpos(M)$ is the infimum
of times $t\le \Tfat$ at which there is a singularity
of positive genus. \end{definition}
\noindent In some papers, $\Tpos(M)$ is written 
  $T_\textnormal{gen}(M)$. 
  
\begin{theorem}
If $\Tpos<\infty$, then 
there is a shrinker at time $\Tpos$
of genus $\ge 1$.
\end{theorem}

\begin{proof}
Note that, by definition, 
\begin{equation*} 
      \Tpos(M) \leq \Tfat(M).
  \end{equation*}
If $\Tpos=\Tfat$,
then there is a positive-genus
shrinker at time $\Tpos$ by 
  Theorem~\ref{fat-genus-theorem}.
If $\Tpos<\Tfat$, then there is a positive-genus
shrinker at time $\Tpos$ because the set of spacetime points
at which the genus is positive
is a closed subset of spacetime.
See, for example, \cite{white-genus}*{Theorem~25}.
\end{proof}

\begin{remark*}
In this paper, we only need to flow up to time $\Tpos$.
\end{remark*}

\begin{theorem}[Local Regularity Theorem]\label{local-theorem}\cite{white-local}
Suppose $U$ and $W$ are open subset of a smooth
Riemannian manifold and that $\overline{U}$ is  a compact
subset of $W$.  Suppose that
\[
  t\in (a,b]\mapsto M(t),
\]
is a mean curvature flow in which each $M(t)$ 
is a smooth, properly embedded submanifold of $W$.
Suppose for each $n$ that 
\[
  t\in (a,b]\mapsto M_n(t)
\]
a unit-regular Brakke flow in $W$
such that $M_n(t)$ converges weakly to $M(t)$
for almost every $t\in (a,b)$.
Let $a'\in (a,b)$.
Then, for large $n$, $M_n(t)\cap U$ is smooth 
and converges smoothly to $M(t)$
for all $t\in [a',b]$.
\end{theorem}

Here, the $M_n(\cdot)$ and $M(\cdot)$ can be Brakke
flows with respect to Riemannian metrics $\gamma_n$
and $\gamma$ on $W$, where $\gamma_n$ converges
smoothly to $\gamma$.

\begin{theorem}[Initial Regularity Theorem]
\label{initial-theorem}
Suppose $t\in [0,T)\mapsto M(t)$
is a mean curvature flow in an open set $U$
such that $M(0)$ is a smooth, properly embedded,
multiplicity~$1$ submanifold of $U$.
Let $K$ be a compact region in $U$ with smooth
boundary transverse to $M(0)$. 
There is 
an $\eps>0$ with the following properties.
\begin{enumerate}
\item\label{item:initial-1}
The flow
 $t\in [0,\eps]\mapsto M(t)\cap K$ 
 is smooth.
 \item\label{item:initial-2} if $t\in [0,T)\mapsto M_i(t)$
   are mean curvature flows that converge
   weakly to $M(\cdot)$, then for $0<\eps'<\eps$,
   \[
      t \in [\eps',\eps] \mapsto M_i(t)\cap K
    \]
    converges smoothly to
    \[
      t\in [\eps',\eps] \mapsto M(t)\cap K.
    \]
\item\label{item:initial-3} 
Suppose in Assertion~\eqref{item:initial-2} that $M_i(0)$ converges
   smoothly to $M(0)$.  Then
   \[
      t\in [0,\eps] \mapsto M_i(t)\cap K
   \]
   converges smoothly to
   \[
      t\in [0,\eps] \mapsto M(t)\cap K.
   \]
\end{enumerate}
\end{theorem}

\begin{proof}
For notational simplicity, we give the proof
in Euclidean space.
Let
\begin{align*}
    K'&=\{\dist(\cdot,K)\le \delta\}, 
\end{align*}
where $\delta>0$ is small enough that $K'$ is compact subset of $U$ and that $\partial K'$ 
 is
transverse to $M(0)$.

\begin{claim}\label{initial-claim-1}
Suppose $p_i$ converges to~$p\in M(0)$
and that $t_i>0$ converges to~$0$.
Let $q_i$ be the point in $M(0)$ closest to $p_i$.
Let 
\[
   r_i = \max \{|p_i-q_i|, t_i^{1/2}\}.
\]
Translate the flow $M(\cdot)$ in spacetime
by $(-q_i,0)$ and the dilate parabolically
by 
\[
  (x,t)\mapsto (x/r_i, t/r_i^2)
\]
to get a flow $M_i'(\cdot)$.  Then the 
flows $M_i'(\cdot)$ converge to the flow
\[
  t\in [0,\infty) \mapsto P,
\]
where $P$ is the plane $\Tan(M,p)$ with
multiplicity~$1$, and the convergence
is smooth on compact subsets of 
  $\RR^3\times (0,\infty)$.
\end{claim}

To prove the claim, note that, by passing to a 
 subsequence, we can assume that
the $M_i'(\cdot)$ converge weakly to 
a mean curvature flow 
\[
 t\in [0,\infty) \mapsto M'(t).
\]
Note that $M'(0)$ is $\Tan(M(0),p)$ with multiplicity~$1$.  It follows (by Huisken monotonicity) that $M'(t)\equiv M'(0)$ for
all $t\ge 0$.
By the Local Regularity Theorem
 (Theorem~\ref{local-theorem}),
the convergence of $M_i(\cdot)$ to $M'(\cdot)$
is smooth on
compact subsets of $\RR^3\times (0,\infty)$.
This completes the proof of 
 Claim~\ref{initial-claim-1}.

It follows from Claim~\ref{initial-claim-1} and 
 from~\cite{sw-local}*{Corollary~A.3}
 that there is an $\eps>0$ such that
\begin{equation}\label{smoothy}
\text{The flow
$t\in [0,2\eps] \mapsto M(t)\cap K'$ is smooth.}
\tag{$\dagger$}
\end{equation}
Assertion~\eqref{item:initial-1} follows immediately.
Assertion~\eqref{item:initial-2} follows from~\eqref{smoothy}
 and
 the Local Regularity
 Theorem (Theorem~\ref{local-theorem}).

Now let 
\[
\Omega
 = (K'\setminus \partial K')\times (-2\eps,2\eps).
\]

\begin{claim}\label{initial-claim-2}
    For $j\ge 2$ and $\alpha\in (0,1)$,
    there exists $C_{j,\alpha}<\infty$ such that
    \[
    K_{j,\alpha; \Omega}(M_n(\cdot)) \le C_{j,\alpha}
    \]
    for all sufficiently large $n$.
\end{claim}
See~\cite{white-local} for the definition of 
 $K_{2,\alpha; U(R)}(M(\cdot))$.
The proof of the claim is almost identical
to the proof
 of~\cite{white-local}*{Theorem~3.1}, so we omit it.
 Assertion~\eqref{item:initial-3} follows immediatetly from
 Claim~\ref{initial-claim-2}.
\end{proof}

\section{Oriented Flows}\label{orientation-section}

Suppose that $t\in I\mapsto M(t)$
is a unit-regular Brakke flow.
We let $\Reg(t)$ be the set of points $p$
such that the flow is smooth and with multiplicity~$1$
in a spacetime neighborhood of $(p,t)$.
Equivalently, $\Reg(t)$ is the set of points $p$
such that the Gauss Density at $(p,t)$ is $1$.
We let $\Sing(t)$ be the set of $p$ such that
the Gauss Density of at $(p,t)$ is $>1$.

\begin{definition}
Suppose $t\in I \mapsto M(t)$
is an integral Brakke flow of $m$-dimensional
surfaces in an oriented Riemannian $(m+1)$-manifold.  An {\bf orientation} of $M(\cdot)$
is a continuous function that assigns to 
$p\in \Reg(t)$ a unit normal $\nu(p,t)$
to $p$ at $M(t)$.

A {\bf strong orientation} is an orientation  $\nu$ 
with the following property.
For each time $t$, there is a closed region
$K(t)$ of $N$ such that $K(t)\setminus \Sing(t)$
is a smooth manifold-with-boundary, the boundary
being $\Reg(t)$, and such that $\nu(\cdot,t)$
is the unit normal that points out of $K(t)$.
\end{definition}

Suppose $M$ is a compact, connected embedded, oriented 
surface in $\RR^3$ or, more generally,
in a simply connected Riemannian $3$-manifold
with Ricci curvature bounded below.
Let $\nu$ be the unit normal to $M$ given
by the orientation.
Now $M$ divides the ambient space into
two regions.  We let $K$ be the closure
of  the region such that $\nu$ points out of $K$.
Then $\nu$ extends to an orientation
of the level set flow
\[
  t\in [0,\Tfat] \mapsto M(t)
\]
as follows.  As usual, we let $K(t)$ be
the result of flowing $K$ for time $t$ by level flow.  If $(p,t)$ is a regular spacetime
point of $M(\cdot)$, then there is a neighborhood
$B$ of $p$ such that $K(t)\cap B$
is a smooth manifold-with-boundary, the
boundary being $M(t)\cap B$.  We let 
  $\nu(p,t)$ be the unit normal to 
  $M(t)$ at $p$ that points out of $K(t)$.
Note that the orientation is strong.
We say that $\nu(\cdot)$ is the 
 {\bf standard orientation} induced
 by the given orientation of $M$.

\begin{theorem}
Suppose $M_i(\cdot)$ is a sequence
of unit-regular integral Brakke flows that converge
to a unit-regular Brakke flow $M(\cdot)$.  Suppose $\nu_i$
is an orientation of $M_i(\cdot)$.
Then, after passing to a subsequence, the $\nu_i(\cdot)$ converge to an orientation
  $\nu(\cdot)$
of $M(\cdot)$.

If the $\nu_i$ are strong, then so is $\nu$.
\end{theorem}

\begin{proof}
This follows easily from the Local
Regularity Theorem (Theorem~\ref{local-theorem}).
\end{proof}

\section{Platonic
and Regular Polygon Shrinkers}\label{platonic-section}

In this section, we produce shrinkers that have the
symmetries of a Platonic solid and that are disjoint
from the rays through the centers of the faces of that  solid.  
For each $n\ge 3$, we also produce a
shrinker that has the symmetries
of a regular $n$-gon in the $xy$-plane and that
is disjoint from the rays through the vertices of the $n$-gon.
Specifically, we prove

\begin{theorem}\label{plato-n-gon-theorem}
Consider a set $Q$ consisting of a finite collection of rays in $\RR^3$ from the origin.
We assume that the rays intersect the unit sphere in one of the following:
\begin{enumerate}[\upshape(i)]
\item\label{plato-case} 
The centers of the faces of a Platonic solid, or
\item\label{n-gon-case}
Three or more equally spaced points in  circle $\SS^2\cap\{z=0\}$.
\end{enumerate}
Let $g$ be number of rays minus $1$.
There exists a smoothly embedded shrinker $\Sigma$ such that
\begin{enumerate}
\item\label{symmetry-item} $\Sigma$ has the symmetries of $Q$ and is disjoint from $Q$.
\item\label{genus-item} $\Sigma$ has genus~$g$.
\item\label{blowup-item} $\Sigma$ occurs as a shrinker  in 
   a mean curvature flow $M(\cdot)$ 
   in $\RR^3$ for which
   $M(0)$ is a compact, smoothly embedded surface
   in $Q^c$.
\end{enumerate}
\end{theorem}

We do not know whether any of the shrinkers in 
Theorem~\ref{plato-n-gon-theorem} are compact.
In the cases of the cube and the octahedron,
shrinkers as in Theorem~\ref{plato-n-gon-theorem} were found numerically by
 Chopp~\cite{chopp},
and they are compact.  
See Figure~\ref{fig:cube}.
 We would guess that all the Platonic solid shrinkers are compact.
In the regular $n$-gon case, 
one can show that for all sufficiently large $n$, the shrinker must be non-compact.
Indeed, one can show that as $n\to\infty$,
the shrinker converges to a multiplicity $2$ plane.  

\begin{figure}
    \centering
    \includegraphics[width=0.4\linewidth]{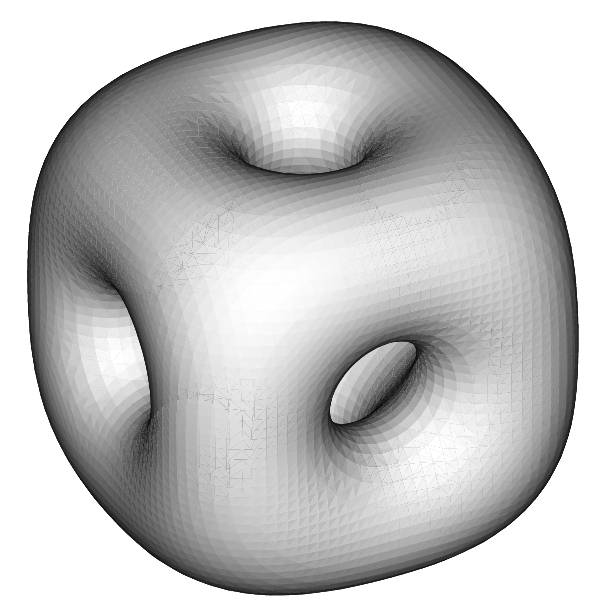}
    \caption{A cubical shrinker, found numerically by 
    David Chopp. Figure courtesy of David Chopp.}
    \label{fig:cube}
\end{figure}

Using minimax methods, Ketover proved existence
 of shrinkers
with properties~\eqref{symmetry-item} 
 and~\eqref{genus-item}
in~\cite{ketover-platonic}
 (the Platonic case~\eqref{plato-case})
and in~\cite{ketover-fat} (the regular polygon case~\eqref{n-gon-case}).
Assertion~\eqref{blowup-item} in the Platonic solid case is new.

In the regular $n$-gon case (Case [ii]),
Theorem~\ref{plato-n-gon-theorem} was proved, by similar methods to those
in this paper, by Ilmanen and White~\cite{iw-fat}.  However, 
 the surface $M$ in that paper is quite different
 than the surface $M$ in our proof of
  Theorem~\ref{plato-n-gon-theorem}:
 the $M$ in this proof is disjoint from $Q$ and intersects $Z$ in $4$ points, whereas the 
 corresponding $M$
  in~\cite{iw-fat} intersects every horizontal
 ray from the origin, and it intersects $Z$ in a pair of points.  The shrinkers $\Sigma$ produced here may well be different from the corresponding shrinkers produced in~\cite{iw-fat}.
 In particular, the shrinkers produced in \cite{iw-fat}
 are all non-compact, whereas 
 the shrinkers produced here might be compact
 when $g$ is small.

Note that $Q^c$ is diffeomorphic to $(\SS^2\setminus \{q_1, \dots, q_{g+1}\})\times \RR$
and therefore homotopy equivalent to $\SS^2\setminus\{q_1,\dots, q_{g+1}\}$.  Thus
\[
H_1(Q^c; \ZZ_2) \cong \ZZ_2{}^g.
\]

We let $G$ be the group of symmetries of $Q$:
\[
 G = \{\sigma\in O(3): \sigma(Q)=Q \}.
\]

If $K$ is a compact set in $\RR^3$, we let 
\[
  T_Q = T_Q(K) = \inf \{t\ge 0: Q\cap K(t)\ne \emptyset\}.
\]
If $T_Q<\infty$, then $K(T_Q)\cap Q$ is nonempty.

The following lemma says that if $M$ separates $K$ from $Q$, then $M$ has to bump into $Q$ before $K$ can:

\begin{lemma}\label{between}
Suppose $K$ and $M$ are disjoint compact sets,
and that no connected component of $M^c$ contains points
of $K$ and of $Q$.  Then $T_Q(M)< T_Q(K)$.
\end{lemma}
\begin{proof}
    The proof of this lemma is an immediate consequence of Theorem 5.2 in \cite{white-95}.
\end{proof}

\begin{definition}\label{A-definition}
Let $\Aa$ be the class of compact, smoothly embedded surfaces $M$ in $Q^c$
such that 
\begin{enumerate}
\item $\sigma(M)=M$ for every $\sigma\in G$.
\item $\genus(M)\le g$.
\item $T_Q(M)<\infty$.
\end{enumerate}
\end{definition}

\begin{theorem}\label{preservation}
If $M\in \Aa$ and if $t<\min\{T_Q, \Tpos\}$ is a regular time of the  flow $M(\cdot)$,
then $M(t)\in \Aa$.
\end{theorem}

Later we will see that $T_Q=\Tpos$.

\begin{proof}
The theorem is a trivial consequence of the facts that (1) level set flow
preserves symmetries, and (2) for any compact smooth embedded surface, $\genus(M(t))$ is a decreasing function of regular times in $[0,\Tpos]$.  
(See~\eqref{genus-decreases}.)
\end{proof}

\begin{theorem}\label{convex-hull}
If $M\in \Aa$, then there is a connected component $\Sigma$ of $M$
such that the convex hull of $\Sigma$ contains points of $Q$.
\end{theorem}

\begin{proof} 
Since $M(T_Q)$ contains a point $p$ of $Q$, there must be a connected
component $\Sigma$ of $M$ such that $p\in \Sigma(T_Q)$.
Now $\Sigma(t)$ lies in the convex hull $H$ of $\Sigma$ for all $t$, 
so $p\in \Sigma(T_Q)\subset H$.
\end{proof}

Thus every $M\in \Aa$ has the properties described by the following theorem:

\begin{theorem}\label{geometry-theorem}
Let $M$ be a compact,
 $G$-invariant surface embedded in $\RR^3$,
  disjoint from $Q$, and of genus~$\le g$.
Suppose $M$ has a connected component $\Sigma$ whose convex hull
contains points of $Q$.
Then $\Sigma$ is a $G$-invariant, genus~$g$ surface, and if $\Omega$ is a  
  simply connected
open subset of $Q^c$, then $\genus(M\cap\Omega)=0$.
\end{theorem}

\begin{proof}[Proof of
   Theorem~\ref{geometry-theorem}]
Let $\Pp$ be the collection of planes of reflectional symmetry of $Q$.
Note that $\SS^2\setminus \cup \Pp$ is a union of open, congruent, disjoint, 
 triangular regions in $\SS^2$,
 as in Figure~\ref{icosahedron-figure}.
Likewise, $\RR^3\setminus \cup \Pp$
is a union of regions. 
Each region in $\RR^3\setminus \cup\Pp$ corresponds to one of the triangles 
in $\SS^2\setminus \cup\Pp$; the region has $3$ edges, corresponding to the $3$ vertices of the triangle, and three faces, corresponding to the three
edges of the triangle.

Let
\[
  \Tt = \{ \overline{W}: 
  \text{$W$ is a component of $\RR^3\setminus \cup \Pp$}  \}.
\]
Let $\Delta\in \Tt$ be such that the interior of 
 $\Delta$ contains a point
of $\Sigma$. 
Exactly one of the edges of $\Delta$ is a ray of $Q$.
Label it $E_1$.
Let $E_2$ and $E_3$ be the other two edges of $\Delta$.
For $i\ne j$ in $\{1,2,3\}$, let $F_{ij}$ the face bounded by rays $E_i$ and $E_j$.
\begin{figure}
    \centering
    \includegraphics[width=0.5\linewidth]{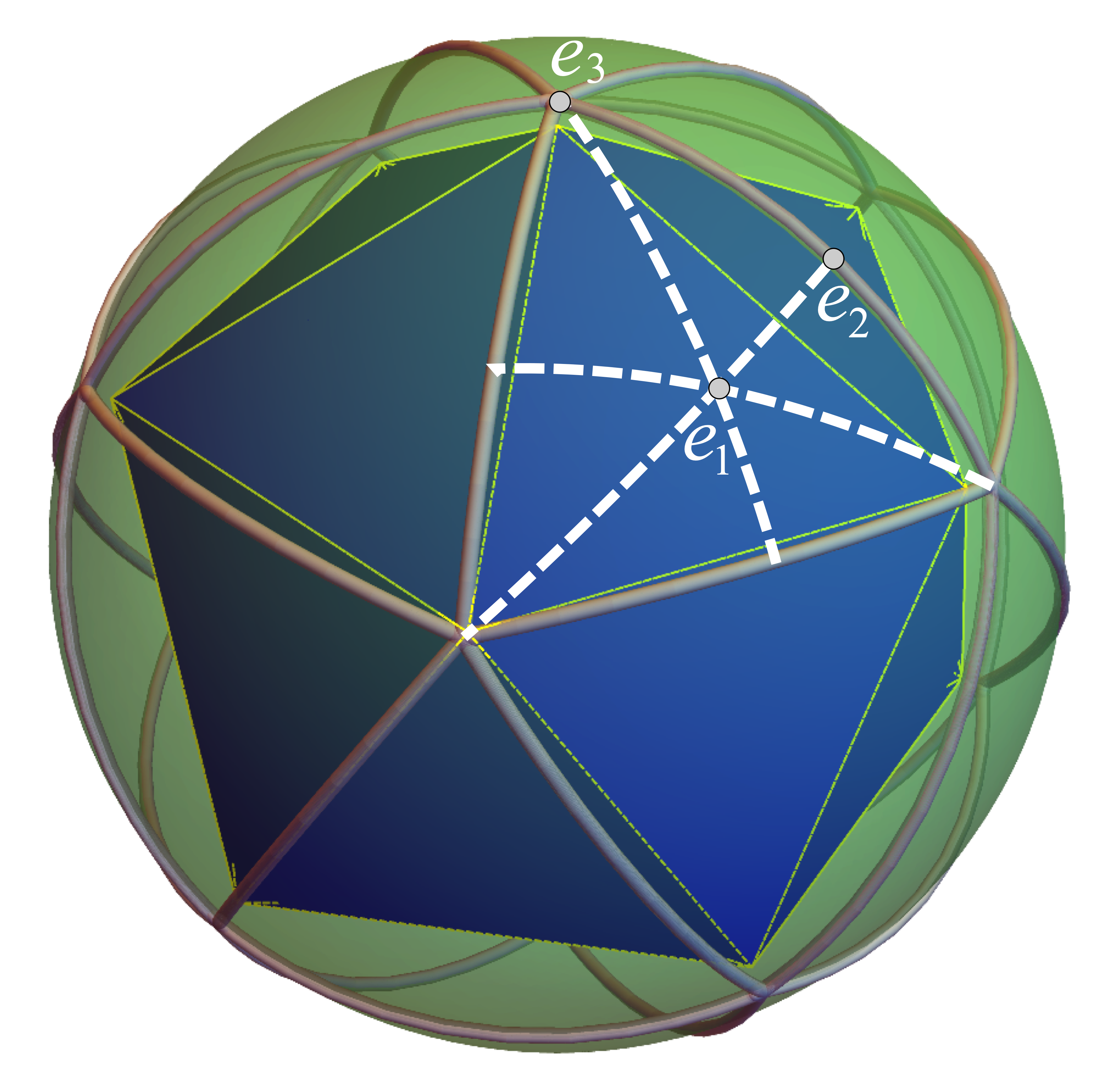}
    \caption{
    $\SS^2\setminus \cup \Pp$  is a union of spherical triangles. 
    In the case of the icosahedron, the figure shows the spherical triangles in the region corresponding to one face.  The points $e_i$ are defined by  $e_i:= E_i \cap \SS^2$, $i=1,2,3$.  Here, the $E_i$ are the rays defined at the beginning of the proof of Theorem~\ref{geometry-theorem}.}
    \label{icosahedron-figure}
\end{figure}

For $i\in \{1,2,3\}$, let $K_i$ be the union of the $\Delta'\in \Tt$
such that one edge of $\Delta'$ is $E_i$.   Note that  $K_i$ is a convex cone.
Also,
\begin{align*}
Q\cap K_1 &= E_1, \\
Q\cap \interior(K_i) &= \varnothing \quad
 (i=2, 3).
\end{align*}

Let $\Psi$ be a connected component of $\Sigma\cap \Delta$.

\begin{claim} 
$\Psi$ contains points in $F_{12}$ and points in $F_{13}$.  
\end{claim}

Suppose not. Then (by relabelling) we may assume that $\Psi$ contains
no points in the face $F_{12}$.  It follows, by iterated reflections 
in the planes in $\Pp$ containing $E_3$, that $\Sigma$ lies
in $\interior(K_3)$, a convex open set disjoint from $Q$.
But that contradicts the hypothesis that the convex hull of $\Sigma$ contains points of $Q$.
Thus the claim is proved.

Let $\gamma$ be an embedded curve in $\Psi$ such that one of its endpoints is
in $F_{12}$ and the other endpoint is in $F_{13}$.
  Note that $\gamma$ is disjoint 
 from~$E_1$ since $M\subset Q^c$.
Iterated reflections in the planes of $\Pp$ containing $E_1$ produces
a simple closed curve $C$ in
$\Sigma^*\cap \interior(K_1)$ that winds once around $E_1$ in $K_1\setminus E_1$.
In particular, 
\begin{equation}\label{winding}
\begin{gathered}
\text{$C$ is homotopic in $Q^c$ to an arbitrarily}
\\
\text{small circle in $\SS^2$ centered at $e_1:=E_1\cap \SS^2$.}
\end{gathered}
\end{equation}
Since $C\subset \Sigma$ is homologically nontrivial in $Q^c$,
it is homologically nontrivial in $\Sigma$, and thus
\[
  \genus(\Sigma)\ge 1.
\]

\begin{claim}
    $\Psi$ contains points in $F_{23}$.
\end{claim}

Suppose not.  
It follows, by iterated reflections
 in the planes in $\Pp$ containing $E_1$, 
 that $\Sigma$ lies in $K_1$.  For each ray $E_i'$ of $Q$,
 $M$ contains a copy of $\Sigma$.  There are $g+1$ such
 rays, so $M$ contains $g+1$ such copies.  Since each copy has positive genus, $M$ has genus $\ge g+1$, 
 a contradiction.  Thus the claim is proved.

Since $\Sigma$ contains points in each $F_{ij}$, it follows
that $\Sigma$ is invariant under reflection in each of the corresponding three planes.  Those reflections generate $G$, so $\Sigma$ is invariant under $G$.

Now $E_1$ is one of the rays of $Q$.  Let $E_1'$
be another ray of $Q$, and let $\sigma\in G$ map
$E_1$ to $E_1'$.  Then $C':=\sigma(C)$ is a simple closed
curve in $\Sigma\cap \interior(K_1')$, where 
$K_1'=\sigma(K_1)$, and $C'$ winds once around 
 $E_1'$ in $K_1'\setminus E_1'$.
Note that $C$ and $C'$ are disjoint since 
  $\interior(K_1)$ and $\interior(K_1')$
are disjoint.

We get such a curve $C'$ for each of the $(g+1)$ rays $E_1'$
 of $Q$.
 Let $\Gg$ be the set of  those $g+1$ curves,
 and let $\Gg'$ be a subset consisting of $g$ of those curves.
 From~\eqref{winding}, we see that the curves in $\Gg'$
 form a basis for 
  $H_1(Q^c)\cong H_1(\SS^2\setminus Q)$.
 We claim that $\Sigma\setminus \cup \Gg'$
 is connected. For if not, let $\Rr$ be one of the
 components of $\Sigma\setminus \cup\Gg'$.
 Then $\partial \Rr$ is consists of one or more of the curves in $\Gg'$, which is impossible since those curves are linearly
 independent in $H_1(Q^c)$.

Since $\Sigma\setminus \cup\Gg'$ is connected,
\[
  0 \le \genus(\Sigma\setminus \cup\Gg') 
  = \genus(\Sigma) - g
  \le
  \genus(M)-g \le 0.
\]  
Thus
\begin{equation}\label{gee}
   \genus(\Sigma) = \genus(M) =g,
\end{equation}
and
\begin{equation*}
 \genus(\Sigma\setminus \Gg') = 0.
\end{equation*}
By~\eqref{gee}, $M\setminus\Sigma$
is a union of spheres, so
\begin{equation}\label{more-zero}
  \genus(M\setminus \Gg') 
  =
  \genus(\Sigma\setminus \Gg')
  = 0.
\end{equation}

We have shown that $\Sigma$ is a $G$-invariant
 surface of genus $g$, 
 and also that the genus of $M$ is equal to $g$. It remains to prove the last assertion of the theorem.

Let $U$ be the region bounded
by $M$ that is disjoint from $Q$. (Thus, a point in $M^c$ is in $U$ if and only if
there is a path from that point to a point in $Q$ that is transverse to $M$ and that
intersects $M$ in an odd number of points.) Let 
 $K=\overline{U} = U\cup M$. Thus $K$ is smooth
3-manifold-with-boundary, the boundary being $M$.
  Since $K$ is a  smooth, compact $3$-manifold-with-boundary in $\RR^3$, we can use \cite{hatcher}*{Lemma 3.5} and Poincaré duality to get that
\[
   \dim H_1(K) \geq \frac12 \dim H_1(M) = g.
   \]
On the other hand, one can easily get the reverse inequality using the Mayer-Vietoris Theorem. Hence, we deduce that:
\[
   \dim H_1(K) = \frac12 \dim H_1(M) = g.
\]

Now consider the inclusions
\[
   \Gg' \subset K \subset Q^c.
\]
Since $\Gg'$ is a basis for $H_1(Q^c)$,
and since $H_1(\Gg)$, $H_1(K)$, and $H_1(Q^c)$ each
have dimension~$g$, the inclusion~$\iota$ of $K$ into $Q^c$
induces an isomorphism on $H_1$:
\begin{equation}\label{iota}
   \iota_\# : H_1(K) \xrightarrow{\cong} H_1(Q^c).
\end{equation}

It remains to show that if $\Omega$ is a simply connected open subset of $Q^c$,
then $M\cap \Omega$ has genus~$0$.
Suppose, to the contrary, that $M\cap \Omega$
has positive genus.
Then, by elementary topology
 (see Lemma~\ref{local-lemma} below),
$M\cap \Omega$ contains a simple closed curve
  $\alpha$
  that does not bound a surface
in $K$.  Therefore $\alpha$ does not bound a surface in $Q^c$ since the map~\eqref{iota}
is an isomorphism, and therefore $\alpha$ does not bound a surface in $\Omega$.
But that contradicts the assumption that 
 $\Omega$ is simply connected.
\end{proof}

The following theorem gives additional information about the surface
 in Theorem~\ref{geometry-theorem}.

 \begin{theorem}\label{more-info-theorem}
Let $M$ be as in Theorem~\ref{geometry-theorem}
and let $R>0$.
Suppose the convex hull of a connected
component of $M\cap B(0,R)$
contains points in $Q$.  Then for every $r>R$, $M\cap B(0,r)$ has genus~$0$ or genus~$g$.
\end{theorem}

\begin{proof}
It suffices to prove
 Theorem~\ref{more-info-theorem} in the case when $\partial B(0,r)$ intersects $M$
transversely.  

{\bf Case 1}: The convex hull of one of the components of $M\cap\partial B(0,r)$
contains a point of $Q$.
Let $C$ be such a component, and let $\Gg$ be the union of its images under 
 the group~$G$. Exactly as in the proof of 
  Theorem~\ref{geometry-theorem}
  (see~\eqref{more-zero}),
  $M\setminus \cup\Gg$ has genus~$0$.
Thus $M\cap B(0,r)$ has genus~$0$.
 
{\bf Case 2}: The convex hull of each component of $M\cap \partial B(0,r)$ is disjoint from $Q$.  Then each such component bounds an embedded disk 
  in
\[
 \RR^3\setminus (B(0,r)\cup Q).
\]
We can choose those disks so that the union $M'$ of those disks and
  $M\cap B(0,r)$ is a smoothly embedded,
  $G$-invariant surface.
By Theorem~\ref{geometry-theorem},
 $M'$ has genus $g$. Thus
\[
   \genus(M\cap B(0,r)) = \genus(M')=g.
\]
\end{proof}

\begin{lemma}\label{local-lemma}
Suppose that  $K$ is a smooth, $3$-dimensional 
  manifold-with-boundary. Let $M=\partial K$.
Suppose that $W$ is a relatively open subset of $K$ such that
 $M\cap W$ has positive genus.
Then $M\cap W$ contains a simple closed curve that is homologically nontrivial in $K$.
\end{lemma}

\begin{proof}
Let $\alpha_1$ and $\alpha_2$ be a pair of simple closed curves in $M\cap W$ that intersect transversely
and at exactly one point.  We claim that at most one of $\alpha_1$ and $\alpha_2$ can be homologically trivial in $K$.
For otherwise, there would be smooth, compact embedded surfaces $\sigma_i$ in $K$ such that $\partial \sigma_i = \alpha_i$.
We can choose the $\sigma_i$ so that they intersect transversely and so that they are never tangent to 
  $M$. 
Thus $\sigma_1\cap\sigma_2$ is a finite collection of curves, and so they have an even number of endpoints.
But the set of their endpoints is $\alpha_1\cap\alpha_2$, which has only one point, a contradiction.
\end{proof}

\begin{theorem}\label{Q-pos}
Suppose $M\in \Aa$.
(See Definition~\ref{A-definition}.)
Then
$T_Q\le \Tpos$.
\end{theorem}

\begin{proof}
Let $0<T< \min\{T_Q,\Tpos\}$.
Let $(p,T)$ be a spacetime singular point.  Since $T<T_Q$,  
$p\in Q^c$.  Let $W$ be an open ball in $Q^c$ centered at $p$.
By Theorems~\ref{preservation}, \ref{convex-hull}, and~\ref{geometry-theorem},  $\genus(M(t)\cap W)=0$ for all regular times $t<T$.
Thus $(p,T)$ must be a genus~$0$ singularity of the flow.
Thus $T<\Tpos$, since at time $\Tpos$, there must be a 
singularity of positive genus.

We have proved that $T<\Tpos$ for every $T< \min\{T_Q, \Tpos\}$.
Hence $T_Q \le \Tpos$.
\end{proof}

\begin{theorem}\label{shrinker}
Suppose that $M\in \Aa$, and, therefore, that $T:=T_Q(M)< \infty$.
Let $p\in M(T)\cap Q$ and let
  $\Sigma$
be a shrinker at the spacetime point $(p,T)$.
\begin{enumerate}
\item\label{case1}
If $p\ne 0$, then $\Sigma$ is a cylinder whose axis is the line $\{rp: r\in \RR\}$.
\item\label{case2} If $p=0$, then $\Sigma$ is a connected, $G$-invariant, genus-$g$ surface disjoint from $Q$.
\end{enumerate}
\end{theorem}

\begin{proof}[Proof in Case~\eqref{case1}]
If $(p,T)$ were a regular point of the flow, 
then for all $t<T$ sufficiently close to $T$,
there would be a unique point $p(t)$ in $M(t)$
closest to $p$. 
 Now the group $G$ contains a nontrivial group  of rotations about the line
 $\{sp: s\in\RR\}$.   Since $p(t)\notin Q$,
the image of $p(t)$ under those rotations would be other points in $M(t)$
closest to $p$,
a contradiction.
Hence $(p,T)$ is a singular spacetime point
of the flow.

Let $\sigma$ be the genus of the singular point $(p,T_Q)$.
Let $S$ be the set consisting of the $g+1$ points in $Q$ at distance $|p|$ from $0$, and let 
\[
 S' = \{(p,T_Q): p\in S\}.
\]
Then all of the $(g+1)$ points in $S'$ have the same genus,
so
\[
   (g+1)\sigma \le \genus(M) = g
\]
by Corollary~\ref{genera-corollary}.
Thus $\sigma=0$, so $\Sigma$ is a sphere or cylinder.
If it were not the cylinder with the indicated axis, then $\Sigma$ would intersect the line $\{sp: s\in \RR\}$
transversely in a nonempty set.  But then $M(t)$ would intersect $Q$ for all $t<T$ sufficiently
close to $T$, which is impossible.   This completes the proof in Case~\eqref{case1}.
\end{proof}

\begin{proof}[Proof in Case~\eqref{case2}]
 First, we show that $\Sigma$ is disjoint from~$Q$.
Note that there exist $t_i\uparrow T$ such that the surfaces
\[
   M_i' := (T-t_i)^{-1/2}M(t_i)
\]
converge smoothly and with
 multiplicity one to~$\Sigma$.
 Thus if $q\in \Sigma$, then 
 (for large~$i$) there is a unique point $q_i$ in $M_i'$ closest to $0$.  But 
  if $q$ were in $Q$, then, since $M_i'$ is disjoint from $Q$,  the symmetries of $M_i'$ imply that there are multiple points in $M_i'$ closest to $q$,
 a contradiction.
Hence $q\notin Q$.
Thus $Q$ and $\Sigma$ are disjoint.

By {\color{black}Theorem~\ref{more-info-theorem},} for each $i$ and each $r>0$,
\[
   M_i'\cap B(0,r)
\]
has genus~$0$ or genus~$g$.
Thus, by smooth convergence of $M_i'$ to $\Sigma$,
 $\Sigma$ has genus~$0$ or genus~$g$.

If $\Sigma$ had genus~$0$, then,
by Brendle's Theorem \cite{brendle}. it would be a plane through the origin,
a cylinder whose axis contains the origin, 
or a sphere centered at the origin, all of which intersect~$Q$.
But $\Sigma$ is disjoint from~$Q$.
Thus,  $\Sigma$
has genus~$g$.
\end{proof}

\begin{theorem}\label{crunch-theorem}
Suppose $M_n$ are compact, smoothly embedded surfaces in $Q^c$ such that 
$T_Q(M_n)=\infty$.  Suppose the $M_n$ converge smoothly to a surface $M$ in $\Aa$.
(See Definition~\ref{A-definition}.)
Then 
\[
   M(T_Q) \cap Q = \{0\},
\]
where $T_Q=T_Q(M)$.
\end{theorem}

\begin{proof}
We can assume that each $M_n$ is nonfattening.  
For let $f_n:\RR^3\to \RR$
be a smooth function such that $f_n^{-1}(0)=M_n$, such that $0$ is not a critical value 
of $f$, and such that $f>0$ at all points of $Q$.  
Note that there are $\eps_n\uparrow 0$ such that if $\eps_n<s_n<0$,
then the surfaces $M_n'=f^{-1}(s_n)$ will converge smoothly to $M$.
We can choose the $s_n$ so that each $M_n'$ is nonfattening (since $f^{-1}(s)$ fattens for at most  countable set of $s$.)   Finally, note that $M_n'(t)$ never 
intersects $Q$, since $M_n'$ lies in the region enclosed by $M_n$.
Thus we can assume $M_n$ is nonfattening (by replacing $M_n$ by $M_n'$, if necessary.)

Suppose the theorem is false. Then there exists a point $p\ne 0$ in $M(T_Q)\cap Q$.

By Theorem \ref{shrinker},
\begin{equation}\label{angy}
    (T_Q- t)^{-1/2}(M(t)-p)
\end{equation}
converges smoothly and with multiplicity $1$ to a cylinder $\Sigma$
whose axis is the line $\RR p:= \{sp: s\in \RR\}$.
Let $A$ be a scaled Angenent torus that lies in the non-simply connected component of  the complement of
$\Sigma$, and that collapses to the origin.

Then, for all $t<T_Q$ sufficiently close to $T_Q$, $A$ and $0$ will lie in different connected
components of  the complement of the surface~\eqref{angy}. 
 Fix such a time~$t$ that is regular for flow~$M(\cdot)$ and also for each of the 
 flows~$M_n(\cdot)$.  Note that the
 scaled Angenent torus
\begin{equation}\label{tilde-A}
  \tilde A := (T_Q-t)^{1/2}A + p
\end{equation}
and the point $p$ lie in different components of {\color{black} the complement} of $M(t)$ {\color{black}in $\RR^3$}.

(The reader may notice that the 
  right-hand side of~\eqref{tilde-A}
  would not be a mean curvature flow
  if $t$ were varying.
But that is irrelevant: we have fixed a time~$t$ and thus are defining a single surface $\tilde A$, not a one-parameter
family of surfaces.)

By the smooth convergence of 
 $M_n(t)$ to $M(t)$, 
$\tilde A$ and $p$ lie in different components of $M_n(t){\color{black} ^c}$ for all sufficiently large~$n$.

But then, by Lemma~\ref{between}, $p\in M_n(\tau)$ for some $\tau>t$ and thus
  \[ 
     T_Q(M_n) < \infty,
\]
a contradiction.
\end{proof}

\begin{theorem}\label{basic-plato-theorem}
Suppose $\Ff$ is a family
of compact, connected, smoothly embedded, $G$-invariant surfaces in $\RR^3\setminus Q$
of genus~$g$.
Suppose also that $\Ff$
is connected
with respect to smooth convergence,
and that all the $M\in \Ff$ lie in some large ball $B(0,R)$.
Let $\Ss$ be the subfamily of $M\in \Ff$ for which $T_Q(M)< \infty$.
In other words, let $\Ss=\Ff\cap\Aa$.
(See 
Definition~\ref{A-definition}.)
If $\Ss$ and $\Ss^c:=\Ff\setminus \Ss$ are both nonempty, then there is an 
  $M\in \Ss$ for which
\[
M(T_Q)\cap Q = \{0\}
\]
(where $T_Q=T_Q(M)$).
If $\Sigma$ is a shrinker at the origin at
time $T_Q$ to the flow $M(\cdot)$, 
then $\Sigma$ has genus~$g$ and is disjoint
from $Q$.
\end{theorem}

\begin{proof}
Under mean curvature flow, the
 sphere $\partial B(0,R)$
shrinks to a point in time $R^2/4$.
Thus, if $M\in \Ff$, then $M(t)$ is empty for all 
  $t> R^2/4$.
We claim that $\Ss$ is a closed subset of $\Ff$.
To see that, suppose that $M_n\in \Ss$ converges
to $M\in \Ff$. 
By hypothesis, there exist $t_n$ and $p_n$ such that
\[
   p_n \in Q\cap M_n(t_n).
\]
Since $p_n\in B(0,R)$ and since $t_n\le R^2/4$,
we can assume that $p_n$ and $t_n$ converge
so to finite limits $p$ and $t$.  Thus
\[
  p\in Q\cap M(t),
\]
so $M\in \Ss$.
Thus $\Ss$ is a closed subset of $\Ff$.

By connectedness, $\Ss^c$ cannot be closed: there must be $M_n\in \Ss^c$
converging to an $M\in \Ss$. 
 By Theorem~\ref{crunch-theorem}, $M(T_Q)\cap Q = \{0\}$.
By Theorem~\ref{shrinker}, any shrinker at $(0,T_Q)$ has genus~$g$ and is disjoint from $Q$.
\end{proof}

\begin{theorem}
\label{family-theorem}
For each $Q$ as in Theorem~\ref{plato-n-gon-theorem},
there exists a family $\Ff$
satisfying the hypotheses of
 Theorem~\ref{basic-plato-theorem}.
\end{theorem}

The reader may regard 
Theorem~\ref{family-theorem} as obviously true.  For completeness, we give a proof.

\begin{proof}
First, we give the proof in the case of a Platonic solid.
  Here, $g+1$ equals the number of faces. Let $E$, $U$, and $C$ be the radial projections to $\SS^2(1)$ of the
closed edges, the open faces, and the centers of the faces, respectively, of the 
solid.  Note that $C= Q\cap \SS^2(1)$.

We will produce a continuous, one-parameter family of smooth surfaces that degenerate at one end to $E$. At the other end, they degenerate to the union of $\SS^2(1/2)$ and 
$\SS^2(2)$, together with the the intervals in $Q$ that connect the two spheres. They will be level sets of a function whose gradient is nonvanishing, away from the limiting sets.

 Let $f:\SS^2\to [0,\infty]$ be a continuous, $G$-invariant function such that
\begin{enumerate}
\item $f^{-1}(0) = E$, 
\item $f^{-1}(\infty) = C$,
\item $f|\SS^2\setminus C$ is smooth and has as no critical points in $U\setminus C$.
\item For each $s\in (0,\infty)$, each component of $\{f\ge s\}$  is 
geodesically convex. (There is one component in each face of $U$.)
\end{enumerate}

Let 
\[
  h: (1/2, 2)\to [0, \infty)
\]
be a smooth function such that 
\begin{enumerate}
\item $h(1)=0$.
\item $h'<0$ on $(1/2,1)$.
\item $h'>0$ on  $(1,2)$.
\item $h(r)\to\infty$ as $r\to 1/2$ or as $r\to 2$.
\end{enumerate}
For example, we could let
\[
   h(r) = 
   \left(\log \frac{(2r-1)}{2-r} \right)^2.
\]
Extend $h$ to the  closed interval by setting $h(1/2) = h(2)=\infty$.

Now let  $${\bf W}:=\{p\in\RR^3: 1/2\leq |p|\leq 2\}, $$ and define
$
\Phi: {\bf W} \to [0,\infty]
$
by
\[ 
\Phi(p) = f(p/|p|) + h(|p|).
\]
For $s\in [0,\infty]$, let 
$ M^s:=\Phi^{-1}(s)$.
By construction, $\Phi$ is $G$-invariant and has nonvanishing gradient on 
$$
\interior({\bf W}) \setminus(Q\cup E) =\bigcup_{0<s<\infty}M^s.
$$ 
Hence the surfaces $M^s$ are $G$-invariant, smooth, and vary smoothly for $s\in(0,\infty)$. Also, for $s>0$ small, the genus of $M^s$ is the same as the genus of the boundary of a tubular neighborhood of $E=M^0$, which is equal to $g$. It follows that the genus of $M^s$ equals $g$ for all $s\in(0,\infty)$.

Now,  $M^0 = E$  vanishes instantly under level set flow,
so $T_Q(M^0)=\infty$ and therefore $T_Q(M^s)=\infty$ for all sufficiently small  $s$.

Moreover,  as $s\rightarrow\infty$, 
\[
M^s\rightarrow M^\infty 
= \SS^2(1/2)\cup\SS^2(2)\cup ({\bf W}\cap Q). 
\]

Thus, for all sufficiently large $s$, the region $\{\Phi<s\}\subset ({\bf W}\setminus Q)$
enclosed by $M^s$ contains a  suitably scaled Angenent torus centered at a point 
$p\in {\bf W}\cap Q$.
It follows from Lemma~\ref{between} that $T_Q(M^s) < \infty$. 
Therefore, the family
\[
   \Ff= \{ M^s:  0 < s <\infty\}
\]
satisfies the hypotheses of
 Theorem~\ref{basic-plato-theorem}.
This completes the proof in the case of Platonic solids.

Now consider the case when $Q$ consists of the rays
\[
\{ (r\cos\theta, r\sin\theta, 0): r\geq 0\}
\]
such that $\theta$ is  multiple of $2\pi/(g+1)$. 
In this case, let $E$ be the set of points
 $(r\cos\theta,r\sin\theta, z)$
  in $\SS^2$
 such that $\theta$ is an odd multiple
  of $\pi/(g+1)$.
  Thus $E$ consists of semicircles from the North Pole to the South Pole.
Let $U= \SS^2\setminus E$ and let
 $C=Q\cap \SS^2$.
The rest of the proof is exactly as in the Platonic solid case.
\end{proof}

\section{A Topological Lemma}
\label{topology-section}

In \S\ref{platonic-section}, we produced shrinkers that are disjoint
from certain collections of rays from the origin.
In the remainder of the paper, 
we will produce
shrinkers that contain certain collections of lines
through the origin.
This section proves basic facts about the topology
of such surfaces.

Let $g\ge 1$ be an integer.
Let 
 $Q=Q_g$ be the union of the $g+1$ lines
\[
  \{ (r\cos\theta, r\sin\theta,0): r\in \RR\},
\]
  such that $\theta$ is a multiple of $\pi/(g+1)$.
For $\theta\in \RR$, let $P_\theta$ be the vertical
plane spanned by $(\cos\theta,\sin\theta,0)$ and $\ee_3$.
Let $G=G_g$ be the group of isometries of $\RR^3$
generated by:
\begin{enumerate}
    \item Reflections in the planes $P_\theta$
    for which $\theta$ is an odd multiple of 
    $\pi/(2(g+1))$,
    \item Rotations about $Z$ by multiples of $2\pi/(g+1)$.
    \item Rotation by $\pi$ about each of lines
       in $Q_g$.
\end{enumerate}

(Equivalently, $G_g$ is the group of symmetries of the surface given in cylindrical coordinates by 
\[
  z = r^{g+1}\sin( (g+1)\theta).
\]
The group $G_g$ is also the group of symmetries of the Costa-Hoffman-Meeks surface of genus~$g$.)

In \S\ref{desing-section}, we will produce (for each $g\ge 1$)
a $G_g$-invariant shrinker
$\Sigma$ of genus~$g$ such that
\begin{equation}\label{slicer}
  \Sigma\cap\{z=0\} = Q_g.
\end{equation}
In \S\ref{new-examples-section}, we will produce (for all sufficiently large $g$)
a $G_g$-invariant shrinker $\Sigma$
of genus~$2g$ such that~\eqref{slicer} holds.
In this section, we prove some theorems that 
apply
in both of those cases.

\begin{lemma}[Topology Lemma]\label{topology-lemma}
Let $U$ be a $G$-invariant  region in  $\RR^3$ 
such that $\partial U$ is smooth and such that $U$ is diffeomorphic to an open ball.
Suppose that $M\subset \overline{U}$ is a smoothly embedded,
 $G$-invariant, compact, connected
  $2$-manifold-with-boundary 
 such that 
\begin{gather*}
\partial M = M\cap\partial U,
\\
M\cap \{z=0\}\cap U = Q\cap U, \,\text{and}
\\
\genus(M)\le 2g.
\end{gather*}
Then $M$ has genus~$0$, $g$, or~$2g$.

Let $c$ be the number of components 
of $\partial M$ 
that lie in $\{z>0\}$ and that
wind once around $Z$.
Let $d$ be the number of points in
\[
   M\cap Z^+ = M\cap Z\cap \{z>0\}.
\]
Then
\[
   \genus(M) = (c+d)g.
\]
\end{lemma}

\begin{remark} For surfaces of genus $\le g$, a result
similar to Lemma~\ref{topology-lemma} was proved,
using similar methods, by 
 Buzano, Nguyen, and Schulz.
 See Lemma~2.10 and Corollary~2.11 of~\cite{buzano}.\end{remark}

\begin{proof}[Proof of Lemma~\ref{topology-lemma}]
 Let $W$ be the wedge given in cylindrical 
coordinates by
\[
 \frac{-\pi}{2(g+1)} < \theta < \frac{\pi}{2(g+1)}.
\]
Note that
\[
    \{x: (x,0,0)\in M\}
\]
is a compact interval $[-R,R]$.

Consider the component $J$ of 
\[
    (\partial M)\cap W = M\cap \partial U\cap W
\]
that contains the point $(R,0,0)$.
 Because $J$ is an embedded curve in $\partial U$ and is invariant  under rotation by $\pi$ about the $x$-axis, it cannot be a closed curve.  Thus, $J$ is a non-closed curve
with $\theta=\pi/(g+1)$ on one endpoint and 
  $\theta= -\pi/(g+1)$ on the other endpoint.
 It follows
  from the symmetries of $M$ that the component $J'$ of 
  $\partial M$ containing $J$ is a simple
  closed curve that winds once around the $z$-axis and that passes through all the points of
   $Q\cap  \partial U$.

Attach disks symmetrically to each of the other 
connected
components of $\partial M$ to get a surface  $M'$ with
only one boundary component, namely $J'$.
Note that
\begin{equation}\label{same-genus}
  \genus(M')=\genus(M).
\end{equation}

Let $k$ be the number of points in $M'\cap Z^+$.
Note that
\begin{equation}\label{kcd}
   k = c + d.
\end{equation}

By symmetry, the number of points of $M'\cap Z$ is 
 $2k+1$.  Let $M''$ be the quotient of $M'$
under the group consisting of rotations about $Z$
by angles that are multiples of $2\pi/(g+1)$.
Then
\[
  \chi(M'\setminus Z) = (g+1)\chi(M''\setminus Z)
\]
so
\begin{equation}\label{eulerish}
  \chi(M') - (2k+1) = (g+1)(\chi(M'') - (2k+1)).
\end{equation}
Now $M'$ is connected and has exactly one boundary
component (namely $J'$), and likewise for $M''$, so
\begin{align*}
\chi(M') &= 1 - 2\genus(M'), \\
\chi(M'') &= 1 - 2\genus(M'').
\end{align*}
Thus, multiplying~\eqref{eulerish} by $-1$ gives
\[
 2\genus(M') +  2k
 =  
 (g+1)(2\genus(M'')  + 2k),
\]
so
\begin{equation}\label{genus-formula}
  \genus(M') = gk + (g+1)\genus(M'').
\end{equation}
The image of $Q\cap U$ under the 
branched covering map
\[
  \pi: M' \to M''
\]
is an embedded curve  $\Gamma$ whose endpoints are both on $\partial M''$.   Now $\Gamma$ divides
 $M''$ into two congruent surfaces
 \[
 (M'')^+ = M''\cap \{z>0\}
 \]
 and
 \[
 (M'')^- = M'' \cap \{z<0\}.
 \]
 Thus
 \begin{align*}
    \genus(M'')
    &=
    \genus((M'')^+) + \genus((M'')^-)
    \\
    &= 2\genus((M'')^+).
\end{align*}

Therefore, by~\eqref{genus-formula},
\[
\genus(M') = gk + 2(g+1)\genus((M'')^+).
\]
Since $\genus(M')\le 2g$, we see that $(M'')^+$ has
genus~$0$, that $k$ is $0$, $1$, or $2$, and thus that
\begin{align*}
\genus(M)
&=
\genus(M')
\\
&=
gk
\\
&=
g(c+d).
\end{align*}
by~\eqref{same-genus} and~\eqref{kcd}.
\end{proof}

\begin{corollary} \label{topology-corollary}
Suppose $M$ is a $G_g$-invariant embedded shrinker in $\RR^3$ such that
\[
  M\cap \{z=0\} = Q_g
\]
and such that $\genus(M)\le 2g$.
Then $M$ has genus~$g$ or $2g$.
\end{corollary}

\begin{proof}
For $R>0$, let $M_R$ be the connected component of
  $M\cap B(0,R)$ that contains the origin.
Since $M$ is a shrinker, it is connected.
Thus
\[
   \genus(M_R)= \genus(M)
\]
for all sufficiently large $R$.
Choose such an $R$ for which $\partial B(0,R)$
is transverse to $M$.
Now apply Lemma~\ref{topology-lemma} to $U=B(0,R)$ and to $\overline{M_R}$.
\end{proof}

\begin{corollary}\label{r-R-corollary}
Let $M$ be as in Corollary~\ref{topology-corollary}.
Suppose also that $M\setminus\{z=0\}$
 has genus~$0$.
\begin{enumerate}
\item If $M$ has genus~$g$, then there is an $r\in(0, \infty)$ such that 
\[
    \genus(M\setminus \SS(r)) = 0,
\]
where $\SS(r)=\{p \in  \RR^3 \; : \; |p|=r\}.$ 
\item If $M$ has genus~$2g$ and if 
\[
   \genus(M\setminus\{z=0\})= 0,
\]
    then
    there exist $r$ and 
   $R$ with $0< r\le R < \infty$ such that 
   \[
      \genus(M\setminus(\SS(r)\cup \SS(R)) = 0.
   \]
\end{enumerate}
\end{corollary}

\begin{proof}
Let $\Gamma(r)=\genus(M\cap B(0,r))$.
Trivially, $\Gamma$ is an increasing function
of $r$,  and  $\Gamma(r)=0$ for all sufficiently small $r$.
Let $M'_r$ be the connected component
of $M\cap B(0,r)$ containing the origin.
The portion of $M\cap B(0,R)$
not in $M'_r$ lies in
\[
  M\setminus\{z=0\},
\]
and therefore has genus~$0$,
 so
\begin{equation}\label{Gamma(r)}
   \Gamma(r)= \genus(M_r').
\end{equation}

Now suppose that $M$ has genus~$g$.
By Lemma~\ref{topology-lemma}, The right hand
 side of~\eqref{Gamma(r)} is either $0$ or $g$.
Thus there is an $r$ such that
\[
   \genus(M\cap B(0,\rho)) =
   \begin{cases}  
   0 &(\rho\le r), \\
   g &(r  < \rho).
   \end{cases}.
\]
Thus if $\rho>r$,
\[
  \genus(M\setminus \overline{B(0,\rho)}
  \le \genus(M) - \genus(M\cap B(0,\rho) 
  = g -g = 0.
\]
Letting $\rho\to r$ gives 
$\genus(M\setminus \overline{B(0,r)})= 0$.
Thus
\begin{align*}
\genus(M\setminus \SS(r))
&=
\genus(M\cap B(0,r)) 
+ \genus(M\setminus \overline{B(0,r)})
\\
&=0,
\end{align*}
(Recall that  
$\SS(r)
=
\{p\in \RR^3: |p|=r\}$.)
 This completes the proof
 when $M$ has genus~$g$. 
The case when $M$ has genus~$2g$ is essentially the 
same; the function $\Gamma(\cdot)$ only takes
the values $g$ and $2g$.  We can let 
  $r$ be the largest $r$ for which 
   $\Gamma(r)<g$ and $R$ be the largest $R$
   for which $\Gamma(r)<2g$.
\end{proof}

\section{Surfaces in \texorpdfstring{$\SS^2\times \RR$}{S2 x R}}
\label{surfaces-section}

In the next three sections, we find it useful to introduce a vertical line at infinity in $\RR^3$ (as explained below), allowing us to identify the resulting manifold $N$ with $\SS^2 \times \RR$.  
 This gives rise to compact embedded surfaces (serving as initial surfaces for mean curvature flow) with greater symmetry than is  possible in $\RR^3$.  
 
We add a line at infinity  as follows.
Let $\RR^2\cup\{\infty\}$ be the conformal compactification
of $\RR^2$, and let
\[
  N = (\RR^2\cup \infty)\times \RR.
\]
Thus $N$ is $\RR^3$ with a vertical line $Z_\infty$
at infinity attached:
\[
   Z_\infty= \{\infty\}\times\RR.
\]
We let $O$ be the origin, and we let
\[
  O_\infty= (\infty,0)= Z_\infty\cap\{z=0\}.
\]

For $r>0$, let $\eta_r$ be the metric
on $\RR^2\cup\{\infty\}$ obtained by stereographic
projection from the north pole of a sphere of radius $r$
to the tangent plane at the south pole.
In this metric, the circle
\[
  (x^2+y^2)^{1/2} = 2r
\]
is a geodesic (i.e., a great circle).
Let $\gamma_r$ be the corresponding product metric on
 $N$.
That is, $\gamma_r$ is the product 
of the metric $\eta_r$ on $\RR^2\cup\{\infty\}$
and the standard metric on $\RR$.
Thus $(N,\gamma_r)$ is isometric to $\SS^2(r)\times\RR$.
As $r\to\infty$, $\gamma_r$ converges  to the Euclidean metric
on $\RR^3$.

Except where otherwise specified,  $N$ will refer
the manifold $N$ with the Riemannian metric 
 $\gamma:=\gamma_1$.
 One easily checks that
 \begin{equation}\label{gamma-delta}
   \gamma_{ij}\le \delta_{ij}
 \end{equation}
at all points.

Note that the circle 
\begin{equation}\label{circle-C}
C:=\{(x,y,0): (x^2+y^2)^{1/2} = 2\}
\end{equation}
and the completed line
\[
   \{(r\cos\theta, r\sin\theta, 0): r\in \RR\}
   \cup \{O_\infty\}
\]
are  totally geodesic great circles with respect to $\gamma$.
Thus the cylinder
\begin{equation}\label{cylinder}
    \Cyl: = \{(x,y,z): (x^2+y^2)^{1/2} = 2\}
\end{equation}
is also totally geodesic in $N$.

If $g$ is a positive integer, we let 
$Q=Q_g$ be the set of $g+1$ horizontal lines
and we let $G=G_g$
be the group of isometries of $\RR^3$
that were defined at the beginning of
  Section~\ref{topology-section}.
Note that each $\sigma \in G$ extends
to an isometry of $N:=\RR^3\cup Z_\infty$.

\begin{definition}\label{definitions}
If $\Gamma$ is a great circle in $\{z=0\}$, let
$\rho_\Gamma:N\to N$ denote rotation by $\pi$ about
$\Gamma$. 
Let $\tG=\tG_g$  
denote the group of isometries of $N$
generated by $G$ together with $\rho_C$. 
Let
\begin{align*}
\tQ
&=
\tQ_g 
\\
&=
\overline{Q} \cup C \\
&= Q\cup \{ O_\infty\} \cup C.
\end{align*}
\end{definition}
Note that $\sigma(\tQ_g)=\tQ_g$
for each $\sigma\in \tG_g$.

The theorems in this section and in the following
two sections will all be about the following
class of surfaces:

\begin{definition}\label{MM-definition}
We let $\MM_g$ be the collection of surfaces $M$
with the
 following 
properties:
\begin{enumerate}
    \item $M$ is a compact, smoothly embedded,
      $\tG_g$-invariant surface in $N$.
\item $M$ intersects the cylinder $\Cyl$ 
          (see~\eqref{cylinder}) transversely, 
       and the intersection is the circle
        $C$ (see~\eqref{circle-C}).  
\item $M\cap\{z=0\} = \tQ_g$.
 
    \item $M\setminus\{z= 0\}$ has genus~$0$.
    \item $M$ has genus~$\le 4g$.
    \item $\garea(M) <  2\garea(\{z=0\}) = 8\pi$.
\end{enumerate}
(Recall that $\{z=0\}$ is isometric to the sphere of radius~$1$.)
\end{definition}

Most of this section is dedicated to establishing that if \( M \in \MM_g \), then the evolving surface \( M(t) \) remains in \( \MM_g \) for all regular times \( t \) in the interval \( [0, \Tpos(M)] \). See Theorem~\ref{preservation-theorem}. 
 We also analyze the behavior of \( M(t) \) as $t$ approaches \( \Tpos(M) \). Specifically, we demonstrate that if \( \Tpos(M) \) is finite, then the origin is a singular point at 
  time \(\Tpos(M) \). 
Additionally, we show that each shrinker $\Sigma$   at \( (O, \Tpos) \) has genus~$g$ or $2g$, and
that its intersection with $\{z=0\}$ is $Q$. 
See Lemma~\ref{origin-lemma} and Theorem~\ref{shrinker-good-theorem}.
     
\begin{proposition}\label{topology-proposition}
Suppose that $M$ is a surface in $\MM_g$
and let $M'$ be the connected component of $M$ 
containing $\tQ$.  Then
\[
  \genus(M)=\genus(M')= 2dg,
\]
where $d$ is the number of points
 in $M'\cap Z^+$.
\end{proposition}

\begin{proof} 
Note that
\begin{equation*}\label{winken}
  \genus(M)=\genus(M')
\end{equation*}
since 
\[
  \genus(M\setminus M') 
  \le \genus(M\setminus\tQ)
  = \genus(M\setminus\{z=0\})
  = 0.
\]
Let $ U = \{(x,y,z): (x^2+y^2)^{1/2}< 2\}$, and note that  $\overline{U} = U \cup \Cyl$, where $\Cyl$ is the cylinder defined in  \eqref{cylinder}. 
Now $M'\cap \overline{U}$ has exactly one boundary component,
namely the circle~$C$ $\subset \{z=0\}$.  
Thus, by 
 Lemma~\ref{topology-lemma},
\[
  \genus(M'\cap U) = dg.
\]
But $M'\setminus C$ is the disjoint union of 
$M'\cap {U}$ and $\rho_C(M'\cap {U})$. Hence 
\begin{equation*}\label{blinken}
\genus(M')
=
2 \genus(M'\cap U).
\end{equation*}
From the three displayed equations above,  we conclude that 
\[
\genus(M)=\genus(M')
=2dg.
\]
\end{proof}

For $\Lambda\ge 0$, let
$\Cone(\Lambda)$ be the set of points~$p$ 
in
\begin{equation}\label{upper-cylinder}
  \{(x^2+y^2)^{1/2}\le 2,\, z\ge 0\}
\end{equation}
such that
\begin{equation}\label{cone-def}
|z|\le \Lambda \dist(\pi(p),C),
\end{equation}
where $\pi$ is the projection $\pi(x,y,z)=(x,y,0)$
and where $\dist$ is geodesic distance with repect
to the metric $\gamma$.
We let $\Cone'(\Lambda)$ be the set
of all $p$ satisfying~\eqref{cone-def}.
Thus $\Cone(\Lambda)$ is the portion 
of $\Cone'(\Lambda)$ in the
 set~\eqref{upper-cylinder}.
Conversely, $\Cone'(\Lambda)$ consists
of $\Cone(\Lambda)$ together with its images
under $\rho_X$, $\rho_C$,
 and $\rho_C\circ \rho_X$.

Note that if $\Lambda$ is sufficiently large, then 
\begin{equation}
\label{Lambda}
M\subset \Cone'(\Lambda).
\end{equation}

\begin{lemma}\label{intersection-lemma}
For $t\in [0,\Tpos ]$,
\[
   M(t)\cap \{z=0\} = \tQ.
\]
Also, if $M\in \Cone'(\Lambda)$,
then $M(t)\in \Cone'(\Lambda)$ for all
$t\in [0,\Tpos]$.
\end{lemma}

\begin{proof}
Let $\Sigma_0$
be the closure of 
\[
   M\cap \{(x^2+y^2)^{1/2} < 2\}\cap \{z>0\},
\]

Note that
\begin{equation}\label{Gamma}
\begin{aligned}
\Gamma & := \partial \, \Sigma_0\\
&=
\Sigma_0\cap\{z=0\} 
\\
&=
(Q\cap \{(x^2+y^2)^{1/2} < 2\}) \cup S,
\end{aligned}
\end{equation}
where $S$ consists of alternate components
of $C\setminus Q$ (see Fig. \ref{fig:gamma}.)
\begin{figure}[htpb]
    \centering    \includegraphics[width=0.5\linewidth]{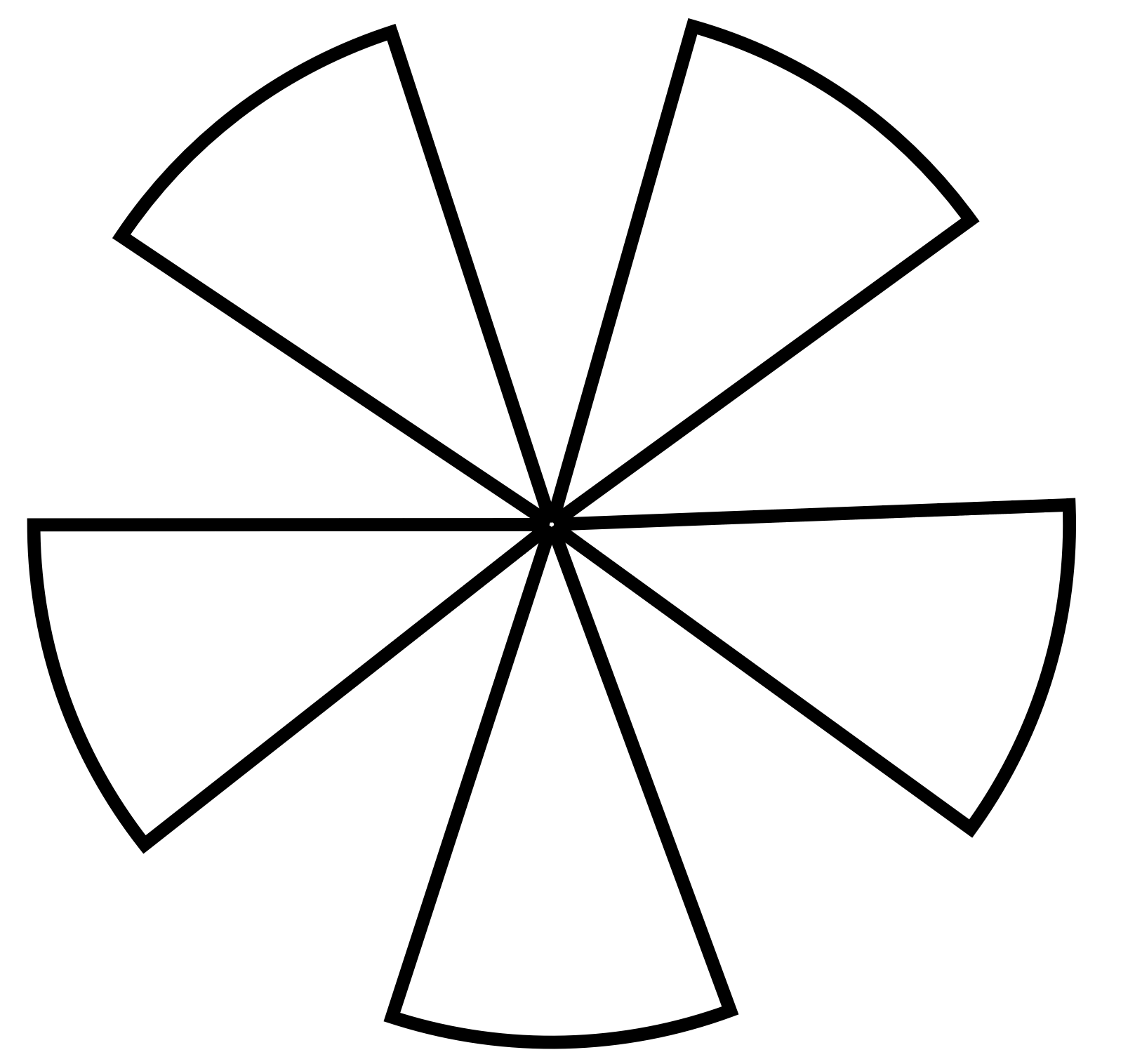}
    \caption{The curve $\Gamma$ in the case $g=4$.}
    \label{fig:gamma}
\end{figure}

We prove the lemma by elliptic regularization.
We work in $N\times\RR$.
We let $\tG$  act on $N\times \RR$ by acting on the first factor.
Recall that $\tG^+$ is the subgroup of $\tG$ that fixes the positive $z$-axis of $N$.
Let $w:N\times \RR \to \RR$ be the
projection onto the second factor.
For $\lambda\ge 0$, let $g_\lambda$ be 
$e^{-(2/3)\lambda  w}$ times the product
metric  on $N\times\RR$. 

Let $A_\lambda$ be a $3$-dimensional surface
(flat chain mod $2$)
that minimizes $g_\lambda$-area among
surfaces in
\[
   \Cone(\Lambda)  \times [0,\infty)
\]
whose boundary consists of
the surface $\Sigma_0\times \{0\}$
  together with
the surface 
  $\Gamma\times [0,\infty)$.

Standard cut-and-paste arguments
 (as in~\cite{morgan}*{Lemma~7.3})
show that $A_\lambda$ is $\tG^+$-invariant.
Alternatively, if one does not wish to invoke those arguments, one can simply let $A_\lambda$
minimize $g_\lambda$-area in the class
of $\tG^+$-invariant surfaces.

We remark that, 
by standard GMT regularity, $A_\lambda$
is smooth except possibly at those points
where its boundary is not smooth.
However, we do not need those regularity results here.

Now let $A'_\lambda$ be the surface
consisting of $A_\lambda$ together with
its images under $\rho_X$, $\rho_C$,
and $\rho_C\circ \rho_X$.
Thus $A_\lambda'$
 is contained in 
 $\Cone'(\Lambda)\times\RR$.
Note that the boundary of $A'_\lambda$
is $M\times \{0\}$, and that 
the integral varifold
associated to $A'_\lambda$
is stationary in the set~$\{w>0\}$
with respect to the 
 metric~$g_\lambda$.

It follows that
\begin{equation}\label{A-flows}
\begin{aligned}
  &t\in (0,\infty) \mapsto 
  A_\lambda(t):=A_\lambda - t\lambda\ee_4,
\\
  &t\in (0,\infty) \mapsto 
  A'_\lambda(t):=A'_\lambda - t\lambda\ee_4
\end{aligned}
\end{equation}
are mean curvature flows
(with moving boundaries).

By the theory of elliptic regularization,
there is a sequence of $\lambda$ tending
to infinity for which the flows~\eqref{A-flows}
converge to integral Brakke flows
\begin{equation}\label{A-limits}
\begin{aligned}
  &t\in (0,\infty) \mapsto A(t),
\\
  &t\in (0,\infty) \mapsto A'(t)
\end{aligned}
\end{equation}
Furthermore, these flows have a special form:
there exist mean curvature flows 
\begin{align*}
  &t\in [0,\infty)\mapsto \Sigma(t), \\
  &t\in [0,\infty)\mapsto \Sigma'(t).
\end{align*}
in $N$ with the following properties:
\begin{enumerate}[\upshape(1)]
\item $A(t)=\Sigma(t)\times \RR$ 
   and $A'(t)=\Sigma'(t)\times\RR$
   for all $t>0$.
\item $\Sigma(0)=\Sigma_0$ and 
   $\Sigma'(0)=M$.
\item $\Sigma(\cdot)$ is an mean curvature
flow with (fixed) boundary, the boundary being 
  $\Gamma=\partial \Sigma_0$.
\item $\Sigma'(\cdot)$ is an integral
   Brakke flow (without boundary).
\end{enumerate}

By construction, 
\[
  \Sigma(t)\subset \Cone(\Lambda) \quad (t\ge 0)
\]
and
\begin{equation}\label{in-cone}
  \Sigma'(t)\subset \Cone'(\Lambda) \quad (t\ge 0).
\end{equation}

Let $U$ be a connected component of 
  $\{z=0\}\setminus \Gamma$.
If $\Sigma(t)$ contained a point of $U$,
then, by the strong maximum principle,
$\Sigma(\tau)$ would contain $U$ for all 
  $\tau\in [0,t]$.
But $\Sigma(0)=\Sigma_0$ is disjoint from $U$.
Thus
\[
   \Sigma(t)\cap U =\emptyset \quad (t\ge 0).
\]
Since $U$ can be any component of 
 $\{z=0\}\setminus \Gamma$,
\[
   \Sigma(t)\cap\{z=0\} = \Gamma.
\]
for all $t\ge 0$.
Now $\Sigma'(t)$ consists of $\Sigma$ 
together with its images under $\rho_X$,
$\rho_C$, and $\rho_C\circ \rho_X$.
Thus
\begin{equation}\label{Sigma-prime-good}
  \Sigma'(t)\cap \{z=0\}= \tQ
  \qquad
  (t\ge 0).
\end{equation}

Now the flow $\Sigma'(\cdot)$ is unit-regular 
and cyclic mod $2$.  (This 
is true for any flow constructed by elliptic
regularization.  The unit-regularity 
follows from the fact
that the flows~\eqref{A-flows} are unit-regular, which
is an immediate consequence of Allard's 
Regularity Theorem.)

Thus 
\begin{equation}\label{the-same}
  \Sigma'(t) = M(t) \quad\text{for all $t\in [0,\Tfat(M))$}.
\end{equation}
(See~\cite{ccs}*{Lemma~9.3}.)

The assertions of the lemma now
follow immediately
from~\eqref{in-cone}, \eqref{Sigma-prime-good}, and~\eqref{the-same}.
\end{proof}

\begin{lemma}\label{C-regular-lemma}
The points in $C$ are regular points
of the flow at all times 
   $t\in [0,\Tpos]$.
\end{lemma}

\begin{proof}
Let $\Sigma$ be a shrinker to the flow
at a spacetime point $(p,t)$, where 
 $p\in C$.
By~\eqref{Lambda} and Lemma~\ref{intersection-lemma}, $\Sigma$ (after a rotation)
is contained in the set
\[
 \{(x,y,z): |z|\le \Lambda |x| \}.
\] 
By \cite{white-boundary}*{Theorem~15.1}, 
$\Sigma$ is a union of halfplanes bounded by
the $y$-axis.  
But $\Sigma$ is smooth with multiplicity $1$,
so it is a single multiplicty~$1$ plane,
and therefore $(p,T)$ is a regular point.
\end{proof}

\begin{lemma}\label{genus-lemma}
Suppose $M\in \MM_g$.  If $t<\Tpos(M)$ is a regular time,
then
\begin{gather*}
\genus(M(t))\le 4g, \\
\genus(M(t)\setminus\tQ) = 0.
\end{gather*}
for all regular times $t<\Tpos(M)$.
\end{lemma}

\begin{proof}
This is a special case of the following
fact (see~\cite{white-genus}*{Corollary~22}):
if $M$ is a smoothly embedded, compact surface,
in a complete $3$-manifold $N$ of bounded Ricci
curvature, if $K\subset N$ is the union of finitely many smooth, simple closed curves, no two of which meet tangentially,
and if 
 $K\subset \Reg M(t)$
for all $t\in [0,\Tpos)$, then
for every regular time $t<\Tpos$,
\[  \genus(M(t)\setminus K) 
\le \genus(M\setminus K).
\]
\end{proof}

\begin{theorem}\label{preservation-theorem}
Suppose $M\in \MM_g$.
Then $M(t)\in \MM_g$ for every regular time $t<\Tpos(M)$.
\end{theorem}

\begin{proof}
This is an immediate consequence
of Lemmas~\ref{intersection-lemma}--\ref{genus-lemma}, and the fact that $\area(M(t))$ is  
a decreasing function of $t$. 
\end{proof}

\begin{theorem}
Suppose $M\in \MM_g$.
\label{graph-theorem}
If
 $\Tpos=\infty$, then
$M(t)$ converges smoothly to $\{z=0\}$ with
multiplicity~$1$.  In particular, there
is a time $T<\infty$ such that 
$M(t)$ is a smooth graph (i.e., such that
$M(t)$ projects diffeomorphically onto $\{z=0\}$)
for all $t\ge T$.
\end{theorem}

\begin{proof}
Since the area of $M(t)$ is a positive, decreasing function of $t$, it has a limit $A$ as $t\to\infty$.
Let $t_n\to\infty$.
After passing to a subsequence,
the flows
\[
 t\in [-t_n,\infty)\mapsto
   M_n(t):= M(t+t_n)
\]
converge to an integral
Brakke flow $t\in \RR\mapsto M'(t)$.
Note that the area of $M'(t)$ is $A$ for every~$t$.
Since the area is independent of $t$,
the flow is static: 
there is a stationary integral varifold
$V$ such that $M'(t)$ is the Radon measure   associated to $V$ for $t$.

Now $V$ is contained in $\Cone'(\Lambda)$
by~\eqref{Lambda}
and Lemma~\ref{intersection-lemma},  so the support of $V$ is compact. In particular,  $z(\cdot)$ attains a maximum
$z_\textnormal{max}$ on the support of $V$.  It follows from the maximum principle 
  (\cite{sw}, for example)
that $\spt(V)$ contains all
 of $\{z=z_\textnormal{max}\}$.
Since $V$ is contained in $\Cone'(\Lambda)$,
it follows that $z_\textnormal{max}=0$.
Likewise, the minimum of $z(\cdot)$
 on $\spt(V)$ is $0$.
  Thus $V$ is $\{z=0\}$
with some integer multiplicity $k$.
Since $M_n(t)$ contains $\tQ$ for all $t$,
we see that $\tQ$ is in the support of $V$,
and thus that 
\[
  k\ge 1.
\]
Also,
\[
k\area(\{z=0\}) 
 =     \area(V) 
 \le    \area(M)
 <     2\area(\{z=0\}).
\]
Thus $k=1$.

We have shown that
the flows $M_n(\cdot)$ converge to the constant
flow given by $\{z=0\}$ with multiplicity~$1$.
By the Local Regularity Theorem
 (Theorem~\ref{local-theorem}), the convergence is smooth
with multiplicity $1$.
\end{proof}

\begin{theorem}\label{tQ-regularity}
 Suppose $M\in \MM_g$.
Let $p$ be a point in $\tQ \setminus \{O, O_\infty\}$.
Then for all $T\in (0,\Tpos(M)]$,
$(p,T)$ is a regular point of the flow $M(\cdot)$.
\end{theorem}

\begin{proof}
By the $\rho_C$ symmetry, we can assume
that $0<|p|\le 2$.
By Lemma~\ref{C-regular-lemma}, the points of $C$ are regular points, so we
can assume that $0<|p|<2$.  
By symmetry, we can assume that $p=(r,0,0)$.
Thus $0<r<2$.
Let $S$ be the points
of $\SS(r)\cap \tQ$, together with their images
under $\rho_C$.  Then $S$ is a set of $4(g+1)$ points,
all of the same genus~$\tilde g$.
Thus
\[
  4(g+1) \tilde g \le \genus(M) \le 4g,
\]
so $\tilde g=0$.  Thus if $(p,T)$ were a singular point,
the shrinker would be sphere or a cylinder, which
is impossible since $p\in M(t)$ for all $t\in [0,T]$.
Thus $(p,T)$ is a regular point.
\end{proof}

\begin{lemma}\label{origin-lemma}
The origin is a regular point of the flow 
 $M(\cdot)$
at all times $<\Tpos$.
If $\Tpos<\infty$, then the origin
is singular at time $\Tpos$,
and if $\Sigma$ is a shrinker at $(O,\Tpos)$, then
$\Sigma$ has genus~$>0$.
\end{lemma}

\begin{proof}
Let $0< T < \Tpos$
and let $\Sigma$ be a shrinker
at the spacetime point $(0,T)$.
Then $\Sigma$ is a plane, a sphere, or 
a cylinder.
Since $\Sigma$ contains $Q$, it must be
a plane, and thus $(O,T)$ is a regular point.

Now suppose $T=\Tpos<\infty$.
Then there is a spacetime point $(p,T)$
of positive genus.
Now $p$ must lie in $\{z=0\}$,
since
\[
  \genus(M(t)\setminus \{z= 0\}) = 0
\]
for all regular times $t\in [0,\Tpos]$.
Thus $p\in \tQ$.  
By Theorem~\ref{tQ-regularity}, $p$ is either the origin or the
point $O_\infty$.
By symmetry, we can assume it is the origin.
\end{proof}

\begin{theorem}\label{shrinker-good-theorem}
 Suppose that $M\in \MM_g$ and that $\Tpos<\infty$.
Let $\Sigma$ be a shrinker at
the origin  at time $\Tpos$.
Then $\Sigma\cap\{z=0\}=Q$.
Furthermore,
\[
\genus(\Sigma)\le \genus(M)/2,
\]
and  $\Sigma$ has genus~$g$ or genus~$2g$.
\end{theorem}

\begin{proof}
Let $t_i\uparrow T=\Tpos$ be such that
\[
 S_i:=  (T-t_i)^{-1/2}M(t_i)
\]
converges smoothly to $\Sigma$.
Since $\tQ\subset M(t)$ for each $t$,
 $Q\subset \Sigma$.
Since $(O,T)$ is not a regular point
 (by Lemma~\ref{origin-lemma}),
$\Sigma$ is not a plane.  Thus, as $\Sigma$ 
and $\{z=0\}$ are both minimal with
respect to the shrinker metric,
 $\Sigma\cap \{z=0\}$ is a network of curves,
 the vertices of which are the points of tangency.
 Hence if $\Sigma\cap\{z=0\}$ contained
 a point not in $Q$, it would also contain
 such point $q$ where $\Sigma$ and $\{z=0\}$
 intersect transversely.
But then $q$ would be a limit of points
where $S_i$ and $\{z=0\}$ intersect, which
is impossible since $M(t_i)\cap \{z=0\}$
is $\tQ$.

Because the singularities at $O$ and $O_\infty$
at time $T$ have the same genus, namely, $\genus(\Sigma)$,
\[
  2\genus(\Sigma)\le \genus(M).
\]
Finally, $\Sigma$ has genus~$g$ or~$2g$ by 
  Lemma~\ref{topology-lemma}.
\end{proof}

\section{Shrinkers with one end}
\label{desing-section}

In this section, we let $Q_g$
be the set of $g+1$ horizontal lines through the
origin defined in \S\ref{topology-section},
and we let $G_g$ be the group of isometries
of $\RR^3$ defined in \S\ref{topology-section}.
We let $d_k$ denote the entropy of a round $k$-sphere, which
is the same as the entropy of $\SS^k\times \RR^j$ for any~$j$.
Recall that
\[
   2 > d_1 > d_2.
\]

\begin{figure}
    \centering
    \includegraphics[width=0.75\linewidth]{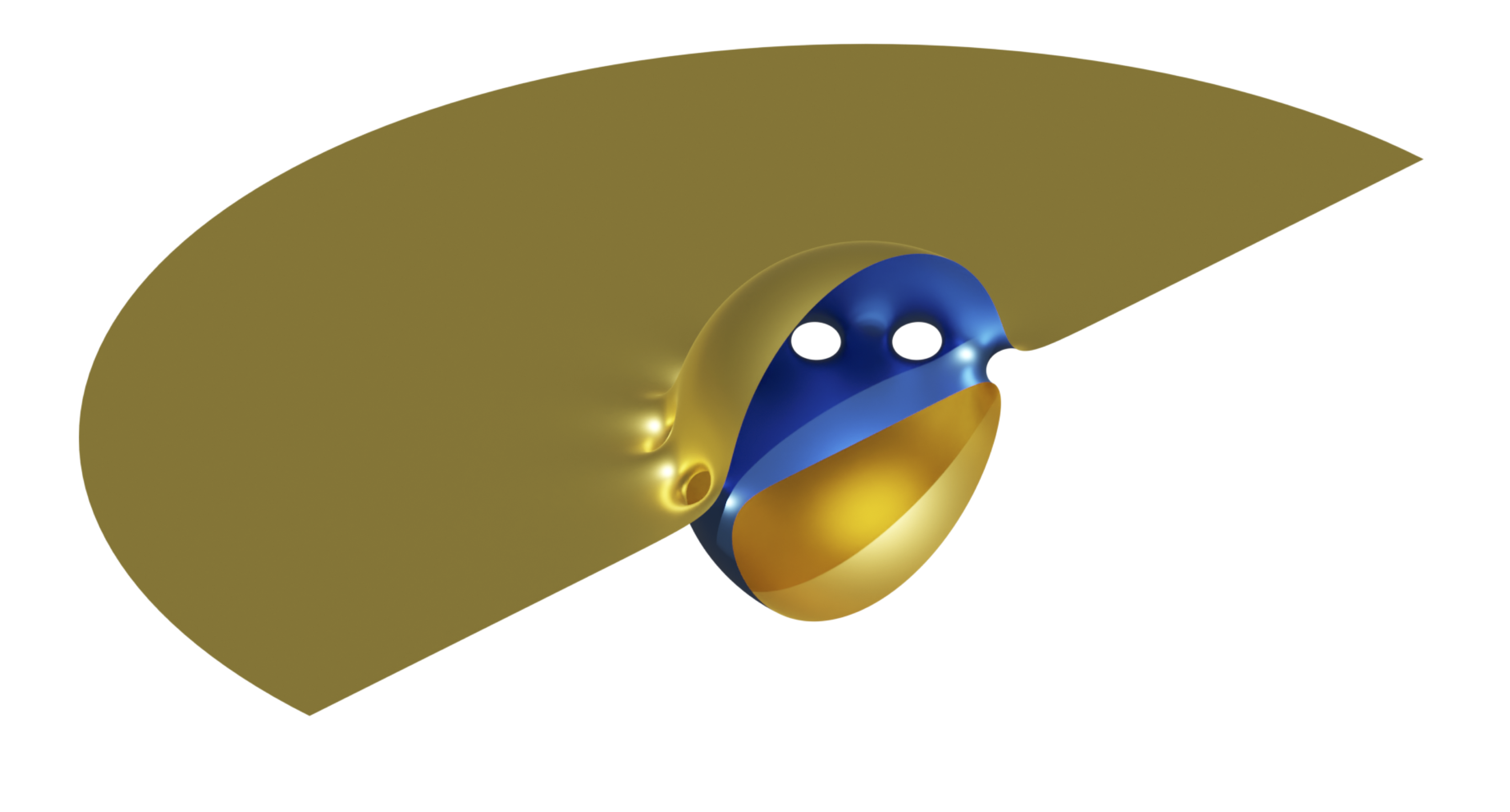}
    \caption{A numerical simulation of the shrinker $\Sigma_5$ in Theorem \ref{desing-theorem}. 
    Figure courtesy of Mario Schulz.}
    \label{fig:one-ended}
\end{figure}

\begin{theorem}\label{desing-theorem}
Suppose that $g\ge 1$ is an integer
and that $1+d_2< \delta < 1+d_1$.
There is a compact, smoothly embedded
surface in $\SS^2\times \RR$
that, under mean curvature flow,
has a singularity at which
every shrinker $\Sigma_g$ has
the following properties:
\begin{enumerate}
\item $\Sigma_g$ is $G_g$-invariant and has 
   genus~$g$.
\item $\Sigma_g\cap\{z=0\}=Q_g$. 
\item $\Sigma_g\setminus\{z= 0\}$ has genus~$0$.
\item $\Sigma_g$ has entropy $\le \delta$.
\end{enumerate}
Furthermore, for large $g$, $\Sigma_g$ has exactly one end, and, as $g\to\infty$, $\Sigma_g$ converges to the union
of the plane $\{z=0\}$ and the sphere $\partial B(0,2)$.  The convergence is smooth with multiplicity $1$ away from the intersection 
 $\{z=0\}\cap\partial B(0,2)$.
\end{theorem}

See Figure~\ref{fig:one-ended}.

Existence of shrinkers with properties
(1)--(4) was proved (by surgery),
for all sufficiently large $g$, by
Kapouleas, Kleene, and Moeller~\cite{kkm}, and, independently, by X. H. Nguyen~\cite{nguyen}.
It was subsquently proved for all $g$ 
by Buzano, Huy The Nguyen, and Schulz
using minimax~\cite{buzano}.
Those papers also proved that,
as $g\to\infty$, 
 the shrinkers 
 have the behavior 
stated in Theorem~\ref{desing-theorem}.
However, none of those papers proved that any
of the shrinkers arise
as  blowups of the mean curvature flow
of a compact, smoothly embedded initial 
surface.

\begin{figure}[htpb]
    \centering
    \includegraphics[width=0.75\linewidth]{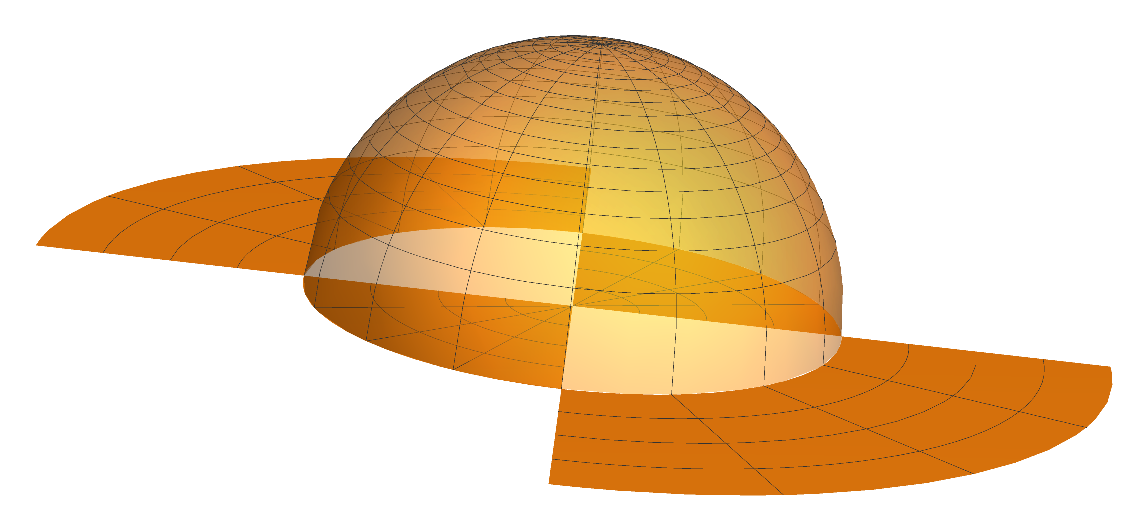}
    \caption{The piecewise smooth surface 
    $$\Delta_R\cup (\{z>0\}\cap \partial B(0,R))$$ in the case $g=1$. 
    Its boundary consists of a pair of diameters,  and a
    pair of circular arcs.}
    \label{fig:gorra}
\end{figure}

\begin{proof}[Proof of Theorem~\ref{desing-theorem}]
As in \S\ref{surfaces-section}, we work in 
\[
 N = \RR^3\cup \ZZ_\infty \cong \SS^2\times\RR,
\]
we let
\[
  \tQ = Q\cup C\cup\{O_\infty\},
\]
and we let $\tG=\tG_g$
be the group of isometries defined
in~\S\ref{surfaces-section}.  
Recall that $C$ is the circle
\[
  C = \{(x,y,0):(x^2+y^2)^{1/2}=2\}.
\]

For $0\le R < 2$, let 
\begin{align*}
\Delta_R
&=
\{ (r\cos\theta,r \sin\theta,0):
  \text{$0\le r\le R$
  and $\sin ((g+1)\theta ) \ge 0$} \}
  \\
&\qquad  \cup
\{ (r\cos\theta,r \sin\theta,0):
  \text{$R\le r\le 2$ and
    $\sin ((g+1)\theta ) \le 0$} \}.  
\end{align*}
See Figure~\ref{fig:gorra}.
Note that $\Delta_0$ is the union of $g+1$ wedge-shaped
regions in the unit disk of the $xy$-plane.
Thus $\partial \Delta_0$ consists of $g+1$ diameters
of the unit disk together with $g+1$ arcs in the unit circle.
Note that
 the surface
\begin{equation}\label{the-surface}
   \Delta_R \cup (\{z>0\}\cap \partial B(0,R))
   \tag{*}
\end{equation}
is a piecewise smooth surface whose boundary is equal
to $\partial \Delta_0$.
(Here, $B(0,R)$ denotes the Euclidean ball of radius $R$ about the origin, not the geodesic
ball with respect to the $\SS^2\times\RR$ metric.)
Perturb the surface~\eqref{the-surface} to get a surface $S^+$
such that $\partial S^+ =\partial \Delta_0$ and such that
 $
S^+\setminus \partial S^+
$
is a smooth, simply connected surface in the upper
halfspace $\{z>0\}$.

Now reflect $S^+$ about the $x$-axis to get $S^-$.
Then $S:= S^+\cup S^-$ will be a $2$-manifold whose boundary is
the circle 
 $C$. 

Now let 
\[
  M = S\cup \rho_C S.
\]
Note that we can choose $S^+$ so that $M$
is smooth and $\tG$-invariant.

Note that $S$ can be obtained by taking the  union of the disk~$D$ of 
 Euclidean radius 2 in the plane $\{z=0\}$
and the sphere 
     $\partial B(0,R)$
and then doing surgery along the circle of intersection.
The Euler characteristic of the disjoint union 
is $1+2=3$.
The surgery does not change the number of edges (in a suitable triangulation); it merely changes the way that faces are attached to edges.  But the  surgery does reduce the number of vertices by $2(g+1)$, since $2(g+1)$ vertices on the equator
of the sphere get identified with the corresponding
points in $D$. Thus
\[
  \chi(S) = 3-2(g+1) = 1 - 2g.
\]
On the other hand, since $S$ is connected and has one
boundary component, $\chi(S)=1 - 2\genus(S)$.
Thus $\genus(S)=g$, and so (since the simple
closed curve $C$ disconnects $M$),
\begin{align*}
    \genus(M)
    &=
    \genus(S)\cup \genus(\rho_C S)
    \\
    &=
    2\genus(S)
    \\
    &=
    2g.
\end{align*}

Note that, by choosing $R$ small, we can
make the area of $M$ arbitrarily close
to the area of $\{z=0\}$ (i.e., 
 the area of $\SS^2\times\{0\}$, or $4\pi$).
 In particular, we can
make
\[
  \area(M) < 2\area(\{z=0\}).
\]

To summarize, the  surface $M$ we constructed has the following
properties:
\begin{enumerate}
\item\label{first-MM-prop} $M$ is compact, connected, smoothly  
   embedded,
       and $\tG_g$ invariant.
\item $M\cap\{z=0\}=\tQ_g$.
\item $M$ has genus~$2g$.
\item $M\setminus \{z = 0\}$ has genus~$0$.
\item $M$ intersects $\{x^2+y^2=2\}$
transversely, and the intersection is $C$.
\item\label{last-MM-prop} $\area(M)< 2\area(\{z=0\})$.
\end{enumerate}
We let $\MM_g'$ be the class of surfaces
having properties~\eqref{first-MM-prop}--\eqref{last-MM-prop}.
Thus $\MM_g'$ is the set of $M\in \MM_g$ 
such that $M$ is connected and has genus~$2g$, 
where $\MM_g$ is the class of surfaces
specified in Definition~\ref{MM-definition}.
In particular, all the theorems
in \S\ref{surfaces-section} apply to the surfaces
in $\MM_g'$.

Now let $M$ be any surface in $\MM_g'$.
We orient $M$ so that the unit normal at $(2,0,0)$ is $\ee_3$.
It follows that the unit normal at the origin is
 $-\ee_3$:
\begin{equation}\label{normals-1}
\begin{aligned}
\nu_M(2,0,0)&=\ee_3, \\
\nu_M(0,0,0)&= - \ee_3.
\end{aligned}
\end{equation}

Let $M(t)$ denote the result of letting $M$ flow for time $t$ by the level set flow.   
Let $\Tpos=\Tpos(M)$ be the first time that there is  singularity of
  positive genus;
  if there is no such time, let
 $\Tpos=\infty$.

\setcounter{claim}{0}
\begin{claim}\label{normals-claim}
For all $t\in [0, \Tpos)$,
\begin{equation}\label{normals}
\begin{aligned}
\nu_{M(t)}(2,0,0)&=\ee_3, \\
\nu_{M(t)}(0,0,0)&= - \ee_3.
\end{aligned}
\end{equation}
\end{claim}

\begin{proof}
By Theorems~\ref{tQ-regularity}
and~\ref{origin-lemma},
 the flow is regular at all points of
$\tilde Q$ for all times $t\in [0,\Tpos)$.
In particular, the points $(0,0,0)$
and $(2,0,0)$ are such points.
At each of those two points, two or more great circles
in $\tQ$ intersect transversely.
Since $\tQ\subset M(t)$ for all $t\in [0,\Tpos)$,
we see that the tangent plane to $M(t)$ at
each of those two points is horizontal.
 Thus since~\eqref{normals} holds for $t=0$
 (by~\eqref{normals-1}), it holds
for all $t\in [0,\Tpos)$.
\end{proof}

Next, we claim that
\begin{equation}
    \Tpos<\infty.
\end{equation}
For, if not, then 
 (by   
 Theorem~\ref{graph-theorem})
there would be a $t<\infty$
for which $M(t)$ is a graph.
But that is impossible by 
 Claim~\ref{normals-claim}.

By Theorem~\ref{shrinker-good-theorem},
the flow has a singularity at the origin
at time $\Tpos$ of genus~$\tilde g$, 
where $\tilde g$ is either $g$ or $2g$
and 
\[
  \tilde g \le \genus(M)/2= g.
\]
Thus $\tilde g=g$.

Let $\Sigma_g$ be a shrinker at the spacetime
point $(0,\Tpos)$. 
 By Theorem~\ref{shrinker-good-theorem}, it has all the properties
asserted in Theorem~\ref{desing-theorem} except for the property
\[
 \entropy(\Sigma)< \delta,
\]
and the behavior of $\Sigma_g$ as $g\to\infty$.

By Theorem~\ref{desing-entropy-theorem} below,
we can choose $\Sigma_g$ to have entropy $<\delta$.
By Theorem~\ref{desing-limit-theorem}, 
given that entropy bound, $\Sigma_g$ has the behavior asserted in the statement of Theorem~\ref{desing-theorem}, as $g\to\infty$.
\end{proof}

\begin{theorem}\label{desing-entropy-theorem}
In the proof of
 Theorem~\ref{desing-theorem}, 
  we can choose the initial surface $M$ so that $\entropy(\Sigma)< \delta$.
\end{theorem}

\begin{proof}
Let $\MM'_g$ be as in the proof of 
 Theorem~\ref{desing-theorem}.

For $0<r<1$, let $\Psi(r)$ be the surface
consisting of $P:=\{z=0\}$, $\SS(r)$, and the image
of $\SS(r)$ under $\rho_C$, where recall that $\SS(r)$
is the sphere 
\(
  \{ p\in \RR^3: |p| = r\}.
\)

Note that there exist $M\in \MM'_g$ that are
arbitrarily close to $\Psi(r)$ in the weak sense.
That is, there exists a sequence $M_i\in \MM'_g$
such that $M_i$ converges weakly to $\Psi(r)$.

Consider $P$ and $\SS(r)$,
both with multiplicity~$1$.
Now let $r_i\to 0$.  Then we can choose $M_i\in \MM'_g$
so that 
\begin{enumerate}
    \item $M_i$ converges weakly with
     multiplicity $1$ to $\{z=0\}$
    \item If $\lambda_i\to\infty$ and if 
       $\lambda_ir_i\to r\in [0,\infty]$, 
       then
    the dilated surfaces
    $\lambda_i M_i$ 
    (in $N$, with the suitably dilated metrics~$\gamma_{\lambda_i}$)
    converge weakly
    (as Radon measures) to $P+\SS(r)$ in $\RR^3$.
\end{enumerate}
(Here $\SS(\infty)$ is the empty set.)

Let $T_i=\Tpos(M_i)$ and, as usual, let
\[
  t\in [0,T_i] \mapsto M_i(t)
\]
be the mean curvature flow with $M_i(0)=M_i$.
Let $\Sigma_i$ be a shrinker to the flow
at the origin at time $T_i$.

We claim that
\begin{equation}\label{Ti-to-zero}
    T_i\to 0.
\end{equation}
For suppose not. Then (after passing to a subsequence) $T_i$ converges to a limit $T\in (0,\infty]$.
Since $M_i$ converges weakly with multiplicity~$1$ to $\{z=0\}$, the flow
 $t\in [0,T_i]\mapsto M_i(t)$ converges
weakly to the multiplicity-one flow
\[
  t\in [0,T] \mapsto \{z=0\}.
\]
By local regularity
theory~\cite{white-local},
the convergence is smooth on $[\eps,T]$
for every $\eps>0$.
In particular, $M_i(\eps)$ is a smooth
graph for all sufficiently large~$i$.
But that contradicts 
  Claim~\ref{normals-claim}
in the proof of
  Theorem~\ref{desing-theorem},
 thus proving~\eqref{Ti-to-zero}.

For $t\in [0,1)$, let
\[
  M_i'(t) =  \frac1{\sqrt{T_i}} M_i( T_it).
\]
Thus
\[ 
 t\in [0,1) \mapsto M_i'(t)
\]
is a mean curvature flow in $\RR^3\cup Z_\infty$,
but with the dilated metric $\gamma_{\lambda_i}$,
 where $\lambda_i:=(T_i)^{-1/2}$.

Let $\area_s$ denote area with respect to the shrinker metric.
(The shrinker metric is the Euclidean metric on $\RR^3$
multiplied by the conformal factor
  $(4\pi)^{-1}\exp(-|p|^2/4)$.)
By monotonicity, 
\[
  \area_s(\Sigma_i) \le \area_s(M_i'(0)) + \eps_i,
\]
where $\eps_i\to 0$.
(The $\eps_i$ is there because $M_i'(\cdot)$
is mean curvature flow for a non-flat metric;
and $\eps_i\to 0$ because 
the metric is converging to smoothly to the flat
metric by~\eqref{Ti-to-zero}.)

By passing to a subsequence, we can assume
that the flows $M'_i(0)$ converge weakly
(in the sense of Radon measures)
to $P+\SS(r)$.
Thus
\begin{align*}
\area_s(\Sigma_i)
&\to
\area_s(P) + \area_s(S)
\\
&\le 1 + d_2
\\
&< \delta.
\end{align*}
Hence 
\[
  \entropy(\Sigma_i) = \area_s(\Sigma_i)< \delta
\]
for all sufficiently large $i$.
\end{proof}

\newcommand{\exc}{\operatorname{exc}}

\begin{theorem}\label{desing-limit-theorem}
    As $g\to\infty$, the $\Sigma_g$
    converge to the union of the plane $\{z=0\}$
    and the sphere $\SS(2)$.  The convergence
    is smooth with multiplicity $1$
    away from the circle
    $\{z=0\}\cap  \SS(2)$.
    For large~$g$, $\Sigma_g$ has exactly one end.
\end{theorem}

Theorem~\ref{desing-limit-theorem}
 was proved by Buzano, Nguyen, and 
 Schulz~\cite{buzano}*{Theorem~1.3}.
For the reader's convenience, we include
 a somewhat different proof here.

\begin{proof}
By Theorem~\ref{limits-theorem}, we can assume (after passing to a subsequence)
that the  $\Sigma_g$ converge to a limit $\Sigma$  consisting of the plane $\{z=0\}$
with multiplicity one together with one of the following:
\begin{enumerate}
\item\label{cylinder-case} The cylinder $\{x^2+y^2=2\}$ with multiplicity one, or
\item\label{sphere-case} The sphere $\SS(2)$ with multiplicity one.
\end{enumerate}
In the cylindrical case (Case~\ref{cylinder-case}), 
  $\Sigma$ would have entropy $1+d_1$,
which is impossible since 
\[
 \entropy(\Sigma_g) < \delta < 1 + d_1,
\]
by Theorem~\ref{desing-entropy-theorem}.

Thus Case~\ref{sphere-case} holds.  
Since the limit is independent of choice of subsequence,  the original sequence
converges.
By Theorem~\ref{limits-theorem}, the convergence is smooth away
from the circle $\SS(2)\cap\{z=0\}$.

Finally, for large $g$, $\Sigma_g$ has exactly one end by Lemma~\ref{ends-lemma}.
\end{proof}

\section{Shrinkers with three ends}
\label{new-examples-section}

Let $g\ge 1$ be an integer,
  let $Q_g$
be the set of $g+1$ horizontal lines through the
origin defined in \S\ref{topology-section},
and let $G_g$ be the group of isometries
of $\RR^3$ defined in \S\ref{topology-section}.
The goal of this section is to prove the
following theorem.

\begin{theorem}\label{horgan-theorem}
For all sufficiently large $g$, 
there is a compact, smoothly embedded
surface in $\SS^2\times\RR$
that, under mean curvature flow,
has a singularity at which
every shrinker $\Sigma_g$ has
the following properties:
\begin{enumerate}
\item $\Sigma_g$ has genus~$2g$ and $3$ ends.
\item $\Sigma$ is $G_g$-invariant and
 $\Sigma_g\cap\{z=0\}=Q_g$. 
\item $\Sigma_g\setminus\{z= 0\}$ has 
   genus~$0$.
\item The entropy of $\Sigma_g$ is bounded
above by a constant independent of~$g$.
\end{enumerate}
Furthermore, as $g$ tends to infinity, $\Sigma_g$ 
converges to the plane $\{z=0\}$ with multiplicity $3$.  After passing to a subsequence, the convergence is smooth
away from one circle.
\end{theorem}

See~Figure~\ref{fig:horgan}.

\begin{figure}[htpb]
    \centering
    \includegraphics[width=0.75\linewidth]{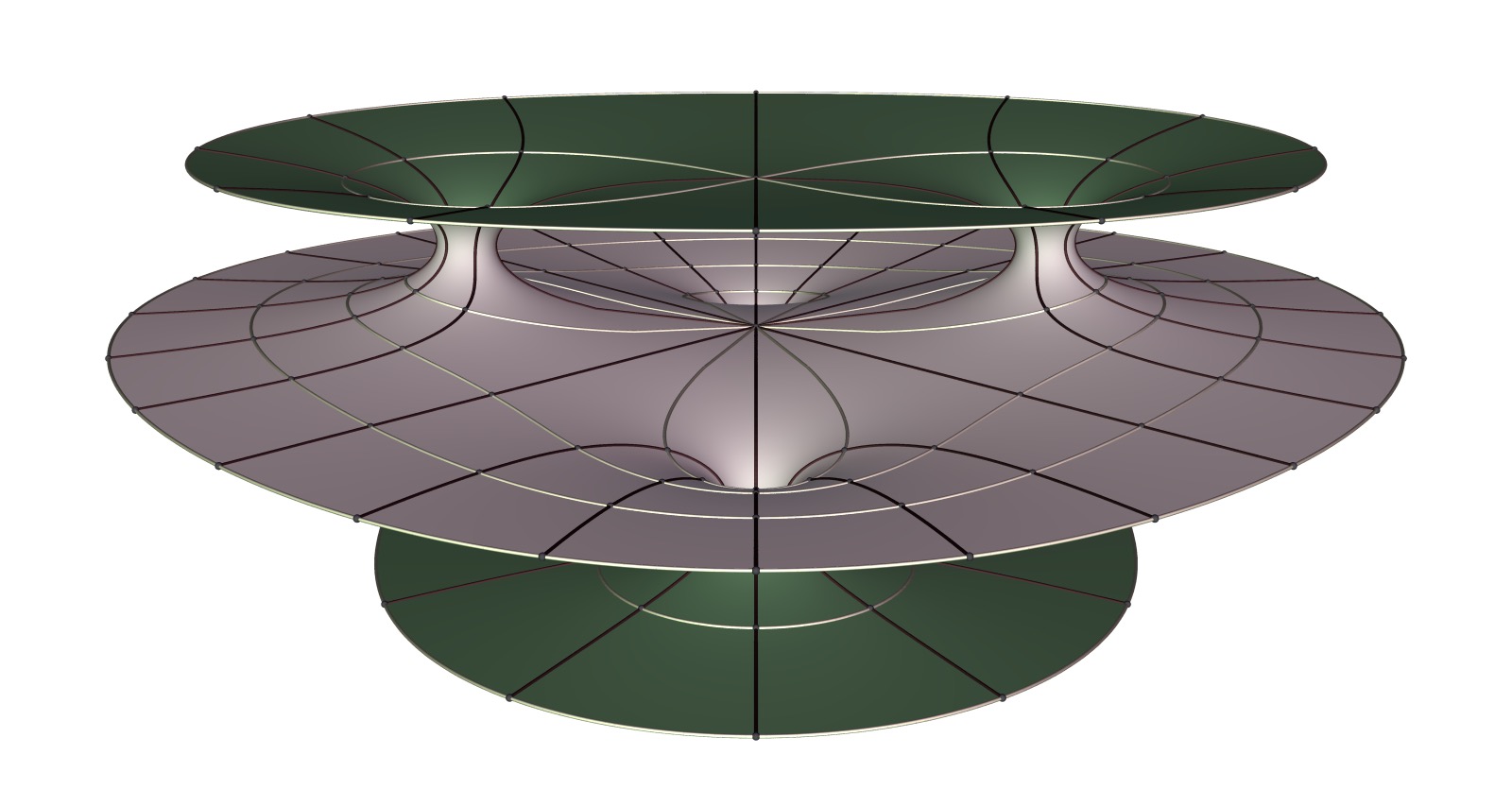}
    \caption{
Conjectured qualitative appearance of the shrinker $\Sigma_g$ 
 in case $g=1$. 
 The surface shown has genus~$2=2g$ and it 
 seems to have $3$ ends.
 However, we know that $\Sigma_g$ has genus~$2g$
 and $3$ ends only when $g$ is large.
 Figure courtesy of Matthias Weber.
 (This figure is not based on a numerical approximation of a shrinker.)}
 \label{fig:horgan}
\end{figure}

Theorem~\ref{horgan-theorem}
is a combination of the following  results:
\begin{enumerate}
    \item In Theorem~\ref{critical-M-theorem}, we prove, for each $g$, existence
    of a initial surface $M_g$ that has various
    nice properties, including a property
    called criticality.
    \item Let $\Sigma_g$ be a shrinker to $M_g(\cdot)$
    at the origin at time $\Tpos$.  
    In Theorem~\ref{critical-Sigma-theorem},
    we show that $\Sigma_g$ has most of
    the properties
     asserted in Theorem~\ref{horgan-theorem}.
    \item In Theorem~\ref{large-g-theorem}
      and Proposition~\ref{3-ends-proposition}, we show, for all sufficiently large~$g$,
    that $\Sigma_g$ has the remaining 
    properties asserted in Theorem~\ref{horgan-theorem}: it has genus~$2g$, 
    it has exactly $3$ 
    ends, and it
    converges as $g\to\infty$ to the plane $\{z=0\}$
    with multiplicity~$3$.
\end{enumerate}

To prove these theorems,
we work in 
\[
\RR^3\cup Z_\infty \cong \SS^2\times\RR
\]
with the metric $\gamma =\gamma_1$,
as defined at the beginning of Section~\ref{surfaces-section}.
We let
 $\MM_g''$ be the set of $M\in \MM_g$ such that
$M$ is a compact, connected surface of 
 genus~$4g$,
 where $\MM_g$ is as in
 Definition~\ref{MM-definition}.

\begin{proposition}\label{nonempty-proposition}
The class $\MM_g''$ is nonempty.
\end{proposition}

\begin{proof}
We prove Proposition~\ref{nonempty-proposition} by defining a certain one-parameter
family of surfaces in $\MM_g''$.  We will use the one-parameter family later.

First, we fix a small $R\in (0,2)$.
(How small will be specified later.
The choice of $R$ will not depend on $g$.)
\begin{figure}[htpb]
    \centering
    \includegraphics[width=0.85\linewidth]{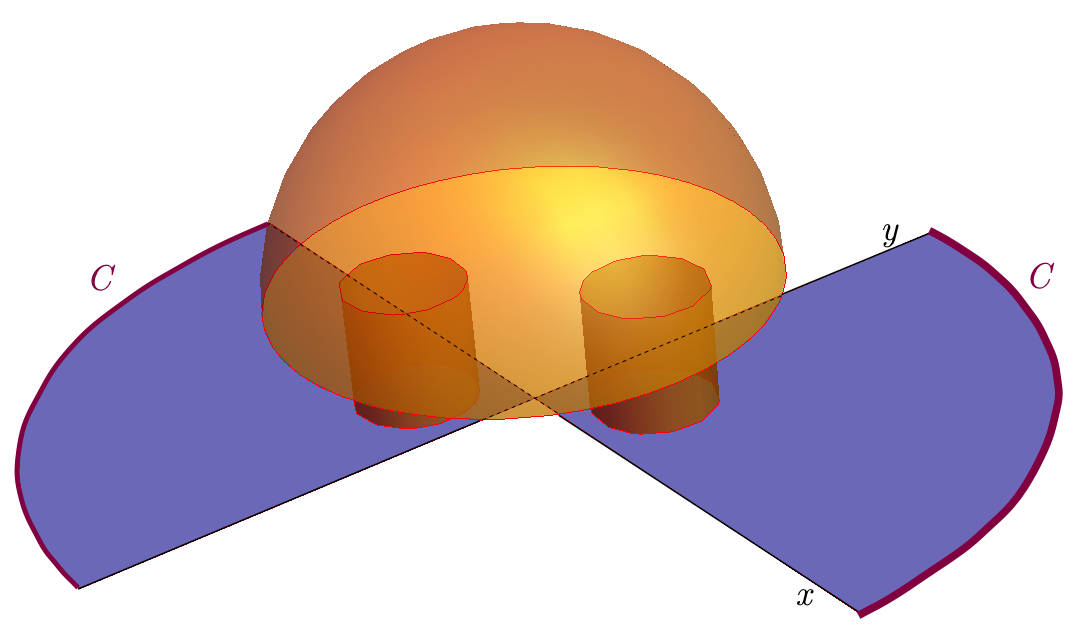}
    \caption{This figure presents a quarter of the surface \( M[s] \) for \( g = 1 \). The complete surface is generated by successively rotating this piece by \( 180^\circ \) around the \( x \)-axis and the geodesic \( C \). The resulting surface has genus~$4$.
    }
    \label{fig:horgan-3}
\end{figure}

Let
\begin{align*}
D &= \{(x,y,0): (x^2+y^2)^{1/2}\le 2\}, \\
D(R)&= \{ (x,y,0): (x^2+y^2)^{1/2} \le R\}, \\
A^+ &= \partial (B(0,R)\cap\{z>0\} ) , \\
A^- &= \partial (B(0,R)\cap\{z<0\} ).
\end{align*}
Thus $A^+$ (or $A^-$) consists of the upper 
(or lower) hemisphere of the
sphere $\SS(R)$ together with the disk $D(R)$.
Its area with respect to the Euclidean metric
 is $3\pi R^2$.

Let $k$ be an integer, and consider the 
 $2(g+1)$ wedge regions
\[
 W_k:=  \frac{\pi k}{g+1} 
   < \theta 
   < \frac{\pi(k+1)}{g+1},
\]
where
$0\le
k\le
2g+1$. 
Let $\Omega(s)$, $0\le s<1$, be a smooth, $1$-parameter
family of nested, closed convex regions in $W_0\cap D(R)$
such that $\Omega(0)$ is a single point, such that
$\cup_s \Omega(s)$ is all of $W_0\cap D(R)$,
and such that each~$\Omega(s)$ is symmetric under reflection in the plane~$P_{\pi/2(g+1)}$.

For $s\in [0,1)$, define $h(s)$ by
\begin{equation}\label{h(s)-formula}
  h(s) =  \frac{\dist(0,\Omega(s))^2}{(g+1)}.
\end{equation}

Translate $A^+$ upward by $h(s)$ 
and translate $A^-$ downward by $h(s)$ to get
\begin{align*}
  A^+_s &:= A^+ + (0,0,h(s)), \\
  A^-_s &:= A^- + (0,0,-h(s)).
\end{align*}

For $0<s<1$, we do surgery on $D\cup A^+_s\cup A^-_s$
as follows.  For each $W_k$ with $k$ even, we connect
$D$ to $A^+_s$ by a neck in the region $W_k$.
Likewise, for each $W_k$ with $k$ odd, we connected
$D$ to $A^-_s$ by a neck in $W_k$.  We do this in a 
  way that preserves the $G_g$ symmetry, so it
  suffices to describe the surgery in $W_0$.
The surgery in $W_0$ is the following:
 remove $\Omega(s)$ from $D$,
remove the corresponding region 
\[
   \Omega(s)+ (0,0,h(s))
\]
from $A^+_s$, and then attach the annular surface
\[
   \{(x,y,z): (x,y,0)\in \partial\Omega(s), 
   \, 0\le z \le h(s)\}.
\]

The resulting surface has $(g+1)$ necks
in $\{0\le z \le h(s)\}$ and $(g+1)$ necks
in $\{-h(s)\le z \le 0\}$.
Now let $M[s]$ be the union of that surface
and its image under rotation by $\pi$ about
  $C=\partial D$.
The resulting surface is $\tG$-invariant.

Let
\[
  \neck[s]:= (\partial\Omega(s))\times [0, h(s)]
\]
We have:
\begin{equation}\label{neck-bound}
\begin{aligned}
\area(\neck[s])
&=
h(s) \length(\partial\Omega(s))
\\
&\le
h(s) \length (\partial (D(R)\cap W_0))
\\
&=
h(s) (2R + (\pi R/(g+1))
\\
&\le h(s) (2\pi R)
\\
&\le (g+1)^{-1} \dist(0,\Omega(s))^2 2\pi R
\\
&\le (g+1)^{-1}2\pi R^3,
\end{aligned}
\end{equation}
where the second-to-last inequality follows from (\ref{h(s)-formula}).
The same bound holds for $\garea(\neck[s])$
since $\gamma_{ij}\le \delta_{ij}$ everywhere~\eqref{gamma-delta}.

Consequently, since $M[s]$ has $4(g+1)$ necks,
\begin{align*}
\garea(M[s])
&\le
\garea(\{z=0\})
+
4\garea(A^+_s)
+ 
4(g+1)\garea(\neck[s])
\\
&\le 
4\pi + 4\area(A^+_s)
+
4(g+1) (g+1)^{-1}2\pi R^3
\\
&=
4\pi + 12 \pi R^2
+ 
8\pi R^3.
\end{align*}
We fix an  $R$ for which this last expression is $<5\pi$.
Thus
\[
  \garea(M[s])< 5\pi
\]
for all $s\in (0,1)$.

Now the surface $M[s]$
is only piecewise smooth.
But we can smooth it (preserving the symmetries) so that the resulting
surface $M'[s]$ is smooth, and so that
\[
  M'[s] \cap \{z=0\} = \tQ.
\]
We can do the smoothing in such a way that
$M'[s]$ depends continuously on $s$,
and also so that
\[
\garea(M'[s])<5\pi.
\]

From the construction, it is straightforward to 
 check that $M'[s]$ has genus~$4g$
 and that $M'[s]\setminus\{z=0\}$ has genus~$0$.
Thus $M'[s]\in \MM_g''$.
\end{proof}

\begin{remark}\label{necks-remark}
Let $U$ be the vertical cylinder in $\SS^2\times \RR$ over the disk $D$. 
Note that as $s\to 0$, the surfaces $M[s]\cap U$
converge (as sets) to 
\begin{equation}\label{necks-collapsed}
   D \cup A_0^+\cup A_0^- \cup I,
\end{equation}
where $I$ is the union of $2(g+1)$ vertical line segments.
(The necks collapse to the segments $I$.)
Recall that $A_0^+$ is $A^+$ translated upward by $h(0)$, and that $h(0)>0$.  (See~\eqref{h(s)-formula}.)
We can do the smoothing in such a way that the 
 surfaces $M'[s]\cap U$
converge to the same limit set~\eqref{necks-collapsed}
as $s\to 0$.

Similarly, note that, as $s\to 1$,
$M[s]\cap U$ converges weakly
to the union of $D$ and $\SS(R)$,
each with multiplicity~$1$.
It follows that we can do the smoothing
in such a way that the same holds for $M'[s]$ as $s\to 1$.
\end{remark}

\begin{lemma}\label{density-lemma}
There is a constant $\Lambda<\infty$ 
(independent of $g$)
with the following
property:
\[
 \frac{\area(M[s]\cap B(0,r))}{\pi r^2}
 < 
 \Lambda
\]
for all $s\in(0,1)$ and all $0<r<2$.
\end{lemma}

\begin{proof}
Note that

\begin{align*}
\area(B(0,r)\cap (D\cup A^+_s \cup A^-_s)
&= 
\pi r^2 + 2\area(B(0,r)\cap A^+_s) 
\\
&\le
\pi r^2 + 2\area(B(0,r)\cap A^+)
\\
&=
\begin{cases}
    3\pi r^2  &\text{if $r<R$} \\
    \pi r^2 + 6\pi R^2 &\text{if $r>R$}.
\end{cases}
\\
&\le 7\pi r^2.
\end{align*}
Thus, if $\Rr$ is the union of the necks we attached
in the surgery to produce $M[s]$, then
\begin{equation}\label{area-bound}
\area(M[s]\cap B(0,r))
<
7 \pi r^2 
+
\area(B(0,r)\cap \Rr)).
\end{equation}
Now we need to bound the area of $B(0,r)\cap\Rr$.
We may assume 
\begin{equation}\label{r-big}
  r > \dist(0,\Omega(s)),
\end{equation}
as otherwise $B(0,r)\cap \Rr$ is empty.
Since $B(0,r)$ touches at most $2(g+1)$ necks, we have
\begin{align*}
\area(B(0,r)\cap \Rr)
&=
2(g+1) \area(B(0,r)\cap \neck[s])
\\
&\le
2(g+1) \area(\neck[s])
\\
&\le
(\dist(0,\partial \Omega(s))^2 4\pi R
\\
&\le
4\pi R r^2
\end{align*}
by~\eqref{neck-bound} and~\eqref{r-big}.
Thus, by~\eqref{area-bound},
\[
\area(M[s]\cap B(0,r))
\le
(7 + 4R) \pi r^2,
\]
so we have proved Lemma~\ref{density-lemma} with 
 $\Lambda=7+4R$.
\end{proof}

\begin{remark}\label{density-remark}
Recall that the piecewise smooth surface $M[s]$
was smoothed to make the
smooth surface $M'[s]$.  
It follows from Lemma~\ref{density-lemma}
that we can do the smoothing in such a way 
that $M'[s]$ satisfies
\[
   \frac{\area(B(0,r)\cap M'[s])}{\pi r^2} 
   <
   \Lambda 
\]
for all $r\in (0,2)$.
\end{remark}

\begin{definition}\label{Ff-definition}
Let $\Ff=\Ff_g$ be the family of 
smooth surfaces $M'[s]$ constructed in
the proof of Proposition~\ref{nonempty-proposition}, where the smoothing has been
done 
in accord with Remarks~\ref{necks-remark}
 and~\ref{density-remark}.
\end{definition}

\begin{theorem}\label{parity-theorem}
Suppose that $M\in \MM_g''$.
If $t<\Tpos(M)$  
is a regular time,
then 
  $M(t)\in \MM_g$, 
 and the number of points
 in $M(t)\cap Z^+$ is even.
Furthermore, if we give $M$ the orientation
such that $\nu_M(2,0,0)=\ee_3$, then
\[
\nu_{M(t)}(0,0,0) =  \ee_3
\]
for all $t\in [0,\Tpos(M))$, where the orientation
is the standard induced orientation
 as defined in Section~\ref{orientation-section}.
\end{theorem}

\begin{proof}
By Theorem~\ref{preservation-theorem}, $M(t)\in \MM_g$.

Give $M$ and $M(\cdot)$ the indicated orientations. 
 By Lemma~\ref{C-regular-lemma}
  and Lemma~\ref{origin-lemma}, $(0,0,0)$ and $(2,0,0)$ are    
regular points at all times $t< \Tpos$.
Since $\tQ\subset M(t)$  
 (by Lemma~\ref{intersection-lemma}), 
 we see that
the tangent plane at each of those points is
horizontal.  Thus
\begin{align}
&\nu_{M(t)}(2,0,0)=\nu_M(2,0,0)=\ee_3, \\
&\nu_{M(t)}(0,0,0)= \nu_M(0,0,0)
\label{nu-at-origin}
\end{align}
for all $t\in [0,\Tpos)$.

Let $t\in [0,\Tpos)$ be a regular time,
let $n(t)$ be the number of points
of $M(t)\cap Z^+$,
and let $K(t)$ be the closed region bounded
by $M(t)$ such that $\nu_{M(t)}(\cdot)$
is the unit normal that points out of $K(t)$.

Since $(2,0,z)\in M(t)$ only for $z=0$,
we see that
\[
   (2,0,z)\in K(t) \iff z\le 0.
\]
Thus if $p$ is any point for which
\[
   z(p) > \max_{q\in M(t)} z(\cdot),
\]
then $p\notin K(t)$.
Consequently, at the highest point
of $M(t)\cap Z$, $\nu_{M(t)}=\ee_3$,
and the next highest point,
$\nu_{M(t)}= - \ee_3$, and so on.
Thus
\begin{equation}\label{even-odd}
\nu_{M(t)}(0,0,0)
=
\begin{cases}
    \ee_3 &\text{if $n(t)$ is even,} \\
    -\ee_3 &\text{if $n(t)$ is odd.}
\end{cases}
\end{equation}
Thus, by~\eqref{nu-at-origin}, the parity of $n(t)$
is the same for all regular times
 $t\in [0,\Tpos)$.

Now at time $0$, 
\begin{align*}
4g
&=
\genus(M)  \\
&=
2\genus(M\cap \{(x^2+y^2)^{1/2} < 2\})
\\
&=
2 n(0)
\end{align*}
by Lemma~\ref{topology-lemma}.
Thus, for all regular times
 $t\in [0,\Tpos)$, $n(t)$ is even and
 so $\nu_{M(t)}(0,0,0)=\ee_3$ 
 by~\eqref{even-odd}.
 \end{proof}

\newcommand{\Uu}{\mathcal{U}}

\begin{definition}
We let
\[
  \Uu_g = \{M\in \MM_g'': \Tpos(M) = \infty\}.
\]
\end{definition}

\begin{lemma}
The set $\Uu_g$
is an open subset of $\MM_g''$
with respect to smooth convergence.
\end{lemma}

\begin{proof}
Suppose that $M\in \Uu_g$
and that $M_n\in \MM_g''$ converges
smoothly to $M$.
Let $T_n=\Tpos(M_n)$.  We must show
that $T_n=\infty$ for all sufficiently large $n$.

First, note that $T_n\to\infty$.
For if not, we can assume (after passing to a subsequence) that $T_n$ converges to a finite limit $T$.   
 By Lemma~\ref{origin-lemma}, the flow $M_n(\cdot)$ has
a singularity at the origin at time $T_n$.
Thus the flow $M(\cdot)$ also
has a singularity at the origin at time $T$
  (by the Local Regularity Theorem, 
  Theorem~\ref{local-theorem}).
 But that implies (by Lemma~\ref{origin-lemma})
that $\Tpos(M)=T$, a contradiction.
Thus $T_n\to \infty$.

By Theorem~\ref{graph-theorem}, there is a time $\tau<\infty$
such that $M(\tau)$ is a smooth graph.
Since $T_n\to\infty$, we can assume that $T_n>\tau$
for all~$n$.
By the Local Regularity 
 Theorem (Theorem~\ref{local-theorem}),
 $M_n(\tau)$ converges
smoothly to $M(\tau)$.  Thus $M_n(\tau)$ is a smooth
graph if $n$ is sufficiently large.
But smooth graphs never develop singularities,
and thus  $\Tpos(M_n)=\infty$
for such $n$.
\end{proof}

\begin{definition} \label{critical}
A surface $M\in \MM_g''$ is called 
 {\bf critical} if 
 \[
    M\in \overline{\Uu_g}\setminus \Uu_g,
 \]
 i.e., provided  $M\notin\Uu_g$ and there exist
 $M_n\in\Uu_g$ that converges smoothly to $M$.
\end{definition}

\begin{theorem}\label{critical-M-theorem}
Let $\Ff$ be the family of surfaces in $\MM_g''$
constructed in the proof
of Proposition~\ref{nonempty-proposition}.
Then $\Ff$ contains a critical surface $M$.
\end{theorem}

\begin{proof}
Note that $\Ff$ is a connected
family.
It suffices to show that 
 $\Ff\cap \Uu_g$ 
and 
  $\Ff\setminus \Uu_g$
 are both nonempty.
 For, in that case, since $\Ff\cap \Uu_g$
 is an open subset of $\Ff$, it cannot be 
 also be closed (by connectedness of $\Ff$).
Thus there is a surface $M$ in $\Ff\setminus \Uu_g$
that is a limit of surfaces in $\Ff\cap \Uu_g$.
By definition, such an $M$ is critical.

\setcounter{claim}{0}
\begin{claim}\label{desing-infinity-claim}
The family $\Ff\,\cap \,\Uu_g$ is nonempty:
 there exist $M\in \Ff$ for which $\Tpos(M)=\infty$.
\end{claim}

Let 
\begin{equation}\label{U-equation}
  U := \{(x,y,z): (x^2+y^2)^{1/2}< 2\}.
\end{equation}
To prove Claim~\ref{desing-infinity-claim},
consider a sequence $M_n = M'[s_n]$ in $\Ff$
for which $s_n\to 0$.
Thus 
 the $M_n\cap \overline{U}$ converge as sets to 
\[
    L:=D\cup A_0^+\cup A_0^- \cup I,
\]
where $I$ is the union
of $2(g+1)$ vertical line segments.  
(See Remark~\ref{necks-remark}.)
In particular,
there is an $\eps>0$ such that
\begin{equation}\label{slice}
   L \cap \{0<|z|< \eps\} = I.
\end{equation}
By the Initial Regularity Theorem
 (Theorem~\ref{initial-theorem}), there is a time $\tau>0$  such that for all sufficiently large $n$,
  the flow $M_n(\cdot)$ has no singularities on $Z$ during the time interval
  $[0,\tau]$.
In particular, for $t\in [0,\tau]$, $M_n(t)\cap Z$
consists of exactly $5$ points, 
 and those points depend
smoothly on $t$.
Thus, by replacing $\tau$ by a smaller $\tau>0$,
we can assume that
\begin{equation}\label{epsilon-over-two}
\text{
  $M_n(t)\cap Z \cap \{|z|\le  \eps/2\} 
   = \{O\}$
   for all $t\in [0,\tau]$.
   }
\end{equation}

Using Angenent Tori as barriers, or by Brakke's Clearing Out Lemma, it follows from~\eqref{slice} that
\begin{equation}\label{emptyset}
   M_n(t)\cap \{|z|=\eps/2\} = \emptyset
\end{equation}
for $t=\tau/2$ and for all sufficiently large $n$.
It follows (for such $n$)
 that~\eqref{emptyset}
holds for all $t\ge \tau/2$.

Let $t\in (\tau/2,\tau)$ be a regular time
for $M_n(\cdot)$.

Let $M_n^c(t)$ be the connected component 
of $M_n(t)$ that contains $\tQ$.
By~\eqref{emptyset}, $M_n^c(t)$ lies in $\{ |z|<\eps/2\}$
Thus, by~\eqref{epsilon-over-two}, $M^c_n(t)$ 
is disjoint from $Z^+$.  
Hence
\[
 \genus(M_n(t)) = 0
\]
by Proposition~\ref{topology-proposition}.
Therefore $\Tpos(M_n)=\infty$.
This completes the proof of
 Claim~\ref{desing-infinity-claim}.

\begin{claim}\label{TWO}
The family $\Ff\setminus\Uu_g$
is nonempty: there exist $M\in \Ff$
for which $\Tpos(M)<\infty$.
\end{claim}

To prove Claim~\ref{TWO},
let $M_n= M[s_n]$ be a sequence $\Ff$ 
for which $s_n\to 1$.
We claim that there are $n$ for which 
 $\Tpos(M_n)<\infty$.
Suppose, to the contrary, that 
 $\Tpos(M_n)=\infty$ for each $n$.
By passing to a subsequence, we can 
assume that the $M_n(\cdot)$ converge
weakly to a mean curvature flow
$t\in [0,\infty)\mapsto M(t)$.

By Remark~\ref{necks-remark}, the surfaces 
\[
   M_n\cap U,
\]
where $U$ is the open
cylinder~\eqref{U-equation},
converge (as Radon
measures) 
to the union of $D$ and $\SS(R)$,
each with multiplicity~$1$.
It follows, by 
 the Initial Regularity 
 Theorem (Theorem~\ref{initial-theorem}), 
 that
if $\tau>0$ is sufficiently small, then
there is an $n_\tau$ such that, for all
 $n\ge n_\tau$, 
\[
   M_n(\tau)\cap \{(x^2+y^2)^{1/2} < R/2  \}
\]
is the union of three smooth graphs.
Note we can choose $\tau$ to be a regular
time for all the $M_n(\tau)$.
Then for $n\ge n_\tau$,
\[
  M_n(\tau)\cap Z^+
\]
contains exactly one point, which
is impossible by Theorem~\ref{parity-theorem}.
This completes the proof of  Claim~\ref{TWO},
and therefore the proof
 of Theorem~\ref{critical-M-theorem}.
\end{proof}

{By Lemma~\ref{origin-lemma}, we know that at time $\Tpos(M)<\infty$, the flow 
 $M(t)$ has a singularity at the origin with positive genus.}

\begin{definition}
Let $\Ss_g$ be the collection of shrinkers $\Sigma$ that occur 
at the spacetime point
 $(O,\Tpos(M))$ for a surface $M\in \Ff_g$ that is critical in the sense of Definition~\ref{critical}.
\end{definition}

\begin{theorem}\label{critical-Sigma-theorem}
Suppose $\Sigma\in\Ss_g$.
Then 
\begin{enumerate}
\item $\Sigma$ is $G_g$-invariant,
\item $\Sigma\setminus \{z= 0\}$ has genus~$0$,
\item $\Sigma\cap\{z=0\} = Q_g$, 
\item $\genus(\Sigma)$ is $g$ or $2g$.
\item The entropy of $\Sigma$ is bounded above
  by a constant independent of $g$.
\end{enumerate}
\end{theorem}

\begin{proof}
Trivially, $\Sigma_g$ is $G_g$-invariant.
Note
  that $M(t)\setminus\{z=0\}$ has
genus~$0$ for all regular $t<\Tpos$
by Theorem~\ref{parity-theorem}.
Hence $\Sigma\setminus\{z= 0\}$ has genus~$0$.
By Theorem~\ref{shrinker-good-theorem},
 $\Sigma\cap\{z=0\}=Q_g$
 and $\Sigma$ has genus~$g$ or $2g$.
Finally, the uniform entropy bound 
follows from Remark~\ref{density-remark} and Huisken monotonicity.
\end{proof}

\begin{conjecture}
If $\Sigma$ is as in Theorem~\ref{critical-Sigma-theorem}, then
$\Sigma$ has genus~$2g$.
\end{conjecture}

The conjecture is true if $g$ is sufficiently
large by Theorem~\ref{large-g-theorem} below.

\begin{definition}
A {\bf special} Brakke flow 
is a unit-regular, integral
Brakke flow 
\[
   t\in I\mapsto M(t)
\]
that admits
a strong orientation $\nu$
with the following property:
for each $t$, if the origin is a regular 
point at time $t$, then
  $\nu(O,t)= \ee_3$.
\end{definition}

\begin{lemma}\label{special-lemma}
Suppose $\Sigma\in \Ss_g$.
Then there is a $G_g$-invariant,
special Brakke flow 
\[
  t\in \RR\mapsto M'(t)
\]
in $\RR^3$ such that
\[
   M'(t)= |t|^{-1/2}\Sigma
   \quad
   \text{for all $t< 0$}.
\]
\end{lemma}

\begin{proof}
Let $T=\Tpos(M)$.
Let $M\in \MM_g''$ be a critical surface
such that $\Sigma$ is a shrinker to $M(\cdot)$
at the spacetime point $(O,T)$.
Since $M$ is critical, there exist $M_n\in \MM_g''$
such that $\Tpos(M_n)=\infty$
and such that $M_n$ converges smoothly to $M$.

Note that $M_n(\cdot)$ is special
 by Theorem~\ref{parity-theorem}.

By passing to a subsequence, we can assume
that the flow $M_n(\cdot)$ converge to a 
unit-regular Brakke flow
\[
   t\in [0,\infty) \mapsto \tilde M(t).
\]
Now
\[
  \tilde M(0) = M,
\]
and thus 
\[
  \tilde M(t) = M(t)
  \quad\text{for $0\le t\le \Tfat(M)$}.
\]
By the Local Regularity 
 Theorem (Theorem~\ref{local-theorem}), $\tilde M(\cdot)$
is special since each $M_n(\cdot)$ is.

Now
\begin{equation}\label{the-tangent-flow}
  t\in (-\infty,0) \mapsto |t|^{-1/2}\Sigma
\end{equation}
occurs as a tangent flow 
to $M(\cdot)$ at the spacetime point $(0,T)$,
so it also occurs as a tangent flow
to $\tilde M(\cdot)$ at that point,
since the two flows coincide up to time $T$.
Thus there is a sequence $\lambda_n\to \infty$
such that the parabolically dilated flows
\begin{equation}\label{dilated-flows}
t \in [-\lambda_nT , \infty)
\mapsto
\lambda_n^{1/2} 
\tilde M(T + \lambda_n^{-1} t)
\end{equation}
converge for $t<0$ to the
 flow~\eqref{the-tangent-flow}.

By passing to a subsequence (if necessary)
we can assume that the 
 flows~\eqref{dilated-flows} converge for all $t$
  to a flow
\begin{equation}\label{extended-flow}
  t\in \RR\mapsto M'(t)
\end{equation}
Since the $\tilde M(\cdot)$ is special,
so are the flows~\eqref{dilated-flows},
 and therefore
(by the Local Regularity
 Theorem, Theorem~\ref{local-theorem}) so is the
 flow~\eqref{extended-flow}.
\end{proof}

\begin{theorem}\label{large-g-theorem}
Suppose, for each $g$, that $\Sigma_g\in \Ss_g$.
Then, for all sufficiently large $g$,
the genus of $\Sigma_g$ is $2g$.  Moreover,
$\Sigma_g$ converges as $g\to\infty$ to the plane $\{z=0\}$
with multiplicity $3$, and the convergence
is smooth except along a circle $C(r)$ 
in the plane $\{z=0\}$.
\end{theorem}

\begin{proof}
Suppose the theorem is false.
In Theorem~\ref{limits-theorem} in following section, we analyze all possible
subsequential limits as $g\to\infty$ 
of surfaces like the $\Sigma_g$
 in Theorem~\ref{critical-Sigma-theorem}.
According to that theorem, if the conclusion
of Theorem~\ref{large-g-theorem} is false, then (by passing
to a subsequence) we can assume that
the $\Sigma_g$ converge with multiplicity~$1$ to 
a limit
\[
  \Sigma= P + S
\]
where $P$ is the plane $\{z=0\}$
and $S$ is either the cylinder $\{x^2+y^2=2\}$
or the sphere
 $\SS(2):=\{x^2+y^2+z^2=4\}$.

By Lemma~\ref{special-lemma}, for each $g$, there is 
a $G_g$-invariant, special Brakke flow
\[
   t\in \RR \mapsto M_g'(t)
\]
such that 
\[
  M_g'(t)= |t|^{-1/2}\Sigma_g 
  \quad
  \text{for all $t<0$}.
\]
By passing to a further subsequence, we
can assume that the flows $M_g'(\cdot)$
converge to a flow
\[
  t\in \RR\mapsto M'(t).
\]
Note that 
\[
   M'(g)=|t|^{-1/2}\Sigma \quad
   (t \le 0). 
\]
As a set, $M'(0)$ is $P$ if $S$
 is the sphere and $P\cup Z$ if $S$ is
 the cylinder.  In either case, the corresponding Radon measure is the plane $P$ with multiplicity one.
It follows (by unit regularity) that $M'(t)$ is the plane $P$ with
 multiplicity one for all $t>0$.

That is, $M'(\cdot)$ is the flow
\begin{equation}\label{M'}
t\in \RR
\mapsto
\begin{cases}
|t|^{-1/2}\Sigma  &(t \le 0), \\
        P     &(t> 0).
\end{cases}
\end{equation}

Since $M'(\cdot)$ is a limit of special
flows, it is also special.

Note that at time $t=-1$,
the singular set of~\eqref{M'} is the circle
 $J:=S\cap P$.
Thus, by specialness, there is a closed set $K$
such that $K\setminus J$ is a smooth manifold-with-boundary.
At the origin, $\ee_3$ points out
of $K$.  It follows that $K$ consists of 
the closure of the 
region below $P$ and   inside $S$,
together with the closure of  region the 
region above $P$ and outside $S$.
See Fig.~\ref{fig:K}.
\begin{figure}[htpb]
    \centering
    \includegraphics[width=0.75\linewidth]{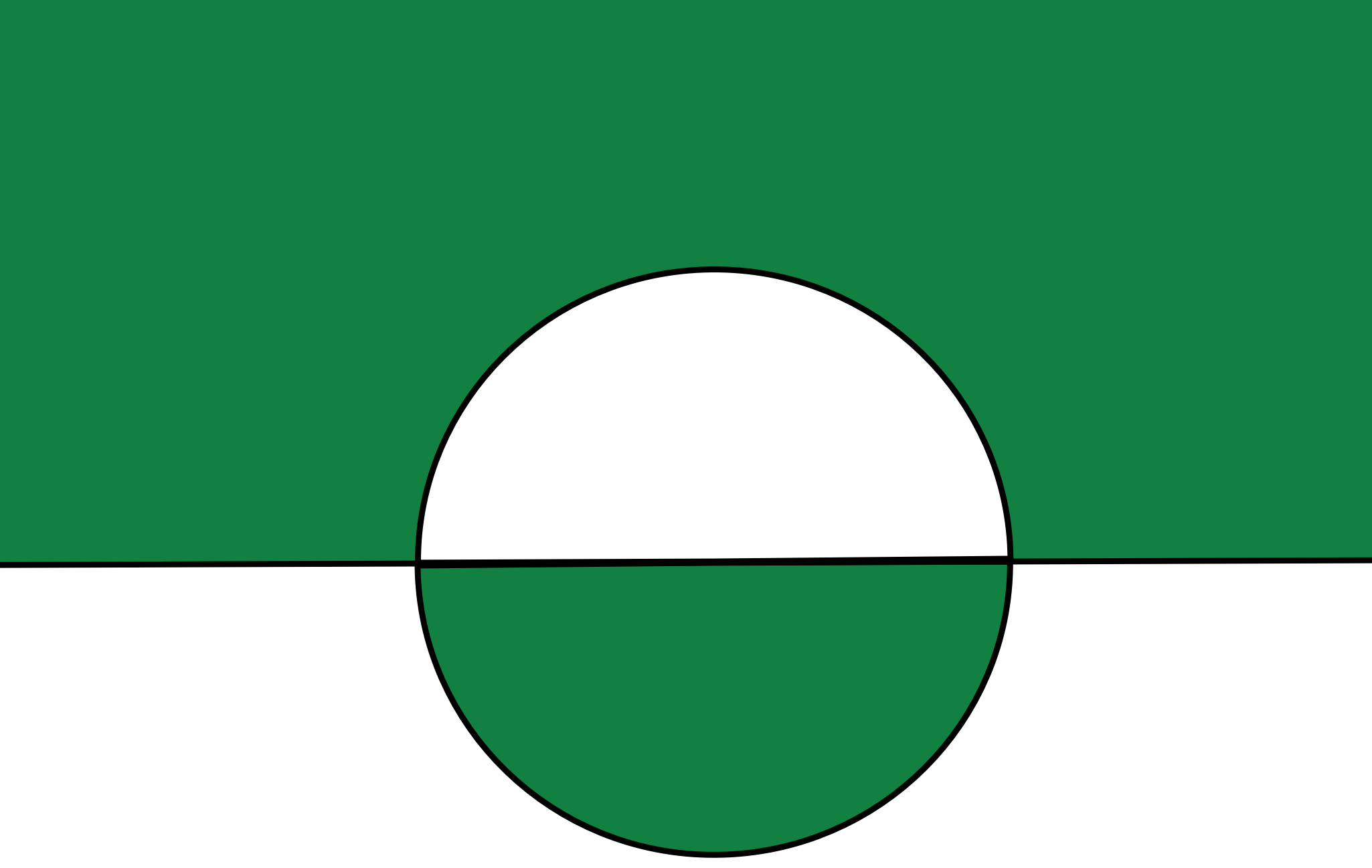}
    \caption{The region $K$
    (in the case that $S$ is a sphere).}
    \label{fig:K}
\end{figure}

Thus $\nu(p,-1)= - \ee_3$
at all points $p$ in $\{z=0\}$ that lie
outside $S$.

By the same argument (or simply by continuity)
\[
  \nu(p,t)= -\ee_3
\]
for all $t<0$ and $p$ in $\{z=0\}$
outside $|t|^{-1/2}S$.

Hence (by continuity) 
\[
   \nu(p,0)=\ee_3
\]
for all $p\ne 0$ in $\{z=0\}$.
Therefore, also by continuity,
\[
  \nu(p,t)= -\ee_3
\]
for all $p$ in $\{z=0\}$ and $t>0$.
In particular, $\nu(0,t)=-\ee_3$ for all $t>0$,
contrary to specialnesss of $M'(\cdot)$.
\end{proof}

\section{Limits of Shrinkers}\label{limits}

In this section, $\Gg$ is an infinite
set of positive integers, and for each $g\in \Gg$, $\Sigma_g$ is a connected, properly embedded, $G_g$-invariant shrinker such that
\begin{enumerate}
\item 
$\Sigma_g\cap \{z=0\}= Q_g$.
\item 
$0<\genus(\Sigma_g)\le 2g$.
\item 
$\genus(\Sigma\cap\{z>0\})=0$.
\item
$\sup_{g\in \Gg}\entropy(\Sigma_g) < \infty$.
\end{enumerate}
(The notation is as in Section~\ref{topology-section}.)
Examples of such shrinkers were produced 
 in~\S\ref{desing-section}
  and~\S\ref{new-examples-section}.
We will analyze the behavior of $\Sigma_g$
as $g\to\infty$.
In this section, ``passing to a subsequence"
means replacing the set $\Gg$ by an infinite subset.

\begin{lemma}\label{end-lemma}
Let $G_g^+$ be the subgroup of $G_g$ consisting of $\sigma$
such that $\sigma(0,0,z)=(0,0,z)$ for all $z$.
Let $\Gamma_g$ be a smooth,
$G^+_g$-invariant, simple closed curve
 in $\Sigma_g^+:=\Sigma_g\cap\{z>0\}$.
Then $\Gamma_g$ divides $\Sigma_g$ into two components.
Let $U_g$ be the  component that
does not contain $Q_g$.  Then 
\[
   E_g:= U_g\cup \Gamma_g
\]
lies in $\{z>0\}$ and is properly embedded in $\RR^3$.
\end{lemma}

\begin{proof}
The proof of~\cite{brendle}*{Theorem~2}
shows that $\Sigma_g^+$ is connected. 
Since $\Sigma_g^+$ has genus~$0$, $\Gamma_g$ divides $\Sigma_g^+$ into two connected
components $U_g$ and $U_g'$.
Now $Q=Q_g$ is in the closure of
$\Sigma_g\setminus Q$, and, therefore, by symmetry, in the closure of $\Sigma_g^+$.
Thus the closure of one of the components, say $\overline{U_g'}$, of 
   $\Sigma_g\setminus \Gamma_g$,
contains a point of $Q\setminus \{O\}$.
Hence $\overline{U_g'}$ contains the corresponding
connected component of $Q_g\setminus \{O\}$.
By $G_g^+$ symmetry, 
  $\overline{U_g'}$ contains all of
  $Q_g\setminus\{O\}$ 
and therefore all of $Q_g$.
It follows that $\Gamma_g$ divides
$\Sigma_g$ into two components, namely
$U_g$ and
\[
   U_g'':=U_g' \cup (\Sigma_g\cap \{z\le 0\}).
\]
Now let $E_g=U_g\cup\Gamma_g$.
Since $E_g$ contains $\Gamma_g$,
and is disjoint from $Q_g=\Sigma_g\cap\{z=0\}$, we see that $E_g$
lies in $\{z>0\}$.
As $\Sigma_g$ is properly embedded in $\RR^3$ and as $\Gamma_g$ is a curve that divides $\Sigma_g$ into two components,
one of which is $U_g$, then $E_g$ is properly embedded in $\RR^3.$
\end{proof}

\begin{theorem}\label{limits-theorem}
After passing to a subsequence, 
the $\Sigma_g$ converge to a shrinker $\Sigma$,
and 
the convergence is smooth away from a circle
\[
   C(r):= \{(x,y,0): x^2+y^2=r^2\}.
\]
Furthermore, one of the following holds:
\begin{enumerate}
\item\label{limit-plane-case} $\Sigma$ is the plane $\{z=0\}$ with  multiplicity~$3$, and $\Sigma_g$ has genus~$2g$ for
  all sufficiently large $g$.
\item $\Sigma$ is the plane $\{z=0\}$ together with the 
  sphere $\partial B(0,2)$, each with multiplicity $1$.
\item $\Sigma$ is the plane $\{z=0\}$ together with
  the cylinder $\{x^2+y^2=2\}$, each with multiplicity~$1$.
\end{enumerate}
\end{theorem}

\begin{proof}
By passing to a subsequence, we can assume that the $\Sigma_g$ converge as sets to a limit 
set $\Sigma$ that is rotationally invariant about~$Z$.   We can also assume that there is a closed set $K$ (the curvature blowup set) such
the convergence $\Sigma_g$ to $\Sigma$ is smooth
(perhaps with multiplicity) in compact
subsets of $\RR^3\setminus K$ and such that
\[
    \sup_{B(p,\eps)}|A(\Sigma_g,\cdot)| \to \infty
\]
for every $p\in K$ and every $\eps>0$, where $A(\Sigma_g,\cdot)$
is the second fundamental form of $\Sigma_g$.

Since $Q_g\subset \Sigma_g$, it follows that 
$\{z=0\}$ is contained in $\Sigma$, and thus
that
\begin{equation}\label{M-form}
\begin{aligned}
\Sigma &= \{z=0\} \cup \Sigma^+ \cup \Sigma^-
\\
&= \{z=0\} \cup \Sigma^+ \cup \rho_X \Sigma^+,
\end{aligned}
\end{equation}
where $\Sigma^+=\Sigma\cap\{z>0\}$,
$\Sigma^-=\Sigma\cap \{z<0\}$, and $\rho_X$
denotes rotation by $\pi$ about the $x$-axis, $X$.

By assumption, $\Sigma_g^+:=\Sigma_g\cap \{z>0\}$ has
genus~$0$.
By the compactness 
theorem~\cite{white-compact}*{Theorem~1.1} for 
minimal surfaces of locally bounded area and genus,
  $\Sigma^+$
is smooth
and  $K\cap \{z>0\}$ is a discrete subset
of $\{z>0\}$.
The proof of~\cite{brendle}*{Theorem~2}
shows that $\Sigma^+$ is unstable (as a minimal surface for the shrinker metric) and connected.  Thus (again by the compactness theorem~\cite{white-compact}*{Theorem~1.1}), the convergence 
of $\Sigma_g^+$ to $\Sigma^+$ is smooth and with multiplicity $1$, and therefore
$K\cap\{z>0\}$ is empty.  Therefore, by symmetry,
\begin{equation}\label{in-xy-plane}
   K\subset \{z=0\}.
\end{equation}

\setcounter{claim}{0}
\begin{claim}\label{top-claim}
One of the following holds:
\begin{enumerate}
\item $\Sigma^+$ is empty.
\item $\overline{\Sigma^+}$ is a disk bounded by a circle
  $C(r)$ in the plane $\{z=0\}$.
\item $\overline{\Sigma^+}$ is an annulus
    whose boundary is 
    a circle
   $C(r)$ in the plane $\{z=0\}$.
\end{enumerate}
\end{claim}

\begin{proof}[Proof of Claim~\ref{top-claim}]
Suppose $\Sigma^+$ is nonempty.
Note that
\begin{equation}\label{nonempty}
    \overline{\Sigma^+}\cap \{z=0\} 
    \ne
    \emptyset,
\end{equation}
as otherwise $\Sigma^+$ would be a smooth
 shrinker (without boundary) in $\{z>0\}$.
 Note also that $\overline{\Sigma^+}$ is a surface of revolution. 
Let $C'$ be a horizontal circle in $\Sigma^+$.
Let $C_g'$ be smooth, simple closed curves
in $\Sigma_g$ that converge smoothly to $C'$.
By Lemma~\ref{end-lemma}, $C_g'$ bounds a region $E_g$ in $\Sigma_g$
that is either a disk or an annulus, that lies in $\{z>0\}$,
and that is properly embedded in $\RR^3$. Passing to the limit,
we see that $C'$ bounds a region $E$ in $\Sigma^+$ that is either
a disk or an annulus, and that is properly embedded in $\RR^3$. 

By the strong maximum principle,
 $E$ lies in $\{z>0\}$.
Since $E$ is properly embedded in $\RR^3$,
\begin{equation}\label{empty}
   \overline{E}\cap\{z=0\}
   =
   \emptyset.
\end{equation}
Now consider the other portion $E'$ of $\Sigma^+$ bounded by $C'$.  
By~\eqref{nonempty} and~\eqref{empty},
\begin{equation}\label{nonempty2}
  \overline{E'}\cap\{z=0\}
  \ne
  \emptyset.
\end{equation}
 Note that $E'$ is a surface of revolution generated
 by a curve $e'$ in the plane $y=0$,
 and that $e'$ has one endpoint $p$ in $C'\cap\{y=0\}$.
 By~\eqref{nonempty2}, 
 the closure of $e'$
 contains a point $q$ in $\{z=0\}$.
Since $\Sigma$ has finite entropy, the
curve $e'$ has locally finite length.
Thus $e'$ is a curve joining $p$ to $q$.

Now $q$ cannot be the origin, since otherwise
 $\Sigma^+\cup\{0\}$ and $\{z=0\}$
would violate the strong maximum principle
at the origin.
(One can invoke various versions of the maximum
principle here.  For example, the classical 
strong maximum principle suffices, because
$\Sigma^+\cup\{0\}$
is a smooth minimal surface with respect 
to the shrinker metric 
  by \cite{structure}*{Corollary~2.7}.)

Thus $q=(x,0,0)$ for some $x\ne 0$. 
Therefore, $\overline{\Sigma^+}$ is a properly
embedded disk or annulus whose boundary is
the circle $C(r)$ (where $r=|x|$)
in the plane~$\{z=0\}$.
This completes the proof of Claim~\ref{top-claim}.
\end{proof}

Recall (Corollary~\ref{r-R-corollary}) that there exist $r_g\in (0,\infty)$
and $R_g\in [r_g,\infty]$ such that 
\[
   \genus(\Sigma_g
    \setminus (\SS(r_g) \cup \SS(R_g))
    =
    0,
\]
where  $\SS(r)$ is the sphere of Euclidean radius $r$
centered at the origin.
Passing to a subsequence, we can assume that $r_g$
and $R_g$ converge to limits $r_\infty$ and $R_\infty$
in $[0,\infty]$
with $r_\infty \le R_\infty$.

By the compactness 
theorem~\cite{white-compact}*{Theorem~1.1} for 
minimal surfaces of locally bounded area and genus,
\[
 K':=  K\setminus (\SS(r_\infty)\cup \SS(R_\infty))
\]
is a discrete subset of the complement of
 $\SS(r_\infty)\cup \SS(R_\infty)$.
It follows, by rotational invariance about 
 $Z$, that $K'\subset Z$.
 Thus
 \[
   K\subset Z\cup \SS(r_\infty)\cup \SS(R_\infty).
 \]
Therefore, since $K\subset \{z=0\}$ 
(see~\eqref{in-xy-plane}),
\begin{equation}\label{K-form}
   K \subset \{0\}\cup C(r_\infty)
   \cup C(R_\infty).
\end{equation}
where $C(r)$ is the circle of radius $r$
about the origin in the plane $\{z=0\}$.

\begin{claim}\label{odd-claim}
Suppose that $B$ is a ball centered at a point $p$ in the $x$-axis, $X$, and
 that $\Sigma_g\cap B$ converges
smoothly with some multiplicity $m$
to $\{z=0\}\cap B$.
Then $m$ is odd.
\end{claim}

\begin{proof}
Since $X\subset Q_g\subset \Sigma_g$, 
it follows that $p\in \Sigma_g$.
By rotational symmetry by $\pi$ about $X$,
if $L$ is a closed vertical line segment centered
at $p$, then $L\cap \Sigma_g$ has an odd number of points (namely, the point $p$ and equal numbers of points above and below $p$.)
\end{proof}

\begin{claim}\label{orthogonal-claim}
The multiplicity of $\Sigma$ is the same
in each component of 
\begin{equation}\label{components}
   \{z=0\}\setminus
   ( \{0\}\cup C(r_\infty) \cup C(R_\infty)).
   \tag{*}
\end{equation}
Furthermore, if $\Sigma^+$ is nonempty, then
it meets $\{z=0\}$ orthogonally.
\end{claim}

\begin{proof}
Note that $\Sigma$ with its multiplicities
is a stationary varifold with respect
to the shrinker metric.
Let $\Omega$ and $\Omega'$ be two components
of~\eqref{components}
with a common boundary circle $C(r)$,
where $\Omega$ is the inner component.
(Thus $r$ is $r_\infty$ or $R_\infty$.)
Let $m$ be the multiplicity in $\Omega$
and $m'$ be the multiplicity in $\Omega'$.
It suffice to show that $m=m'$.

Case 1: Either $\Sigma^+$ is empty, or 
 $\Sigma^+$ is nonempty and $r$ is not equal 
 to the radius of $\partial \Sigma^+$.
Then $m=m'$ follows immediately from stationarity.

Case 2: $\Sigma^+$ is nonempty and 
 $r$ is the radius
 of $\partial \Sigma^+$.

Consider the point $p=(r,0,0)$.
Let 
\[
   \nu^+ = (\cos\theta,0,\sin\theta)
\]
be the unit vector that is normal to $C(r)$,
tangent to $\Sigma^+$ and that points into $\Sigma^+$.
Thus we can choose $\theta\in (0,\pi)$.
Then
\[
  \nu^- = (\cos\theta,0,-\sin\theta)
\]
is the corresponding vector that points
into $\Sigma^-$.
By stationarity,
\[
  - m\ee_1 + m'\ee_1 + \nu^+ + \nu^- = 0.
\]
Taking the dot product with $\ee_1$ gives
\[
  m - m' = 2\cos\theta.
\]
and thus
\begin{equation}\label{key-inequality}
  |m - m'| = 2|\cos\theta| < 2.
\end{equation}
Since $m$ and $m'$ are both odd, $m-m'$ is even. 
 Therefore, by~\eqref{key-inequality}, $m=m'$
and $\theta=\pi/2$.
\end{proof}

\begin{claim}\label{smooth-claim}
The convergence $\Sigma_g \to \Sigma$ is smooth
at the origin.
\end{claim}

\begin{proof}
Choose $r>0$, 
$r\ne r_\infty$, $r\ne R_\infty$,
sufficiently small that 
$\overline{B(0,r)}$ is disjoint from $\Sigma^+$
and such that
  $\Sigma_g\setminus B(0,r)$
is connected for all $g\in \Gg$.
(Such an $r$ exists by Lemma~\ref{connectivity-lemma}.)

Now $\Sigma_g\cap \overline{B(0,r)}$
converges smoothly and with multiplicity $m$
away from the origin to $\overline{D(r)}$.

If $m=1$, the convergence is smooth
by Allard's Regularity Theorem.
Thus we may assume that $m>1$.

For all sufficiently large $g$,
$\Sigma_g\cap\partial B(0,r)$ consists of $m$ simple closed curves.  Let $\Gamma_g^0$ be the curve that passes through the points 
of $Q_g\cap \partial B(0,r)$.  Let $\Gamma_g$
be one of the other curves. 
By symmetry, 
we may assume that $\Gamma_g$ lies 
in $\{z>0\}$.  
Let $\Sigma_g'$ be the component of 
  $\Sigma_g\setminus \Gamma_g$ that
lies in $\{z>0\}$.
(The component exists by 
  Lemma~\ref{end-lemma}.)
Since $\Sigma_g\setminus B(0,r)$ is connected
and contains points in $\{z\le 0\}$,
we see that $\Sigma_g'$ is contained in 
  $B(0,r)$.

We have shown: each component of 
  $\Sigma_g\cap \partial B(0,r)$
other than $\Gamma_g^0$ bounds a connected component of  
  $\Sigma_g\cap B(0,r)$.
Hence the same is also true for $\Gamma_g^0$.
Now those $m$ components together converge
weakly to the disk $\overline{D(0,r)}$
with multiplicity $m$.
Thus each one converges weakly to
$\overline{D(0,r)}$ with multiplicity~$1$.
By the Allard Regularity Theorem, the convergence is smooth.
\end{proof}

We can now complete the proof of
 Theorem~\ref{limits-theorem}.

{\bf Case 1}: $\Sigma^+$ is empty.
Then (as a varifold) $\Sigma$ is the plane
 $\{z=0\}$ with some odd multiplicity $m=2c+1$,
 by Claims~\ref{odd-claim}
 and~\ref{orthogonal-claim}.

If $\Sigma$ were a multiplicity~$1$ plane,
then $\Sigma_g$ would be a multiplicity~$1$ plane
for all sufficiently large~$g$ 
(by Lemma~\ref{standard-reg-lemma}), contradicting
the assumption that $\Sigma_g$ has positive genus.
Thus $m>1$.

Let 
\[
  R \in (0,\infty)
  \setminus \{r_\infty, R_\infty\}
\]
be such that $\Sigma_g\cap B(0,R)$
is connected for all sufficiently large
 $g\in \Gg$.
 (Such an $R$ exists by Lemma~\ref{connectivity-lemma}.)
Note that $\Sigma_g\cap \partial B(0,R)$
consists of $m$ simple closed curves, $c$
of which lie in $\{z>0\}$ and wind around $Z$.
Also, by the smooth convergence of $\Sigma_g$ to $\Sigma$ 
at the origin (Claim~\ref{smooth-claim}), $\Sigma_g\cap Z$
has $m$ points, $c$ of which lie in $Z^+$.

By Lemma~\ref{topology-lemma}, 
\[
  \genus(\Sigma_g\cap B(0,R)) = 2cg.
\]
Since $c>0$ and $\genus(\Sigma_g)\le 2g$, we see
that $c=1$ and that $\genus(\Sigma_g)=2g$.

It remains (in Case 1) only to show
that $R_\infty=r_\infty$ (so that the convergence
is smooth except along a single circle.)
Suppose to the contrary that $r_\infty< R_\infty$.   
Let $r_\infty< R < R_\infty$.
Thus we may assume that
\[
  r_g < R < R_g
\]
for all $g\in \Gg$.

For all sufficiently large $g$, we have
shown that, if $\Sigma_g\cap B(0,R)$ is connected,
then it has genus~$2g$, which is impossible
since $R< R_g$.  Thus we may assume, for all $g$, that $\Sigma_g\cap B(0,r)$ is not connected.
Now it has at least $2$ components and at most $3$ (since $\Sigma_g\cap\partial B(0,R)$ has
$3$ components.)
The number of components is odd (by symmetry) and thus it is $3$.  But now (exactly as in the proof of Claim~\ref{smooth-claim}) each of those components converges smoothly to the disk $\overline{D(0,R)}$,
so $\Sigma_g\cap B(0,R)$ has genus~$0$, contradicting $r_g<R$.
This completes the proof in Case 1.

{\bf Case 2}: $\Sigma^+$ is nonempty.
By Claim~\ref{top-claim}, its boundary is a circle $C(\rho)$.
By Claim~\ref{orthogonal-claim}, $\Sigma^+$ meets $\{z=0\}$ orthogonally,
so the surface
$$
   \widetilde \Sigma:= \Sigma^+ \cup \Sigma^-\cup C(\rho)$$
is a smoothly embedded, rotationally 
invariant shrinker.  Thus $\widetilde \Sigma$ is a sphere or a cylinder.

Now choose $R>0$ such that 
   $R$ is not equal to $r_\infty$ or $R_\infty$, and large enough
   that $\Sigma_g\cap B(0,R)$ is connected 
   for all $g\in\Gg$.
(Such an $R$ exists  
  by Lemma~\ref{connectivity-lemma}.)

If $\widetilde \Sigma$ is a sphere, then (for large $g$)
 $\Sigma_g\cap \partial B(0,R)$ consists of 
  $m=2c+1$ simple closed curves, $c$ of which lie in $\{z>0\}$.
  Also, $\Sigma_g\cap Z^+$ consists $c+1$
  points, $c$ of which are close to the origin
  and one of which is close to $\Sigma^+\cap Z$.
  Thus 
  \[ 
    \genus(\Sigma_g\cap B(0,R)) = (2c+1)g
\]
If, on the other hand, $\widetilde \Sigma$ is a cylinder,
then (for large $g$) $\Sigma_g\cap\partial B(0,R)$
consists of $m+2$ simple closed curves, 
 $m$ of which lie near the circle $C(R)$
 and two of which are near the 
 circles of intersection of the cylinder 
   $\widetilde \Sigma$ and the sphere $\partial B(0,R)$.
 Note that $c+1$ of those curves lie
 in $\{z>0\}$.
 Furthermore, $\Sigma_g\cap Z$ consist of $m$
 points, $c$ of which lie in $Z^+$.
 
Thus, again,
\begin{equation}\label{almost-done}
  \genus(\Sigma_g\cap B(0,R)) = (2c+1)g.
\end{equation}
We have shown that~\eqref{almost-done} holds whether $\widetilde \Sigma$
is a sphere or a cylinder.
Since $\Sigma_g$ has genus~$\le 2g$, we see that $c=0$.
This completes the proof
of Theorem~\ref{limits-theorem} in Case~2.
\end{proof}

\begin{proposition}\label{3-ends-proposition}
Suppose, in Theorem~\ref{limits-theorem}, that the $\Sigma_g$
converge to the plane $\{z=0\}$ with multiplicity~$3$.
Then all sufficiently large~$g$, $\Sigma_g$ has 
 exactly $3$~ends.
\end{proposition}

\begin{proof}
Let $R>r$, where $C(r)$ is the circle away from which
the convergence $\Sigma_g\to \Sigma$ is smooth.
In the proof of Theorem~\ref{limits-theorem},
 we showed that 
for large $g$, $\Sigma_g\cap B(0,R)$ is connected, has
genus~$2g$, and three boundary components.

Since $\Sigma_g$ is also connected and has genus~$2g$, it
it follows (by elementary topology) that 
  $\Sigma_g\setminus B(0,R)$
has three components, $E_g^1$, $E_g^2$, and $E_g^3$.
Note that (by hypothesis) $E_g^i$ converges smoothly (with multiplicity~one)
to $\{z=0\}\setminus B(0,R)$.
It follows (see Lemma~\ref{ends-lemma}) that (for large $g$), each $E_g^i$
is topologically an annulus.
\end{proof}

The following is well-known, but we
include a proof for the reader's convenience.

\begin{lemma}\label{standard-reg-lemma}
Suppose $\Sigma_i$ are shrinkers that converge weakly 
to a multiplicity $1$ plane $\Sigma$.
Then $\Sigma_i$ is a multiplicity~$1$
 plane for all sufficiently large $i$.
\end{lemma}

\begin{proof}
The flows
\begin{equation}\label{sigma-i-flow}
  t\in (-\infty,0]\mapsto |t|^{1/2}\Sigma_i
\tag{*}
\end{equation}
converge weakly to the flow
\[
  t\in (-\infty,0] \mapsto |t|^{1/2}\Sigma = \Sigma.
\]
The latter flow is smooth and multiplicity $1$,
so the convergence
is smooth on compact subsets of spacetime
 (by the Local Regularity
  Theorem, Theorem~\ref{local-theorem}).
In particular, for all sufficiently large $i$,
the flow~\eqref{sigma-i-flow} is smooth near the spacetime origin.

But, for each $i$,  (trivially) the curvature of~\eqref{sigma-i-flow}
 blows up at the spacetime
origin unless $\Sigma_i$ is a plane.
\end{proof}

\begin{lemma}[Ends Lemma]\label{ends-lemma}
Suppose $0<R<\infty$.
Suppose that $C$ is a cone that is smooth
except at the origin, and suppose that $M$
is a smooth shrinker with boundary in $B(0,R)$ and
that $M$
is smoothly asymptotic (with multiplicity~one)
to $C$.
Suppose for each $i$ that $M_i$ is a smooth
shrinker with boundary in $B(0,R)$.
Suppose that $M_i$ converges smoothly (and
with multiplicity one) to $M$
on compact subsets of $\RR^3\setminus \overline{B(0,R)}$.
Then the convergence $M_i\to M$ is uniform at infinity
in the following sense:
if $\lambda_i\to 0$, then $\lambda_i M_i$
converges smoothly (with multiplicity $1$)
on compact subsets of 
 $\RR^3\setminus\{0\}$ to $C\setminus\{0\}$.
\end{lemma}

\begin{proof}
Consider the following open subset of spacetime:
\[
  \Omega = \{(x,y,z,t): x^2+y^2\ne 0\}.
\]
Consider also the following mean curvature flows (with
moving boundary):
\begin{align*}
&t\in (-\infty,0]\mapsto M_i(t)= |t|^{1/2}M_i, \\
&t\in (-\infty,0]\mapsto M(t)= |t|^{1/2}M.
\end{align*}

Note that the flow $M(\cdot)$ is smooth 
except at the spacetime origin.

By hypothesis, the flows $M_i(\cdot)$ converge
smoothly to the flows $M(\cdot)$ on compact 
(spacetime) subsets of $\RR^3\times (-\infty,0)$.

Since the flow $M(\cdot)$ is smooth away
from the spacetime origin, it follows
(by the Local Regularity 
 Theorem, Theorem~\ref{local-theorem})
that the convergence $M_i(\cdot)$ to $M(\cdot)$
is smooth in a spacetime neighborhood of
each point $(p,0)$ with $p\ne O$.
The result follows immediately, since
\[
  \lambda_i M_i = M_i(-\lambda_i^2).
\]
\end{proof}

\section{Connectivity}

\begin{lemma}\label{connectivity-lemma}
Given $\Lambda<\infty$, there exist
$\eps>0$ and $R<\infty$ 
with the following property.
  If $\Sigma$ is an embedded shrinker in $\RR^3$ 
 with $\entropy(\Sigma)\le \Lambda$, and 
 if $0<r\le \eps$, then 
\[
   \Sigma\setminus B(0,r)
\]
is connected, and if $r>R$, then
\[
   \Sigma\cap B(0,r)
\]
is connected.
\end{lemma}

\begin{proof}
Suppose there is no such $\eps$.  Then there is a sequence 
a sequence $\Sigma_n$ of embedded shrinkers
and a sequence $r_n\to 0$ such that 
\[
 \entropy(\Sigma_n)\le \Lambda
\]
and such that $\Sigma_n\setminus B(0,r_n)$
contains two connected components $C_n$
and $C_n'$. (It may have more components.)
Let $\Sigma_n'$ be a surface in
 $\RR^3\setminus B(0,r_n)$  of least shrinker area
 that separates $C_n$ and $C_n'$.
Thus $\Sigma_n'$ is a stable minimal surface
with respect to the shrinker metric.

Now $\Sigma_n'$ is a surface whose
boundary is in $\partial B(0,r_n)$.
It is smoothly embedded away from the boundary, and its shrinker area is 
at most
\[
\area_s(\Sigma_n)
=
\entropy(\Sigma_n) \le \Lambda.
\]
Since the $\Sigma_n'$ are stable, 
they converge 
 smoothly 
 (after passing to a subsequence)
 in $\RR^3\setminus\{0\}$ to a stable shrinker $\Sigma'$ (possibly with
multiplicity).

By~\cite{gulliver-lawson}, $\Sigma'$ is smoothly embedded
everywhere.
But then $\Sigma'$ is unstable, 
a contradiction.
(Every smooth shrinker without
boundary is an unstable minimal surface with respect to the shrinker metric.)

The proof for $R$ is essentially the same.
\end{proof}

\appendix
\section{Genus of Shrinkers}
\label{ap:genus}

Chodosh, Choi,
and Schulze proved that
  that all shrinkers 
at a spacetime point have the same 
 genus~\cite{cms}*{Proposition~H.8}.
Here, we give a different proof that 
is perhaps somewhat simpler.

\begin{lemma}
\label{morse-estimate}
Suppose that $T\le \Tfat$ and that $p\in M(T)$.
Let $\Ss$ be the collection of all shrinkers at $(p,T)$, i.e., the
 set of all subsequential limits of 
\[
   (T-t)^{1/2}( M(t) - p)
\]
as $t\uparrow T$.

There is an $R>0$ with the following property.
If $\Sigma\in \Ss$, then the function
\[
   q\in \Sigma\setminus \overline{B(0,R)} \mapsto |q|
\]
is a Morse function, and the only critical points are local maxima.

(In other words, at each critical point $q$, both principal curvatures
with respect to the inward pointing normal are $>1/|q|$.)
\end{lemma}

\begin{proof}
Suppose not.  Then there exist $\Sigma_i\in \Ss$, and critical points $q_i\in \Sigma_i$
such that $R_i:=|q_i|\to \infty$ and such 
that one of the principal
  curvatures~$\kappa_i$ at $q_i$
(with respect to the inward pointing normal) satisfies:
\begin{equation}\label{tsiamis}
   \kappa_i \le 1/|q_i|.
\end{equation}

Consider the tangent flow:
\begin{equation}\label{the-tangent-flow-1}
  t\in (-\infty, 0) \mapsto |t|^{1/2}\Sigma_i.
  \tag{*}
\end{equation}
The mean curvature  of $\Sigma_i$ at $q_i$ is $-q_i/2$.
Thus the mean curvature of $R_i\Sigma_i$ at $R_iq_i$ is $-q_i/(2|q_i|)$.

Now translate the flow~\eqref{the-tangent-flow-1}
 in spacetime by $(-R_iq_i, R_i^2)$
to get
\begin{align*}
t\in (-\infty, R_i^2) \mapsto     
&|R_i^2 - t|^{1/2}\Sigma_i - R_iq_i
\\
&= |1 - (t/R_i^2)|^{1/2}R_i\Sigma_i - R_iq_i 
\\
&= |(1 - (t/R_i^2)|^{1/2}R_i(\Sigma_i - q_i) + |(1-(t/R_i^2)|^{1/2}R_iq_i  - R_i q_i
\\
&=
   |(1 - (t/R_i^2)|^{1/2}R_i(\Sigma_i - q_i) - 
   \frac12(t + o(t))\frac{q_i}{|q_i|} 
\end{align*}

By~\cite{bamler-kleiner}, this converges (perhaps after passing to a subsequence)
to a flow that is regular at almost all times, and the convergence is smooth
for almost all $t$.
Thus we see that $R_i(\Sigma_i-q_i)$ converges smoothly to a limit $\Sigma'$ such that
\begin{equation}\label{translator}
  t \in \RR\mapsto \Sigma' - tq
\end{equation}
is  a mean curvature flow, 
i.e., such that $\Sigma'$ is a translator with velocity $-q$.
Here $-q$ is the limit of $q_i/(2|q_i|)$.
Thus the mean curvature at each point of $\Sigma'$
is $-q^\perp$.

Let 
\[
  t\in (-\infty,0)\mapsto |t|^{-1/2}\Sigma''
\]
be a tangent flow at $\infty$ to the 
  flow~\eqref{translator}.  Thus, there are $t_n\to-\infty$ such that
  \[
    \Sigma'_n:= |t_n|^{-1/2}\Sigma'(t_n)
              = |t_n|^{-1/2}(\Sigma' - t_nq)
  \]
  converges to $\Sigma''$.
By~\cite{bamler-kleiner}, the convergence is
 smooth with multiplicity one.
Note that the mean curvature vector $H_n$ at each point
of $\Sigma'_n$ is $(-|t_n|^{1/2}q)^\perp$.  Thus 
\[
  |H_n| = |\nu\cdot H_n| = |t_n|^{1/2}\,|q^\perp|,
\]
or
\[
    |q^\perp| = |t_n|^{-1/2}|H_n|.
\]
Letting $n\to\infty$, we see that at all points of $\Sigma''$,
$|q^\perp|=0$.  That is, $\Sigma''$ is translation-invariant
in the $q$ direction.
Thus $\Sigma''$ is the shrinking cylinder,
and therefore $\Sigma'$ is the bowl soliton~\cite{chh}.

But that is a contradiction, since one of the principal curvatures of $\Sigma'$
at $0$ with respect to the direction $-q$
is $\le 0$, by~\eqref{tsiamis}.
\end{proof}

\begin{corollary}\label{padinsky}
If $\Sigma\in \Ss$, if $r>R$, and if $\partial B(0,r)$
intersects $\Sigma$ transversely, then $\genus(\Sigma \cap B(0,r))=\genus (\Sigma)$.
\end{corollary}

\begin{proof}
By Lemma~\ref{morse-estimate} and standard Morse Theory, each component of $M\setminus B(0,r)$
is diffeomorphic to a closed disk  or to $S^1\times [0,1)$.
(Here $B(0,r)$ is the open ball.)  
\end{proof}

\begin{theorem}\label{shrinker-genus}
All shrinkers at $(p,T)$ have the same genus.
\end{theorem}

\begin{proof}
The set $\Ss$ is connected (with respect to smooth convergence),
so it suffices to show that if $\Sigma_i\in \Ss$ converges
smoothly to $\Sigma\in \Ss$, then
\[
  \genus(\Sigma_i) = \genus(\Sigma)
\]
for all sufficiently large $i$.
Let $R$ be as in Lemma~\ref{morse-estimate}, and let $r>R$ be such that $\partial B(0,r)$
intersects $\Sigma$ transversely.

Thus $\partial B(0,r)$ intersects
$\Sigma_i$ transversely for all sufficiently large $i$, say $i\ge I$.

Also, by smooth convergence,
\[
   \genus(M_i\cap B(0,r)) = \genus(M\cap B(0,r))
\]
for all sufficently large $i$.

The result follows immediately, since,
 by Corollary~\ref{padinsky}, 
 $\genus(M\cap B(0,r))=\genus(M)$
and $\genus(\Sigma_i\cap B(0,r)) = \genus(M_i)$ for all $i\ge I$.
\end{proof}\vskip 1.3 cm

\noindent{\bf Conflict of interest statement:} On behalf of all authors, the corresponding author states that there is no conflict of interest.

\noindent{\bf Data availability statement:} This manuscript has no associated data.

\end{document}